\documentclass[11pt]{article}
\usepackage{amsfonts,latexsym,rawfonts,amsmath,amssymb,amsthm,a4wide, times, color}
\usepackage{graphicx}
\numberwithin{equation}{section}
\usepackage{hyperref}
\usepackage[titletoc,title]{appendix}
\usepackage{hyperref}

\numberwithin{equation}{section}

\newcommand{\pd}[2]{\frac {\partial #1}{\partial #2}}

\newcommand{\al}{\alpha}
\newcommand{\bb}{\beta}
\newcommand{\la}{\lambda}
\newcommand{\La}{\Lambda}
\newcommand{\oo}{\omega}
\newcommand{\Om}{\Omega}

\newcommand{\dd}{\delta}
\newcommand{\Na}{\nabla}
\def\Ga{\Gamma}
\def\ga{\gamma}

\newcommand{\ee}{\epsilon}
\newcommand{\si}{\sigma}
\newcommand{\Si}{\Sigma}
\newcommand{\Te}{\Theta}
\newcommand{\te}{\theta}

\newcommand{\beq}{\begin{equation}}
\newcommand{\eeq}{\end{equation}}
\newcommand{\beqs}{\begin{eqnarray*}}
\newcommand{\eeqs}{\end{eqnarray*}}
\newcommand{\beqn}{\begin{eqnarray}}
\newcommand{\eeqn}{\end{eqnarray}}
\newcommand{\beqa}{\begin{array}}
\newcommand{\eeqa}{\end{array}}

\def\td{\tilde}
\def\p{\partial}
\def\m{{\frak m}}
\def\b{\backslash}

\def\RR{{\mathbb R}}
\def\NN{{\mathbb N}}

\def\ri{\rightarrow}

\def\un{\underline}

\def\si{\sigma}

\def\tr{{\rm tr}}

\def\vol{{\rm vol}}
\def\r{{\mathbf r}}

\def\cA{{\mathcal A}}
\def\cC{{\mathcal C}}
\def\cB{{\mathcal B}}
\def\cD{\mathcal D}

\def\cJ{{\mathcal J}}

\def\cM{{\mathcal M}}

\def\cS{{\mathcal S}}
\def\cT{{\mathcal T}}

\def\ka{{\kappa}}

\def\x{{\mathbf{x}}}

\def\n{{\mathbf{n}}}

\def\Area{{\mathrm {Area}}}
\def\Supp{{\mathrm {Supp}}}

\def\loc{{\mathrm {loc}}}
\renewcommand\div{{\rm div}}
\newtheorem{prop}{Proposition}[section]
\newtheorem{theo}[prop]{Theorem}
\newtheorem{lem}[prop]{Lemma}

\newtheorem{cor}[prop]{Corollary}

\newtheorem{defi}[prop]{Definition}
\newtheorem{conj}[prop]{Conjecture}

\title{On Ilmanen's multiplicity-one conjecture for  mean curvature flow with type-$I$ mean curvature }
\author{Haozhao Li \footnote{Supported by NSFC grant  No. 12071449.} \, \footnote{Supported by  the Fundamental Research Funds for the Central Universities.} \quad  and \quad  Bing Wang \footnote{Supported by  NSFC grant No. 11971452 and No. 12026251.}\, \footnotemark[2]}

\begin{document}
\bibliographystyle{plain}


\maketitle

\begin{abstract}
In this paper, we show that if the mean curvature of a closed smooth embedded mean curvature flow in $\RR^3$ is of type-$I$, then the rescaled flow   at the first finite singular time converges smoothly to a self-shrinker flow with multiplicity one. This result confirms Ilmanen's multiplicity-one conjecture  under the assumption that the mean curvature is of type-$I$. As a corollary, we show that the mean curvature at the first singular time of a closed smooth embedded mean curvature flow in $\RR^3$ is at least of type-$I$.
\end{abstract}

\tableofcontents

\section{Introduction}
In this paper, we study  finite time singularities of closed smooth embedded mean curvature flow in $\RR^3. $
A one-parameter family of hypersurfaces $\x(p, t): \Si^n\ri \RR^{n+1}$ is called a mean curvature flow, if $\x$  satisfies the equation
\beq
\pd {\x}t=-H\n,\quad \x( 0)=\x_0, \label{eq:MCF}
\eeq where $H$ denotes the mean curvature of the hypersurface $\Si_t:=\x(t)(\Si)$ and $\n$ denotes the outward unit normal   of $\Si_t$.
In the previous paper \cite{[LW]}, we proved that the mean curvature of (\ref{eq:MCF}) must blow up at the first finite singular time for a closed smooth embedded mean curvature flow in $\RR^3. $ This paper can be viewed as a continuation of \cite{[LW]}, and we will develop the techniques in \cite{[LW]}  further to study the finite time singularities of mean curvature flow.

\subsection{Singularities of mean curvature flow}

The mean curvature flow with convexity conditions has been well studied during the past several decades. In \cite{[Hui1]}, Huisken proved that if the initial hypersurface is uniformly convex, then after rescaling the mean curvature flow exists for all time and converges smoothly to a round sphere. When the initial hypersurface is mean-convex or two-convex, there are a number of  estimates for the mean curvature flow (cf. Huisken-Sinestrari  \cite{[HS99a]}\cite{[HS99b]}, Haslhofer-Kleiner \cite{[HK13]}), and these estimates are important to study  the surgery of mean curvature flow(cf. Huisken-Sinestrari\cite{[HS09]},  Brendle-Huisken \cite{[BH]}, Haslhofer-Kleiner \cite{[HK14]}). Moreover, for mean curvature flow with mean convex initial hypersurfaces,   B. White  gave some structural
 properties of the singularities in  \cite{[White00]} \cite{[White03]}, and  B. Andrews also showed  a noncollapsing estimate in \cite{[Andrews]}.

However, all these results rely on  convexity conditions of  initial hypersurfaces, and it is very difficult to study  general cases. For the curve shortening flow in the plane, following the work Gage \cite{[Gage1]}\cite{[Gage2]} and Gage-Hamilton \cite{[GH]} on convex curves Grayson \cite{[Grayson]} proved that any embedded closed curve in the plane evolves to a convex curve and subsequently shrinks to a point, and Andrews-Bryan \cite{[AB]} gave a direct proof of  Grayson's theorem without using the monotonicity formula or classification of singularities.
In the higher dimensions, we know very little results without convexity conditions. Colding-Minicozzi studied the generic singularities of the mean curvature flow in \cite{[CM2]}\cite{[CM3]}. For the classification of self-shrinkers without convexity conditions,  S. Brendle \cite{[Brendle]} proved that the round sphere is the only compact embedded self-shrinkers in $\RR^3$ with genus $0$, and L. Wang \cite{[WangLu]} showed that each end of  a noncompact self-shrinker in $\RR^3$ of finite
topology is smoothly asymptotic to either a regular cone or a self-shrinking
round cylinder. However, it still remains wide open to understand the behavior of mean curvature flow at the singular time in the general cases.

\subsection{The multiplicity-one conjecture and the main theorems}

To study the singularities of mean curvature flow without convexity conditions,
 Ilmanen proposed a series of conjectures in \cite{[Il]}\cite{[I2]}.  Suppose that the mean curvature flow (\ref{eq:MCF}) reaches a singularity at $(x_0, T)$ with $T<+\infty.$ For any sequence $\{c_j\}$ with $c_j\ri +\infty, $ we rescale the flow (\ref{eq:MCF}) by
\beq
\Si^j_t:=c_j\Big(\Si_{T+c_j^{-2}t}-x_0\Big),\quad t\in [-Tc_j^2, 0). \label{eq:MCF3}
\eeq
By Huisken's monotonicity formula \cite{[Hui2]} and Brakke's compactness theorem \cite{[Bra]}, a subsequence of $\Si^j_t$ converges weakly to a limit flow $\cT_t$, which is called a tangent flow at $(x_0, T)$. In \cite{[Il]} Ilmanen showed that   the tangent flow at the first singular time must be smooth for a smooth embedded mean curvature flow in $\RR^3$, and he conjectured

\begin{conj}(Ilmanen \cite{[Il]}\cite{[I2]}, the multiplicity-one conjecture)
For a smooth one-parameter family of closed embedded surfaces in $\RR^3$ flowing by mean curvature, every tangent flow at the first singular time has multiplicity one.
\label{cje:PK16_1}
\end{conj}

Moreover,  Ilmanen pointed out that the multiplicity-one conjecture implies a conjecture on the asymptotic structure of self-shrinkers in $\RR^3$, and the latter conjecture has been confirmed recently by L. Wang \cite{[WangLu]}. If the initial hypersurface is mean convex or satisfies the Andrews condition, then the multiplicity-one conjecture holds (cf. White \cite{[White00]}, Haslhofer-Kleiner \cite{[HK13]},  Andrews \cite{[Andrews]}). Recently, A. Sun \cite{[Sun]} proved that  the generic singularity of mean curvature flow
of closed embedded surfaces in $\RR^3$ modelled by closed self-shrinkers with multiplicity has multiplicity one.
In  general the multiplicity-one conjecture is still wide open.  It is well-known to experts that this conjecture holds if the second fundamental form $A$ is of type-$I$.
The main contribution of this paper is to confirm the multiplicity-one conjecture under the assumption that the mean curvature is of type-$I$, which is a much weaker condition.

To state our result, we first introduce some notations. A hypersurface $\x: \Si^n \ri \RR^{n+1}$ is called a self-shrinker, if $\x$ satisfies the equation
$$H=\frac 12\langle \x, \n\rangle. $$
If $\Si$ is a self-shrinker, then we call $\Si_t:=\sqrt{-t}\,\Si \,(t<0)$  a self-shrinker flow. The main theorem of this paper is the following result.

\begin{theo}\label{theo:main1}Let
 $\x(t): \Si^2\ri \RR^3 (t\in [0, T))$  a  closed smooth embedded  mean curvature flow with the first  singular time $T<+\infty$.  If the mean curvature satisfies
\beq
\max_{\Si_t}|H|(p, t)\leq \frac {\La}{\sqrt{T-t}},\quad \forall\;t\in [0, T),\label{eq:H}
\eeq for some $\La>0$,  then for any $a, b\in \RR$ with $-\infty <a<b<0$ and any sequence $c_j\ri +\infty$ there exists a subsequence, still denoted by $\{c_j\}$, such that  the flow $\{\Si_t^j,  a<t<b\}$ defined by (\ref{eq:MCF3}) converges smoothly to a  self-shrinker flow with multiplicity one as $j\ri +\infty$.
\end{theo}

It is not hard to see that Theorem~\ref{theo:main1} is equivalent to the following result.

\begin{theo}
\label{theo:removable_intro}
Let $\{(\Si^2, \x(t)), 0\leq t<+\infty\}$ be a closed smooth embedded rescaled mean curvature flow
\beq \Big(\pd {\x}t\Big)^{\perp}=-\Big(H-\frac 12\langle \x, \n\rangle\Big)\n\label{eq:RMCF0}\eeq
 satisfying
\beq d(\Si_t, 0)\leq D,\quad \hbox{and}\quad \max_{\Si_t}|H(p, t)|\leq \La \label{eqn:PK07_2}\eeq
for two constants $D, \La>0$.  Then for any $t_i\ri +\infty$ there exists a subsequence of $\{\Si_{t_i+t}, -1<t<1\}$ such that it  converges in smooth topology to a complete smooth self-shrinker with multiplicity one as $i\ri +\infty$.
\end{theo}

In \cite{[LW]}, we showed Theorem~\ref{theo:removable_intro} under the assumption that the mean curvature decays exponentially to zero.
In this special case,  the flow (\ref{eq:RMCF0}) converges smoothly to a  plane passing through the origin with multiplicity one. Theorem \ref{theo:removable_intro} means that under the assumption that the mean curvature is bounded for all time the flow (\ref{eq:RMCF0}) also converges smoothly to a self-shrinker with multiplicity one.  In fact, Theorem~\ref{theo:removable_intro} is not stated with the optimal condition.
Checking the proof carefully, one can see that the conclusion of Theorem \ref{theo:removable_intro} still holds under the assumption that the mean curvature is uniformly bounded on any ball for all time:
\beq
\max_{B_R(0)\cap \Si_t}|H|(p, t)\leq C_R, \label{eq:A000}
\eeq where $C_R$ is a constant depending on $R$. Note that if the flow (\ref{eq:RMCF0}) converges smoothly to a self-shrinker with multiplicity one, the condition (\ref{eq:A000}) automatically holds by the self-shrinker equation.
Thus, the condition (\ref{eq:A000}) is also necessary for the smooth convergence of the flow (\ref{eq:RMCF0}).
Therefore, the solution of  the multiplicity-one conjecture, i.e.,  Conjecture~\ref{cje:PK16_1},  is equivalent to the examination of (\ref{eq:A000}),  which will be an interesting subject of study in the near future.

The multiplicity-one conjecture is closely related to the extension problem of mean curvature flow.  Huisken  \cite{[Hui1]} proved that if the flow (\ref{eq:MCF}) develops a singularity at time $T<\infty$, then the second fundamental form will blow up at time $T$. A natural question is whether the mean curvature will blow up at the finite singular time of a mean curvature flow. Toward this question, A. Cooper \cite{[Cooper]} proved that $|A||H|$ must blow up at the singular time of the flow. In \cite{[LS1]}  Le-Sesum  affirmirtively answered this question under the assumption that  the multiplicity-one conjecture holds, or the condition that the second fundamental form is of type-$I$ at the singular time
\beq
\max_{\Si_t}|A|\leq \frac {C}{\sqrt{T-t}},\quad \forall\; t\in [0, T). \label{eq:TypeI}
\eeq
Furthermore, Le-Sesum \cite{[LS3]} proved that the mean curvature is at least of type-$I$ for a mean curvature flow satisfying (\ref{eq:TypeI}).
Using Theorem \ref{theo:main1},  we can remove the type-$I$ condition (\ref{eq:TypeI}) of Le-Sesum's result as follows, which can also be viewed as an improvement of the extension theorem in \cite{[LW]}.

\begin{cor}\label{cor:main} If $\;\x(t): \Si^2\ri \RR^3 (t\in [0, T))$ is  a  closed smooth embedded  mean curvature flow with the first singular time $T<+\infty$, then there is a constant $\dd>0$ such that
$$\limsup_{t\ri T}\sqrt{T-t}\max_{\Si_t}|H|\geq \dd,\quad \forall\; t\in [0, T). $$
\end{cor}

\subsection{Outline of the proof}

Now we sketch the proof of Theorem \ref{theo:removable_intro}. Assume that the mean curvature satisfies the type-$I$ condition (\ref{eq:H}) along the flow (\ref{eq:MCF}) and the first singular time $T<+\infty$.
Then the mean curvature is uniform bounded along the rescaled flow (\ref{eq:RMCF0}).
We have to show that  the  flow (\ref{eq:RMCF0}) converges smoothly to a self-shrinker with multiplicity one.
The strategy is similar to~\cite{[LW]}, we first show a weak-compactness  theorem and obtain the flow convergence is smooth away from a singular set.
Then we use stability argument to remove the singular set.  However,  the technique here is much more involved.
The proof consists of three steps:\\

\emph{ Step 1. Convergence of the rescaled mean curvature flow with multiplicities.} In this step, since the mean curvature is uniformly bounded along the flow,   we have the short-time pseudolocality theorem and the energy concentration property, and we can follow the arguments in \cite{[LW]} to develop the weak compactness theory of mean curvature flow. However, compared with \cite{[LW]}, since the mean curvature doesn't decay to zero, we have the following difficulties:
\begin{itemize}
  \item No long time pseudolocality theorem;
  \item The space-time singularities in the limit  don't move along straight lines.
\end{itemize}
Because of lacking  these results, we face a number of new technical difficulties to show the $L$-stability of the limit self-shrinker.
These difficulties force us to use analysis tools to study the asymptotical behavior of the solution of the limit parabolic equation near the singular set.
\\

\emph{ Step 2. Show that the multiplicity of the convergence is one for one subsequence.}
 As in \cite{[LW]}, it suffices to show that the limit self-shrinker is $L$-stable. By the convergence of the flow away from the singular set, if every limit has multiplicity greater than one,
we can renormalize the ``height-difference" function to obtain a positive solution of the equation
\beq
\pd wt=\Delta w-\frac 12\langle x, \Na w\rangle+|A|^2w+\frac 12 w, \label{eq:A000a}
\eeq away from the singular set. To show the $L$-stability of the limit self-shrinker, we have to show the following two estimates:
\begin{itemize}
\item For each time, the asymptotical behavior of $w$ is ``good" near the singular set.
  \item Uniform $L^1$ norm of $w$ independent of time.

\end{itemize}
By its construction, $w$ is defined on any compact set away from the singular set and we have no estimates near the singular set by the geometric method. However, we found that $w$ satisfies many good properties from the PDE point of view. In \cite{[KT]}, Kan-Takahashi  studied similar problem for some semilinear parabolic equations along time-dependent singularities in the Euclidean spaces. Kan-Takahashi  showed their result for one time-dependent singularity, and the solution of the equation looks like $\log \frac 1r$ in dimension $2$, where $r$ is the distance from any point $x$ to the singularity.  However, in our case the solution of (\ref{eq:A000a}) may have multiple singularities, and these singularities may coincide at one point. Thus, we cannot apply Kan-Takahashi's result directly, and we need to develop their techniques  to show  that the solution $w$ is in $L^1$ across the singularities and  near the singular set the solution $w$    roughly looks like
$$w(x, t)\sim \sum_{k=1}^l\; c_k( t) \log \frac 1{\r_k(x, t)},$$
where  $\r_k(x, t)$ denotes the intrinsic distance from a point $x$ to a singularity $\xi_k(t)$ at time $t$. Here  the constant $c_i$ may depend on $t$. In general, the $L^1$ norm of $w$ may tend to infinity as $t\ri +\infty$. In order to show uniform $L^1$ norm of $w$, we refine the argument in \cite{[LW]} and also use the estimate of $w$ near the singularities to choose a  sequence of time slices $\{t_i\}$, and then we  show that for such a special sequence the corresponding function $w$ has uniform $L^1$ bound independent of $t$.  Thus, for the special sequence ${t_i}$,  the auxiliary function $w$ satisfies the two desired estimates. Then we can follow the argument in~\cite{[LW]} to show that the convergence of (\ref{eq:RMCF0}) is smooth and of multiplicity one,
         for the special sequence $\{t_i\}$. \\

 \emph{ Step 3. Show the multiplicity-one convergence for each subsequence.}
  This step is a new difficulty beyond~\cite{[LW]}.  In~\cite{[LW]}, each limit, no matter what multiplicity it is, must be a flat plane passing through the origin.  Therefore, up to rotation, different limits can be regarded as the same.
  By the monotonicity of the entropy, it is clear that if one limit is a multiplicity-one plane, then each limit must also be a multiplicity-one plane.  However,  in the current setting, each limit is only a self-shrinker and the limits may vary as the
  time sequences change.  A priori, it is possible that one sequence converge to a multiplicity-one self-shrinker $A$, and the other sequence converge to a multiplicity two self-shrinker $B \neq A$.  This possibility cannot be ruled out by only
  using the monotonicity of the entropy.   To overcome this difficulty,  we essentially use the smooth compactness theorem of self-shrinkers by Colding-Minicozzi \cite{[CM1]}.
  Since the limit self-shrinkers form a compact set, we know that the local behavior of limit self-shrinkers are very close to that of planes on a fixed small scale.
  From this and the volume continuity, we derive an argument to show that the multiplicity is independent of the choice of subsequences.   Therefore, every subsequence converges with multiplicity one.
  \\

It is interesting to know whether the above argument still works for the multiplicity-one conjecture without the mean curvature bound assumption (\ref{eq:H}).
The main difficulty is the loss of pseudolocality result as in \cite{[LW]}, since the points in the evolving surfaces may move drastically if the mean curvature is large.
Furthermore, the loss of mean curvature bound also induces difficulties in applying PDE tools to analyze the singular set.
However, as we discussed around (\ref{eq:A000}), it is also logically possible to develop the estimate (\ref{eq:A000}) directly.

\subsection{Relation with other geometric flows}
It is interesting to compare the mean curvature flow with the Ricci flow.
The extension problem for Ricci flow has been extensively studied recently.   Corollary~\ref{cor:main} has a cousin theorem in the Ricci flow.
In Theorem 1 of~\cite{[WangBing11]}, it was shown that along the  Ricci flow $\{(M, g(t)), 0 \leq t<T\}$ with the singular time $T<+\infty$, we have
\begin{align}
   \max_{M} |Ric|_{g(t)} \geq \frac{\delta}{\sqrt{T-t}}, \quad  t \in [0,T),  \label{eqn:PK07_3}
\end{align}
which extends the famous Ricci extension theorem of N. Sesum \cite{[Sesum1]}.
Up to rescaling, the gap inequality (\ref{eqn:PK07_3}) is equivalent to
$
  \max_{M} |Ric|_{g(t)}  \geq \delta
$
along the rescaled Ricci flow solution
\begin{align}
   \partial_t g= -Ric + g, \quad \; t \in [0, \infty).   \label{eqn:PK07_4}
\end{align}
Actually, we even believe that a gap for scalar curvature holds for a rescaled Ricci flow solution. In other words, along the rescaled Ricci flow (\ref{eqn:PK07_4}) we should have
\begin{align*}
  \max_{M} |R|_{g(t)}  \geq \delta.
\end{align*}
It is easy to see that the scalar extension conjecture of the Ricci flow will hold automatically if one can prove the above inequality along the rescaled Ricci flow (\ref{eqn:PK07_4}),
just like the extension theorem of mean curvature in~\cite{[LW]} follows directly from Corollary \ref{cor:main}.

The similarity between the regularity theory of rescaled mean curvature flow (\ref{eq:RMCF0}) and the rescaled Ricci flow (\ref{eqn:PK07_4}) was noticed for a while.
For example, such similarity was discussed in the introduction of~\cite{[LW]}.
Along the rescaled flows, the mean curvature bound condition (\ref{eqn:PK07_2})  is comparable to the scalar curvature bound condition $|R| \leq C$.
Note that the Fano K\"ahler-Ricci flow provides many examples of the global solutions of the rescaled Ricci flow (\ref{eqn:PK07_4}) and Perelman showed that $|R| \leq C$ holds automatically.
The boundedness of the scalar curvature is crucial to study the convergence of K\"ahler-Ricci flow to a limit flow(cf.  Theorem 1.5 of \cite{[CW2]}, with journal version~\cite{[CW2A]} and~\cite{[CW2B]}).
For time-slice convergence, see Tian-Zhang~\cite{[TZ]}, Bamler~\cite{[Bam]} and Chen-Wang~\cite{[CW3]} for example.
Since (\ref{eqn:PK07_2}) is the comparable condition of  Perelman's estimates, we can view Theorem \ref{theo:main1} as the analogue of the convergence results in the Fano K\"aher-Ricci flow.
However, we have to confess that we do not know any non-trivial examples  satisfying the condition (\ref{eqn:PK07_2}).  
By non-triviality we mean that the flow (\ref{eq:RMCF0}) is neither a self-shrinker nor convex.  It will be very interesting to find out such examples.

The rescaled mean curvature flow can also be compared with the Calabi flow. In \cite{[Calabi]} E. Calabi studied the gradient flow of the $L^2$-norm of the scalar curvature among K\"ahler metrics in a fixed cohomology class
 on a compact K\"ahler manifold, which is now well-known as the Calabi flow.  X. X. Chen conjectured that the Calabi flow always exists globally for any initial smooth K\"ahler potential. Very recently, Chen-Cheng \cite{[CC]}
 proved that the Calabi flow   exists as long as the scalar curvature is uniformly bounded. Therefore, to study the long time existence of Calabi flow, it is crucial to control the scalar curvature, which is  similar to the mean curvature
 condition (\ref{eqn:PK07_2})  for the rescaled mean curvature flow.  Assuming the long time existence and the uniform boundedness of the scalar curvature, the current authors and K. Zheng showed the convergence of the
 Calabi flow in \cite{[LWZ]}, just as Theorem \ref{theo:removable_intro} for rescaled mean curvature flow.

\subsection{List of notations}
In the following, we list the important notations in this paper.

\begin{itemize}
 \item $d(x, y):$ the Euclidean distance from $x$ to $y$. Defined in Definition \ref{defi:refined}.
  \item $B_r(p):$ the open ball in $\RR^3$ centered at $p$ with radius $r$. Defined in Definition \ref{defi:notations}.
  \item $d_g(x, y):$ the intrinsic distance of $(\Si, g)$ from $x$ to $y$. First appears in the beginning of Section 4.
\item $\cB_r(p):$ the intrinsic geodesic ball in $(\Si, g)$ centered at $p$ with radius $r$. Defined in Definition \ref{defi:notations}.

\item $C_x(B_r(p)\cap \Si):$ the connected component of $ B_r(p)\cap \Si$ containing $x\in \Si. $ Defined in Definition \ref{defi:notations}.

\item $\frak m(x, t):$ the multiplicity at $(x, t)$. Defined in (\ref{eqn:PE18_2}).

\item $\cS:$ the space-time singular set. Defined in Proposition \ref{prop:weakcompactness}.

\item $\cS_t=\{x\in \RR^3\;|\; (x, t)\in \cS\}$: the singular set at time $t$.   Defined in Proposition \ref{prop:weakcompactness}.

\item $\xi(t):$ a Lipschitz singular curve in $\cS$. First appears in Lemma \ref{prop:weakcompactness2}.

\item $\rho: \RR^+\ri \RR^+$: an increasing positive function. First appears in Definition \ref{defi:soliton}.

\item $|\Omega|:$ the volume of a set $\Omega\subset \RR^3$ with respect to the standard metric on $\RR^3$. Defined in Lemma \ref{lem:TN}.

\item $\Om_{\ee, R}(t)$: a subset of the limit self-shrinker away from singularities. Defined in (\ref{eqn:Omega2}).

\item $\cS_I$: the union of the singular set on a time interval $I$. Defined in (\ref{eqn:defi}).

\item $u_i$: the height difference function defined in (\ref{eq:u}).

\item $w_i$: the normalized difference function defined in (\ref{eq:F001}).

\item $d_H:$ the Hausdorff distance in the Euclidean space.

\item $\r(x, t)$: the intrinsic distance function from $x$ to the singular curve $\xi(t)$. Defined in (\ref{eq:rk}).

\item $\r_k(x, t)$: the intrinsic distance function from $x$ to the singular curve $\xi_k(t)$. Defined in (\ref{eq:rk2}) and Theorem \ref{theoA:main}.

\item $F_t^{(k)}(\dd)$ and $A_t^{(k)}(\dd, \rho)$: a subset around the singularities on the limit self-shrinker. Defined in (\ref{eq:G008}) and (\ref{eq:G009}).

\item $\cM_{k, m}(\rho, \Xi) $: a subset of a Riemannian manifold defined in Definition \ref{defi:003}.

\item $\Ga_{\un t, \bar t}:$ the union of space-time singular curves. Defined in (\ref{eqE:001}) and (\ref{eq:E006}).

\item $Q_{r, \un t, \bar t} $ and $ \hat Q_{r, \un t, \bar t}:$ the neighborhood of the singular curves. Defined in and (\ref{eqE:002}) and (\ref{eq:E006}).

\item $\phi_{\xi}$:  cutoff functions around the singular curves.  Defined in Definition \ref{defi:001} and Definition \ref{defi:002}.

\item $I_{\xi}:$ a functional associated with a singular curve $\xi.$ Defined in Definition \ref{defi:002}.

\item $\td \r(x, t)$ and $\td v(x, t)$: defined in  Definition \ref{defi:002}.

\item $H(z):$ a cutoff function defined in Definition \ref{defi:001}. Note that the function $H(z)$ is only used in Section 4. Since the mean curvature doesn't appear in Section 4, we keep the same notation $H(z)$ as in \cite{[KT]}.
\end{itemize}

\subsection{The organization}

The organization of this paper is as follows. In Section 2 we recall some  facts on the pseudolocality theorem and energy concentration property. Moreover, we will show the weak compactness of mean curvature flow under some geometric conditions and we show the multiplicity of the convergence is a constant. In Section 3 we show the rescaled mean curvature flow with bounded mean curvature  converges smoothly to a self-shrinker with multiplicity one, under the assumption that the auxiliary function satisfies good growth properties at the singular set. In Section 4 we will show the estimates of  the auxiliary function by developing  Kan-Takahashi's argument.  Finally, we finish the proof of Theorem \ref{theo:main1} in Section 5. In the appendices, we include two versions of the parabolic Harnack inequality and give the full details on the calculation of the linearized equation of rescaled mean curvature flow.  \\

\subsection{Acknowledgement}

Bing Wang would like to thank  Lu Wang for many helpful conversations.
Both authors are grateful to the anonymous referees for many useful suggestions to improve the exposition of
this paper.

\section{Weak compactness of refined sequences}

\subsection{The pseudolocality theorem and energy concentration property}
In this subsection, we recall  some results in \cite{[LW]}.
First, we have  the following definition.

\begin{defi}\label{defi:notations}
\begin{enumerate}
\item[(1).]We denote by $B_r(p)$ the  ball in $\RR^{n+1}$ centered at $p$ with radius $r$ with respect to the standard Euclidean metric, and $\cB_r(p)\subset (M, g)$ the intrinsic geodesic ball on $M$ centered at $p$ with radius $r$ with respect to the metric $g$.

\item[(2).]
For any $r>0, p\in \RR^{n+1}$ and $\Si^n\subset \RR^{n+1}$, we denote by {$C_x(B_r(p)\cap \Si)$} the connected component of $ B_r(p)\cap \Si$ containing $x\in \Si. $

\end{enumerate}
\end{defi}

We first recall the following result of
Chen-Yin \cite{[ChenYin]}.

\begin{lem}\label{lem:graph1}(cf. Lemma 7.1 of \cite{[ChenYin]}) Let $\Si^n\subset  \RR^{n+1}$ be properly embedded in $B_{r_0}(x_0)$ for some $x_0\in \Si$ with
$$|A|(x)\leq \frac 1{r_0}\,\quad x\in B_{r_0}(x_0)\cap \Si. $$
Let $\{x^1\, \cdots, x^{n+1}\}$ be the standard coordinates in $\RR^{n+1}$.
Assume that $x_0=0$ and the tangent plane of $\Si$ at $x_0$ is $x^{n+1}=0.$ Then there is a map
$$u: \Big\{x'=(x^1, \cdots, x^n)\,\Big|\, |x'|<\frac {r_0}{96}\Big\}\ri \RR$$
with $u(0)=0$ and $|\Na u|(0)=0$ such that the connected component containing $x_0$ of $\Si\cap
\{(x', x^{n+1})\in \RR^{n+1}\;|\;|x'|<\frac {r_0}{96}\}$  can be written as a graph $\{(x', u(x'))\,|\,|x'|<\frac {r_0}{96}\}$ and
$$|\Na u|(x')\leq \frac {36}{r_0}|x'|.$$
\end{lem}

Using Lemma \ref{lem:graph1}, we show that the local area ratio of the surface is very close to $1$.

\begin{lem}\label{lem:ratio}(cf. Lemma 3.3 of \cite{[LW]}) Suppose that $\Si^n\subset B_{r_0}(p)\subset\RR^{n+1}$ is a hypersurface with $\p \Si\subset \p B_{r_0}(p)$ and
$$ \sup_{\Si}\,|A|\leq \frac 1{r_0}.$$
For any $\dd>0$, there is a constant $\rho_0=\rho_0(r_0, \dd)$ such that for any $r\in (0, \rho_0)$ and any $x\in B_{\frac {r_0}{2}}(p)\cap \Si$ we have
 \beq
1-\dd\leq \frac {\vol_{\Si}(C_x(B_r(x)\cap \Si))}{\oo_n r^n}\leq 1+\dd.\label{lem:ratio:007}
\eeq

\end{lem}
\begin{proof}By Lemma \ref{lem:graph1}, for any $x\in B_{\frac {r_0}{2}}(p)\cap \Si$ the component $C_x(B_{\rho_0}(x)\cap \Si)$ with $\rho_0=\frac {r_0}{192}$ can be written as a graph of a function $u$ over the tangent plane at $x$, which we assume to be $P=\{(x_1, \cdots, x_n, x_{n+1})\in \RR^{n+1}\;|\; x_{n+1}=0\}$,   with $|\Na u|(x')\leq \frac {72}{r_0}|x'|$ where $x'=(x_1, \cdots, x_n).$   Let $r\in (0, \rho_0)$. We denote by $\Om_r$ the projection of $C_x(B_r(x)\cap \Si)$ to the plane $P$. Then for any $x'\in \p\Om_r$ we have
\beq
u(x')^2+|x'|^2= r^2. \label{lem:ratio:001}
\eeq
On the other hand, for any $x'\in   \Om_{\rho_0}$ we have the inequality \beq
|u(x')|\leq |u(0)|+\max_{t\in [0, 1]}|\Na u|(tx')\cdot |x'|\leq \frac {72}{r_0}|x'|^2. \label{lem:ratio:002}
\eeq
Note that (\ref{lem:ratio:001}) and (\ref{lem:ratio:002}) imply that for any $x'\in \p\Om_r$,
\beq |x'|^2\leq r^2=u(x')^2+|x'|^2\leq |x'|^2\Big(1+\frac {5184}{r_0^2}\rho_0^2\Big). \label{lem:ratio:003}
\eeq Thus, we have
\beq
 \td r:=\frac r{\sqrt{1+\frac {5184}{r_0^2}\rho_0^2}}\leq|x'|\leq r,\quad \forall\;  x'\in \p\Om_r,
\eeq which implies that
\beq
B_{\td  r}(x)\cap P\subset \Om_r\subset B_r(x)\cap P. \label{lem:ratio:004}
\eeq
Thus, the volume ratio of $C_x(B_r(x)\cap \Si)$ is bounded from above
\beqn \frac {\vol_{\Si}(C_x(B_r(x)\cap \Si))}{\oo_n r^n}&=&\frac 1{\oo_n r^n}\int_{\Om_r}\,\sqrt{1+|\Na u|^2}\,d\mu \nonumber\\&\leq &\frac 1{\oo_n r^n}\int_{B_r(x)\cap P}\,\sqrt{1+|\Na u|^2}\,d\mu\nonumber\\
&\leq& \sqrt{1+\frac {5184}{r_0^2}r^2},  \label{lem:ratio:005}\eeqn
where we used (\ref{lem:ratio:002}) and (\ref{lem:ratio:004}). Moreover, the volume ratio of $C_x(B_r(x)\cap \Si)$ is bounded from below
 \beqn \frac {\vol_{\Si}(C_x(B_r(x)\cap \Si))}{\oo_n r^n}&\geq &\frac 1{\oo_n r^n}\int_{B_{\td r}(x)\cap P}\,\sqrt{1+|\Na u|^2}\,d\mu\nonumber\\
&\geq& \frac {\td r^n}{r^n}\geq \Big(1+\frac {5184}{r_0^2}\rho_0^2\Big)^{-\frac n2}. \label{lem:ratio:006}\eeqn
Combining (\ref{lem:ratio:005}) with (\ref{lem:ratio:006}), for any $\dd>0$  we can choose $\rho_0=\rho_0(n, \dd, r_0)$ further small  such that (\ref{lem:ratio:007}) holds. The lemma is proved.

\end{proof}

Next we recall the two-sided pseudolocality theorem in \cite{[LW]}. If the initial hypersurface can be locally written as a graph of a single-valued function, then we have the pseudolocality type results of the  mean curvature flow   by Ecker-Huisken  \cite{[EH1]} \cite{[EH2]},   M. T. Wang \cite{[WangMT]}, Chen-Yin \cite{[ChenYin]} and Brendle-Huisken \cite{[BH]}. However, in our case we have to apply the pseudolocality theorem for the hypersurfaces which may converge with multiplicities. Thus, we use the boundedness of the mean curvature to get the two-sided pseudolocality theorem in \cite{[LW]}.

\begin{theo}[\textbf{Two-sided pseudolocality}]\label{theo:pseudoA}(cf. \cite{[LW]})
For any $r_0\in (0, 1], \La, T>0$,  there exist $\eta=\eta(n, \La), \ee=\ee(n, \La)>0$ satisfying
\beq
\lim_{\La\ri 0}\eta(n, \La)=\eta_0(n)>0,\quad  \lim_{\La\ri 0}\ee(n, \La)=\ee_0(n)>0\label{eq:B006}\\
\eeq
and the following properties.  Let $\{(\Si^n, \x(t)), -T\leq t\leq T\} $ be a closed smooth embedded mean curvature flow (\ref{eq:MCF}).  Assume  that
\begin{enumerate}

  \item[(1)]  the second fundamental form satisfies
$|A|(x, 0)\leq \frac 1{r_0}$ for any $ x\in C_{p_0}( B_{r_0}(p_0)\cap \Si_0) $ where $p_0=\x_0(p)$ for some $p\in \Si$;
  \item[(2)] the mean curvature of $\{(\Si^n, \x_t), -T\leq t\leq T\}$ is  bounded by $\La$.
\end{enumerate}
Then for any $(x, t)$ satisfying
\beq
x\in C_{p_t}(\Si_t\cap B_{\frac 1{16} r_0}(p_0)), \quad t\in \Big[-\frac {\eta r_0^2}{2(\La+\La^2)}, \frac {\eta r_0^2}{2(\La+\La^2)}\Big]\cap [-T, T] \nonumber
\eeq where $p_t=\x_t(p),$  we have the estimate $$|A|(x, t)\leq \frac {1}{\ee r_0}. $$

\end{theo}

Using the pseudolocality theorem, we have the energy concentration property.

\begin{lem}[\textbf{Energy concentration}]\label{lem:energy concen1}(cf. \cite{[LW]})  For any $\La, K, T>0$, there exists a constant $\ee(n, \La, K, T)>0$ with the following property.
Let $\{(\Si^n, \x(t)), -T\leq t\leq T\} $ be a closed smooth embedded mean curvature flow (\ref{eq:MCF}).  Assume that $\max_{ \Si_t\times [-T, T]}|H|(p, t)\leq \La.$
Then we have
\beq
\int_{\Si_0\cap B_{Q^{-1}}(q)}\,|A|^n\,d\mu_0\geq \ee(n, \La, K, T)
\eeq whenever $q\in \Si_0$ with $Q:=|A|(q, 0)\geq K.$

\end{lem}

A direct corollary of Lemma \ref{lem:energy concen1} is the following $\ee$-regularity of the mean curvature flow, which can be viewed as a generalization of the result of Choi-Schoen \cite{[CS]}.

\begin{cor}[\textbf{$\ee$-regularity}]\label{cor:energy}  (cf. \cite{[LW]})  There exists $\ee_0(n)>0$ satisfying the following property.
Let $\{(\Si^n, \x(t)), -1\leq t\leq 1\} $ be a closed smooth embedded mean curvature flow (\ref{eq:MCF}).
Suppose that the mean curvature satisfies $\max_{ \Si_t\times [-1, 1]}|H|(p, t)\leq 1.$  For any $q\in \Si_0$, if
  $$\int_{\Si_0\cap B_r(q)}\,|A|^n\,d\mu_0\leq \ee_0(n)$$ for some $r>0$, then we have
\beq \max_{B_{\frac r2}(q)\cap \Si_0}|A|\leq \max\{1, \frac 2r\}. \label{eq:AA004}\eeq
\end{cor}

\subsection{ Weak compactness}
As in \cite{[LW]}, we use the pseudolocality theorem and the energy concentration property to develop the weak compactness of the mean curvature flow. Here we will replace the zero mean curvature condition
in \cite{[LW]} by the boundedness of the mean curvature in the definition of refined sequences.  By ``refined", we mean that the sequence is taken after a point-selecting process such that many good properties already hold for the objects in this sequence.
The name of refined sequence originates from~\cite{[CW1]}.

\begin{defi}[\textbf{Refined sequences}]\label{defi:refined}Let  $\{(\Si_i^2, \x_i(t)), -1< t< 1 \}$ be a one-parameter family of closed smooth embedded surfaces satisfying the mean curvature flow equation (\ref{eq:MCF}). It is called a refined sequence if the following properties are satisfied for every $i:$
\begin{enumerate}
  \item[(1)] There exists a constant $D>0$ such that $d(\Si_{i, t}, 0)\leq D$ for any $t\in (-1, 1)$, where $d(\Si, 0)$ denotes the Euclidean distance from the point $0\in \RR^3$ to the surface $\Si\subset \RR^3$ and $\Si_{i, t}=\x_i(t)(\Si_i)$;

  \item[(2)] There is a uniform constant $\La>0$ such that
  \beq \max_{\Si_{i, t}\times (-1, 1)}|H|(p, t)\leq \La,\label{eq:A203}\eeq
  \item[(3)] There exists an increasing positive function $\rho: \RR^+\ri \RR^+$ such that for any $R>0,$
  \beq \int_{\Si_{i, t}\cap B_R(0)}\,|A|^2 \,d\mu_{i, t}\leq \rho(R), \quad \forall \,t\in (-1, 1);\label{eqn:201} \eeq
  \item[(4)] There is uniform $N>0$ such that for all $r>0$ and $p\in \RR^3$ we have
  \beq \Area_{g_i(t)}(B_r(p)\cap \Si_{i, t})\leq N \pi r^2, \quad \forall \,t\in (-1, 1). \label{eqn:202}\eeq

  \item[(5)] There exist  uniform constants $\bar r, \ka>0$ such that for any $r\in (0, \bar r]$ and any $p\in \Si_{i, t}$ we have
\beq \Area_{g_i(t)}(B_r(p)\cap \Si_{i, t})\geq \kappa r^2, \quad \forall \,t\in (-1, 1).\label{eqn:203}\eeq

\item[(6)] There exists $T>1$ such that
\beq \lim_{i\ri +\infty}\int_{-1}^1\,dt\int_{\Si_{i, t}}\, e^{-\frac {|\x_i|^2}{4(T-t)}}\Big|H_i-\frac 1{2(T-t)}\langle \x_i, \n \rangle\Big|^2\,d\mu_{i, t}=0. \label{eqn:204}\eeq

\end{enumerate}

\end{defi}

Following the arguments as in minimal surfaces (cf. White \cite{[White3]}, or \,Colding-Minicozzi \cite{[CMbook2]}), we have the weak compactness  for mean curvature flow.
\begin{prop}\label{prop:weakcompactness} Let $\{(\Si_i^2, \x_i(t)), -1< t< 1 \}$ be a refined sequence. Then there exists a subsequence, still denoted by $\{(\Si_i^2, \x_i(t)), -1<t< 1 \}$, a smooth self-shrinker flow $\{(\Si_{\infty}, \x_{\infty}(t)), -1<t<1\}$ satisfying
\beq
H=\frac 1{2(T-t)}\langle \x_{\infty}, \n\rangle, \label{eq:ssflow}
\eeq for some $T>1$,
 and a space-time singular set $\cS=\{(x, t)\,|\,t\in (-1, 1), x\in  \RR^3\}$ satisfying the following properties:
\begin{enumerate}
  \item[(1).] The sequence $\{(\Si_i^2, \x_i(t)), -1<t< 1 \}$ converges locally smoothly, possibly with multiplicity at most $N_0$, to  $\{(\Si_{\infty}, \x_{\infty}(t)), -1<t<1\}$ away from
$\cS$;

  \item[(2).] For each time $t\in (-1, 1)$ the singular set $\cS_t=\{x\in \RR^3\;|\;(x, t)\in \cS\}$ is locally finite in the sense that   $\sharp\{\cS_t\cap B_R(0)\}$ is uniformly bounded by a number depending only on $\rho(R)$.

\item[(3).] The sequence in (1) also converges in extrinsic Hausdorff distance.

\end{enumerate}

\end{prop}

\begin{proof}
 We first show that after taking a subsequence if necessary,  $\Si_{i, 0}$ converges locally smoothly to $\Si_{\infty, 0}$ away from a locally finite set  $\cS_0. $
  Fix large  $R>0$ and let $\Om=B_R(0)\subset \RR^3$. By Property (1) in Definition \ref{defi:refined}, we have $\Si_{i, t}\cap \Om\neq \emptyset$ for large $R>0$ and any $t\in (-1, 1)$.
 For any $U\subset \Omega$, we define the measures $\nu_i$ by
$$\nu_i(U)=\int_{U\cap \Si_{i, 0}}\,|A_i|^2\,d\mu_{i, 0}\leq \rho(R),$$ where we used (\ref{eqn:201}) in the   inequality.
The general compactness of Radon measures implies that there is a subsequence, which we still denote by $\nu_i$, converges weakly to a Radon measure $\nu$ with
$\nu(\Om)\leq \rho(R).$
We define the set
$$\cS_0=\{x\in \Omega\;|\;\nu(x)\geq \ee_0\},$$
where $\ee_0$ is the constant in Corollary \ref{cor:energy}.
 It follows that $\cS_0$ contains at most $\frac {\rho(R)}{\ee_0}$ points. Given any $y\in \Omega\backslash \cS_0.$ There exists some $s>0$ such that $B_{10s}(y)\subset \Om$ and
$\nu(B_{10s}(y))<\ee_0.$ Since $\nu_i\ri \nu$, for $i$ sufficiently large we have
$$\int_{B_{10s}(y)\cap \Si_{i, 0}}\,|A_i|^2\,d\mu_{i, 0}<\ee_0. $$
Corollary \ref{cor:energy} implies that for $i$ sufficiently large we have the estimate
\beq \max_{B_{5s}(y)\cap \Si_{i, 0}}|A|(x, 0)\leq \max\{1, \frac 1{5s}\}.  \label{eq:A100}\eeq
By Theorem \ref{theo:pseudoA} and (\ref{eq:A203}) , there exists $\ee=\ee(s, n)>0$ such that
\beq
\max_{B_{\ee r_0}(y)\cap \Si_{i, t}}|A|(x, t)\leq \frac 1{\ee r_0},\quad \forall\, t\in [-\ee r_0^2, \ee r_0^2] \label{eq:A101}
\eeq where $r_0=5s. $ Therefore, for large $i$ we have all higher order estimates of the second fundamental form at
any point in $\Si_{i, 0}\b B_{2r_0}(\cS_0)$, where $B_r(\cS_0)=\{x\in \RR^3\,|\,d(x, \cS_0)\leq r\}$.
Using a diagonal sequence argument and taking $s\ri 0$ we can show that a subsequence of $\Si_{i, 0}$ converges in smooth topology, possibly with multiplicities, to a limit surface $\Si_{\infty, 0}$  away from the singular set $\cS_0.$ The properties (\ref{eqn:202})-(\ref{eqn:203})  imply that the multiplicity of the convergence is bounded by a constant $N_0$.

Note that by (\ref{eq:A101})  the second fundamental form is uniformly bounded for any point
$(x, t)\in (\Si_{i, t}\b B_{2r_0}(\cS_0))\times ([-\ee r_0^2, \ee r_0^2]\cap (-1, 1)).$ By compactness of mean curvature flow (cf. Theorem 2.6 of \cite{[LW]}),  the flow $\{\Si_{i, t}\b B_{2r_0}(\cS_0), t\in (-\ee r_0^2, \ee r_0^2)\cap (-1, 1) \}$ converges smoothly to a limit flow $\{\Si_{\infty, t}\b B_{2r_0}(\cS_0), t\in (-\ee r_0^2, \ee r_0^2)\cap (-1, 1) \}$ and by Property (6) in definition \ref{defi:refined} $\Si_{\infty, t}\b B_{2r_0}(\cS_0)$ satisfies the self-shrinker equation (\ref{eq:ssflow}) for $t\in (-\ee r_0^2, \ee r_0^2)\cap (-1, 1) \}$. We can also replace $t=0$ by any other $t_0\in (-1, 1)$ and the above argument still works for the time interval $(-\ee r_0^2+t_0, \ee r_0^2+t_0)\cap (-1, 1).$ Since $r_0=5s>0$ is arbitrary small, by using a diagonal sequence argument and taking $ s\ri 0$ we have that  $\{(\Si_i^2, \x_i(t)), -1<t< 1 \}$ converges locally smoothly to the flow  $\{(\Si_{\infty}, \x_{\infty}(t)), -1<t<1\}$ away from
$\cS$ and  $\Si_{\infty, t}$ satisfies the equation (\ref{eq:ssflow}). Note that $\Si_{\infty}$ is a self-shrinker in $\RR^3$ and it can be viewed as a minimal surface in $(\RR^3, g_{ij})$ with $g_{ij}=e^{-\frac {|x|^2}4}\dd_{ij}$.
Thus,  we can follow the argument in minimal surfaces (cf.   White \cite{[White3]}, or \,Colding-Minicozzi \cite{[CMbook2]}) to show that $\Si_{\infty, t}\cup \cS_t$ is smooth and embedded and $\Si_{i, t}$ converges to $\Si_{\infty, t}$ in Hausdorff distance.
  The proposition is proved.

\end{proof}

As in \cite{[LW]}, we show that the multiplicity in Proposition \ref{prop:weakcompactness} is constant.
To study the multiplicity, we define a function
\beq
\Te(x, r, t) :=\lim_{i\ri +\infty}\frac {\Area_{g_i(t)}(\Si_{i, t}\cap B_r(x))}{\pi r^2},\quad
\forall \;(x, t)\in \Si_{\infty, t}\times (-1, 1).\label{eqD:002}
\eeq
Then the multiplicity at $(x, t)\in \Si_{\infty, t}\times (-1, 1)$  is defined by
\beq
\frak m(x, t) :=\lim_{r\ri 0}\Te(x, r, t).  \label{eqn:PE18_2}
\eeq
It is clear that $\frak m(x, t)$ is an integer. In the following result, we show that $\frak m(x, t)$ is independent of $x$ and $t$. Note that in Lemma 3.14 of  \cite{[LW]} we proved the same result under the assumption that the mean curvature decays exponentially to zero. The first two steps of the proof here are similar to that of \cite{[LW]} while the third step is different. We give all the details  for completeness.

\begin{lem}\label{lem:multiplicity}
Under the  assumption of Proposition~\ref{prop:weakcompactness}, the function $\frak m(x, t)$ is a constant integer on $\Si_{\infty, t}\times (-1, 1)$. Namely, $\frak m(x, t)$ is independent of $x$ and $t$.
\label{lma:PE16_1}
\end{lem}

\begin{proof}The proof can be divided into three steps.\\

{\it Step 1. For each $t\in (-1, 1)$, $\frak m(x, t)$ is  constant on $\Si_{\infty, t}\b \cS_t$.} Fix $t_0\in (-1, 1), R>0$ and  $x_0\in (\Si_{\infty, t_0}\cap B_R(0))\b \cS_{t_0}$.
There exists $r_0>0$ such that for large $i$,
\beq
|A|(x, t_0)\leq \frac 1{r_0}, \quad \, \forall\; x\in B_{r_0}(x_0)\cap \Si_{\infty, t_0}.\label{eqE:001xa}
\eeq
By Lemma \ref{lem:graph1}, we can assume $r_0$ small such that $B_{r_0}(x_0)\cap \Si_{\infty, t_0}$ can be written as a graph over the tangent plane of $\Si_{\infty,  t_0}$ at $x_0$. Since $x_0$ is regular, we have $d(x_0, \cS_{t_0})>0$.
Let $r_1=\frac 14\min\{r_0, d(x_0, \cS_{t_0})\}$. For any $p\in B_{r_1}(x_0)\cap \Si_{i, t_0}$, we have $B_{r_1}(p)\subset B_{r_0}(x_0)$. Thus, (\ref{eqE:001xa}) implies that for large $i,$
\beq
|A|(x, t_0)\leq \frac 1{r_1}, \quad \, \forall\; x\in B_{r_1}(p)\cap \Si_{i, t_0}.\label{eqE:002}
\eeq By Lemma \ref{lem:ratio}, for any $\dd>0$ there exists $\rho_0=\rho_0(r_1, \dd)\in (0, \frac {r_1}{200})$ such that for any $r\in (0, \rho_0)$ and any $p\in B_{\frac {r_1}2}(x_0)\cap \Si_{i, t_0}$  we have
\beq
 1-\dd\leq \frac {\Area_{g_i(t_0)}(C_p(B_r(p)\cap \Si_{i, t_0}))}{\pi r^2}\leq 1+\dd. \label{eqE:003}
\eeq
Suppose that $B_{r_1}(x_0)\cap \Si_{i, t_0}$ has $m_i$ connected components, where $m_i$ is an integer bounded by a constant independent of $i$ by Proposition \ref{prop:weakcompactness}. After taking a subsequence of $\{\Si_{i, t_0}\}$ if necessary, we can assume that $m_i$ are the same
integer denoted by $m$ with $m\geq 1.$
For any $x\in B_{\frac {r_1}2}(x_0)\cap \Si_{\infty, t_0}$, we denote by $\al_x$ the normal line passing through $x$ of $\Si_{\infty, t_0}$. Since each component of $B_{r_1}(x_0)\cap \Si_{i, t_0}$ converges to $B_{r_1}(x_0)\cap \Si_{\infty, t_0}$ smoothly and $B_{r_1}(x_0)\cap \Si_{\infty, t_0}$ is a graph over the tangent plane of $\Si_{\infty, t_0}$ at $x_0$, $\al_x$ intersects transversally each component of $\Sigma_{i, t_0}$  at exactly one point.
Suppose that
$$\al_x\cap \Big(B_{r_1}(x_0)\cap \Si_{i, t_0}\Big)=\{p_i^{(1)}, p_i^{(2)}, \cdots, p_i^{(m)}\}. $$
Then (\ref{eqE:003})  implies that for any integer $j$ with $1\leq j\leq m$ and any $r\in (0, \rho_0)$,
\beq
1-\delta \leq \frac {\Area_{g_i(t_0)}(C_{p_i^{(j)}}(B_r(p_i^{(j)})\cap \Si_{i, t_0}))}{\pi r^2}\leq 1+\dd. \label{eqE:004q}
\eeq
After shrinking $r_0$ if necessary, we can assume that $B_r(x)\cap \Si_{\infty, t_0}$
has only one component for any $r\in (0, \frac {r_1}2)$ and any $x\in B_{\frac {r_1}2}(x_0)\cap \Si_{\infty, t_0}$.
Since for any $1\leq j\leq m$ and $r\in (0, \rho_0)$ we have $p_i^{(j)}\ri x$ and
$C_{p_i^{(j)}}(B_r(p_i^{(j)})\cap \Si_{i, t_0})$ converges smoothly to $B_r(x)\cap \Si_{\infty, t_0}$ as $i\ri +\infty$,  (\ref{eqE:004q}) implies that
 \beq
 m(1-\delta) \leq \lim_{i\ri +\infty}\frac {\Area_{g_i(t_0)}(B_r(x)\cap \Si_{i, t_0})}{\pi r^2}\leq m(1+\dd).
 \eeq
In other words, for any $x\in B_{\frac {r_1}2}(x_0)\cap \Si_{\infty, t_0}$ and any $r\in (0, \rho_0)$ we have
\beq m(1-\delta) \leq \Te(x, r, t_0)\leq m(1+\dd). \label{eqE:005}\eeq
Taking $r\ri 0$ in (\ref{eqE:005}), we have
$$\m(x, t_0)=m,\quad \forall\; x\in B_{\frac {r_1}2}(x_0)\cap \Si_{\infty, t_0}.$$
By the connectedness of $\Si_{\infty, t_0}\b \cS_{t_0}$, we know that $\m(x, t_0)$ is constant on  $\Si_{\infty, t_0}\b \cS_{t_0}$.\\

{\it Step 2. For each $t\in (-1, 1)$, $\frak m(x, t)$ is  constant on $\Si_{\infty,  t}$.}  Fix $t_0\in (-1, 1)$. It suffices to consider a singular point $p_0\in \cS_{t_0}$. Suppose that $B_{r}(p_0)\cap \Si_{\infty, t_0}$ has no other singular points except $p_0$ for any $r\in (0, r_0).$ Then all points in $(B_r(p_0)\b B_{\ee}(p_0))\cap \Si_{\infty, t_0}$ are regular and $(B_r(p_0)\b B_{\ee}(p_0))\cap \Si_{i, t_0}$ has $m$ connected components.  Thus, we have
\beqn
\Area_{g_i(t_0)}(\Si_{i, t_0}\cap B_r(p_0))&\leq&\Area_{g_i(t_0)}(\Si_{i, t_0}\cap (B_r(p_0)\b B_{\ee}(p_0)))+\Area_{g_i(t_0)}(\Si_{i, t_0}\cap B_{\ee}(p_0))\nonumber\\
&\leq &\Area_{g_i(t_0)}(\Si_{i, t_0}\cap (B_r(p_0)\b B_{\ee}(p_0)))+N \ee^2,\label{eqF:001xx}
\eeqn where we used (\ref{eqn:202}) in the last inequality. Since each  component of $\Si_{i, t_0}\cap (B_r(p_0)\b B_{\ee}(p_0)) $ converges to $(B_r(p_0)\b B_{\ee}(p_0))\cap \Si_{\infty}$ smoothly, we have
\beq
\lim_{i\ri +\infty}\Area_{g_i(t_0)}(\Si_{i, t_0}\cap (B_r(p_0)\b B_{\ee}(p_0)))=m\,\Area_{g_\infty(t_0)}(\Si_{\infty, t_0}\cap (B_r(p_0)\b B_{\ee}(p_0))).\label{eqF:002xx}
\eeq  Note that  $m$ is also the multiplicity at each regular point in $\Si_{\infty, t_0}$ by Step 1.
Combining (\ref{eqF:001xx}) with (\ref{eqF:002xx}), we have
\beqn &&m\,\Area_{g_\infty(t_0)}(\Si_{\infty, t_0}\cap (B_r(p_0)\b B_{\ee}(p_0)))\nonumber\\&\leq&
\lim_{i\ri +\infty}\Area_{g_i(t_0)}(\Si_{i, t_0}\cap B_r(p_0))\nonumber\\&\leq& m\,\Area_{g_\infty(t_0)}(\Si_{\infty, t_0}\cap (B_r(p_0)\b B_{\ee}(p_0)))+N \ee^2.\label{eqF:003}
\eeqn
Taking $\ee\ri 0$ in (\ref{eqF:003}), we have
\beq
\lim_{i\ri +\infty}\Area_{g_i(t_0)}(\Si_{i, t_0}\cap B_r(p_0))=m\,\Area_{g_\infty(t_0)}(\Si_{\infty, t_0}\cap B_r(p_0)). \label{eqF:004}
\eeq Thus, we have
\beqs
\m(p_0, t_0)&=&\lim_{r\ri 0}\frac {\Area_{g_i(t_0)}(\Si_{i, t_0}\cap B_r(p_0))}{\pi r^2}\\
&=&m \lim_{r\ri 0}\frac {\Area_{g_\infty(t_0)}(\Si_{\infty, t_0}\cap B_r(p_0))}{\pi r^2}=m.
\eeqs
This implies that the multiplicity of each singular point is the same as that of any regular point. \\

{\it Step 3. $\frak m(x, t)$ is constant in $t$.} Fix any $t_0\in (-1, 1), R>0$ and $x_0\in (\Si_{\infty, t_0}\cap B_R(0))\b\cS_{t_0}$. There exists $r_0>0$ such that for large $i$,
\beq
|A|(x, t_0)\leq \frac 1{r_0}, \quad \forall \,x\in B_{r_0}(x_0)\cap \Si_{\infty, t_0},\label{eq:lem218:001}
\eeq and for any $r\in (0, r_0)$ the surface $B_{r}(x_0)\cap \Si_{\infty, t_0}$ has only one component.
Let $m_0=\frak m(x_0, t_0)$ and $r_1=\frac 14\min\{r_0, d(x_0, \cS_{t_0})\}>0$. For large $i$,   $B_{r_1}(x_0)\cap \Si_{i, t_0}$ has $m_0$ connected components, which we denote by $\Om_{i, 1}, \cdots, \Om_{i, m_0}$. Since for each integer $k\in [1, m_0]$ the component $\Om_{i, k}$ converges smoothly to $\Si_{\infty, t_0}\cap B_{r_1}(x_0)$ as $i\ri +\infty$. Similar to Step 1, we can find  $x_{i, k}\in \Om_{i, k}$ such that $\lim_{i\ri +\infty}d(x_{i, k}, x_0)=0 $. By the choice of $r_1$, we have
\beq
B_{r_1}(x_{i, k})\subset B_{r_0}(x_0). \label{Z:016}
\eeq
Thus, (\ref{eq:lem218:001}) implies that for any integer $k\in [1, m_0]$ and large $i$,
\beq
|A|(x, t_0)\leq \frac 1{r_1}, \quad \forall \,x\in C_{x_{i, k}}(B_{r_1}(x_{i, k})\cap \Si_{i, t_0}).\label{Z:017}
\eeq
By Lemma \ref{lem:ratio}, for any $\dd>0$ there exists $\rho_0=\rho_0(r_1, \dd)\in (0, \frac {r_1}2)$ such that for any $r\in (0, \rho_0)$ we have
\beq
1-\dd\leq \frac {\Area_{g_{i, t_0}}(C_{x_{i, k}}(B_{r}(x_{i, k})\cap \Si_{i, t_0}))}{\pi r^2}\leq 1+\dd. \label{Y:001}
\eeq
Note that by (\ref{Z:016}) and the definition of $\Om_{i, k}$,   for any large $i $ we have
\beq
C_{x_{i, k}}(B_{2\rho_0}(x_{i, k})\cap \Si_{i, t_0})\neq C_{x_{i, k'}}(B_{2\rho_0}(x_{i, k'})\cap \Si_{i, t_0}),\quad \forall \; k\neq k'. \label{Y:004}
\eeq
Using (\ref{Z:017}) and the assumption that $\max_{\Si_{i, t}}|H|\leq \La$ by (\ref{eq:A203}), Theorem \ref{theo:pseudoA} implies that  there exists $\eta(\La)$ and $ \ee(\La)>0$ such that
\beq
|A|(x, t)\leq \frac 1{\ee\, r_1},\quad \forall\; x\in C_{x_{i, k, t}}(B_{\frac 1{16}r_1}(x_{i, k, t})\cap \Si_{i, t}),\quad t\in [t_0-\eta r_1^2, t_0+\eta r_1^2],
\eeq
 where $x_{i, k, t}=\x_t(\x_{t_0}^{-1}(x_{i, k}))$. Similar to (\ref{Y:001}), there exists $\rho_1=\rho_1(r_1, \dd)\in (0, \frac {\rho_0}{10})$ such that for any $r\in (0, \rho_1)$ we have
\beq
1-\dd\leq \frac {\Area_{g_{i, t}}(C_{x_{i, k, t}}(B_{r}(x_{i, k, t})\cap \Si_{i, t}))}{\pi r^2}\leq 1+\dd, \quad t\in [t_0-\eta r_1^2, t_0+\eta r_1^2]. \label{Y:002}
\eeq\\

We show that we can choose $\rho_1$ and $\tau=\tau(r_0, \dd, \La)\in (0, \eta r_1^2]$ small such that for any $k\neq k'$
\beq
C_{x_{i, k, t}}(B_{\rho_1}(x_{i, k, t})\cap \Si_{i, t})\neq C_{x_{i, k', t}}(B_{\rho_1}(x_{i, k', t})\cap \Si_{i, t}), \quad \forall\; t\in [t_0-\tau, t_0+\tau].\label{Y:003}
\eeq  Suppose not, we can find $\tau_0\in (0, \eta r_1^2]$,  a continuous curve $\ga_{\tau_0}(s)(s\in [0, 1])$ connecting $x_{i, k, t_0+\tau_0}$ and $x_{i, k', t_0+\tau_0}$ with
\beq
\ga_{\tau_0}\subset B_{\rho_1}(x_{i, k, t_0+\tau_0})\cap \Si_{i, t_0+\tau_0},\quad \ga_{\tau_0}\subset B_{\rho_1}(x_{i, k', t_0+\tau_0})\cap \Si_{i, t_0+\tau_0}. \label{Y:006}
\eeq
Let $\ga_{\tau}=\x_{t_0+\tau}(\x_{t_0+\tau_0}^{-1}(\ga_{\tau_0}))$. Then $\ga_{0}(s)(s\in [0, 1])$ is a curve connecting $x_{i, k}$ and $x_{i, k'}$. Since the mean curvature satisfies $\max_{\Si_{i, t}}|H|\leq \La$, we have
\beq
|\x(p, t)-\x(q, t)|\leq |\x(p, s)-\x(q, s)|+2\La|t-s|. \label{Y:005}
\eeq
For small $\tau_0$,  (\ref{Y:005}) with (\ref{Y:006}) implies that
\beq
\ga_{0}\subset B_{\rho_0}(x_{i, k})\cap \Si_{i, t_0},\quad \ga_{\tau_0}\subset B_{\rho_0}(x_{i, k})\cap \Si_{i, t_0},
\eeq which contradicts (\ref{Y:004}). Therefore, (\ref{Y:003}) holds.\\

  Since $x_{i, k}\in B_{r_1}(x_0)$ and the mean curvature is uniformly bounded, $x_{i, k, t}$ lies in a bounded domain for any $t\in [t_0-\tau, t_0+\tau]$. Thus,  for each integer $k\in [1, m_0]$ and any $t\in [t_0-\tau, t_0+\tau]$ a subsequence of $C_{x_{i, k, t}}(B_{\rho_1}(x_{i, k, t})\cap \Si_{i, t})$ converges to $C_{x_t}(B_{\rho_1}(x_{t})\cap \Si_{\infty, t})$ smoothly, where $x_t\in \Si_{\infty, t}$ is a limit point of $\{x_{i, k, t}\}_{i=1}^{\infty}$. Then (\ref{Y:002}) and (\ref{Y:003}) imply that for any $r\in (0, \rho_1)$ and $t\in [t_0-\tau, t_0+\tau]$ we have
\beq
\lim_{i\ri+\infty}\frac {\Area_{g_{i}(t)}(B_{r}(x_{i, t})\cap \Si_{i, t})}{\pi r^2}\geq\lim_{i\ri +\infty}\frac {\Area_{g_{i}(t)}(C_{x_{i, t}}(B_{r}(x_{i, t})\cap \Si_{i, t}))}{\pi r^2} \geq m_0(1-\dd).
\eeq
Thus,  we have
\beq
\m(x_t, t)\geq m_0=\m(x_0, t_0),\quad \forall\; t\in [t_0-\tau, t_0+\tau].  \label{eq:lem218:002}
\eeq
By  {Step 2}, (\ref{eq:lem218:002}) implies that
for any $x\in \Si_{\infty, t}$ and $y\in \Si_{\infty, t_0}$ we have
\beq
\m(x, t)\geq\m(y, t_0),\quad \forall\; t\in [t_0-\tau, t_0+\tau].
\eeq
Thus, the multiplicity $\m(x, t)$ is a constant independent of $x$ and $t$. The lemma is proved.

\end{proof}

To characterize the singular and regular points in $\Si_{\infty, t}$, we have the following result.

\begin{lem}\label{lem:singular}The same  assumption as in Proposition~\ref{prop:weakcompactness}. Fix any $t_0\in (-1, 1)$ and any $\dd, R>0$.
\begin{enumerate}
  \item[(1).]If $x_0\in (\Si_{\infty, t_0}\cap B_R(0))\b \cS_{t_0}$ and $x_i\in \Si_{i, t_0}$ with $x_i\ri x_0$, then there exists a positive number  $r'=r'(\dd, \Si_{\infty, t_0}, R, x_0, \cS_{t_0})$ such that for any $r\in (0, r')$ we have
\beq
1-\dd\leq \lim_{i\ri +\infty}\frac {\Area_{g_i(t_0)}(C_{x_i}(B_{r}(x_i)\cap \Si_{i, t_0}))}{\pi r^2}\leq 1+\dd. \label{lem:singular:001}
\eeq
  \item[(2).] If $x_0\in \cS_{t_0}\cap B_R(0)$,  then there exist $r'=r'(\dd, \Si_{\infty, t_0}, R, \La,  x_0, \cS_{t_0})>0$ and a sequence $x_i\in \Si_{i, t_0}$ with $x_i\ri x_0$ such that for any $r\in (0, r')$  we have
\beq
\lim_{i\ri +\infty}\frac {\Area_{g_i(t_0)}(C_{x_i}(B_{r}(x_i)\cap \Si_{i, t_0}))}{\pi r^2}\geq 2(1-\dd).\label{lem:singular:002}
\eeq

\end{enumerate}

\end{lem}

\begin{proof}(1).   Since $\Si_{\infty, t_0}$ is a smooth self-shrinker, there exists $r_0=r_0(\Si_{\infty, t_0}, R)>0$ such that for large $i$ we have
\beq
|A|(x, t_0)\leq \frac 1{r_0}, \quad \, \forall\; x\in B_{r_0}(x_0)\cap \Si_{\infty, t_0}.\label{eqE:001x}
\eeq
Since $\cS_{t_0}$ is locally finite and $x_0\in (\Si_{\infty, t_0}\cap B_R(0))\b \cS_{t_0}$, the distance from $x_0$ to $\cS_{t_0}$ satisfies $d(x_0, \cS_{t_0})>0.$
Let $r_1=\frac 12\min\{r_0, d(x_0, \cS_{t_0})\}$. Then for large $i$, we have
\beq
|A|(x, t_0)\leq \frac 1{r_1}, \quad \, \forall\; x\in B_{r_1}(x_i)\cap \Si_{i, t_0}.\label{eqE:001y}
\eeq
By Lemma \ref{lem:ratio}, for any $\dd>0$ there exists $r'= \rho(\dd, r_1)>0$ such that    for any $r\in (0, r')$  the area ratio of $C_{x_i}(B_r(x_i)\cap \Si_{i, t_0})$ is given by
\beq 1-\dd\leq \frac {\Area_{g_{i, t_0}}(C_{x_i}(B_r(x_i)\cap \Si_{i, t_0}))}{\pi r^2}\leq 1+\dd. \label{lem:singular:003} \eeq Thus,   (\ref{lem:singular:001}) holds.

(2). Let $x_0\in \cS_{t_0}\cap B_R(0)$ and $r_0=r_0(\Si_{\infty, t_0}, R, x_0, \cS_{t_0})>0$ such that $\Si_{\infty, t_0}\cap B_{2r_0}(x_0)$ has only one component and no other singular points except $x_0. $
Let $ Q_i:=\max_{\overline{B_{r_0}(x_0)}\cap \Si_{i, t_0}}|A|\ri +\infty. $ Then $Q_i$ is achieved by some point $x_i\in \overline{B_{r_0}(x_0)}\cap \Si_{i, t_0}$ with $x_i\ri x_0.$
As in Step 2 of the proof of Lemma \ref{lem:multiplicity}, for any $r\in (0, r_0)$ we have
\beq
\lim_{i\ri +\infty}\frac {\Area_{g_i(t_0)}(C_{x_i}(B_r(x_i)\cap \Si_{i, t_0}))}{\pi r^2}=m \frac {\Area_{g_{\infty}(t_0)}(B_r(x_0)\cap \Si_{\infty, t_0})}{\pi r^2},\label{lem:singular:005}
\eeq where $m$ is a positive integer.
Note that Lemma \ref{lem:ratio} implies that for any $\dd>0$ there exists $r_0'=r_0'(\dd, \Si_{\infty, t_0}, R, x_0, \cS_{t_0})\in (0, r_0)$ such that
\beq
\frac {\Area_{g_{\infty}(t_0)}(B_r(x_0)\cap \Si_{\infty, t_0})}{\pi r^2}\leq 1+\dd,\quad \forall\; r\in (0, r_0']. \label{lem:singular:006}
\eeq
Assume that $m=1$. Then (\ref{lem:singular:006}) and (\ref{lem:singular:005}) imply that for  large $i$, \beq
\frac {\Area_{g_i(t_0)}(C_{x_i}(B_{r'}(x_i)\cap \Si_{i, t_0}))}{\pi {r'}^2}\leq 1+2\dd,\quad \forall\; r\in (0, r_0']. \label{lem:singular:011}
\eeq We choose $r'=r'(\dd, \Si_{\infty, t_0}, R, \La, x_0, \cS_{t_0})\in (0, r_0')$ small such that $(1+2\dd)e^{\La r'}\leq 1+3\dd.$ Since the mean curvature satisfies $\max_{\Si_{i, t}\times (-1, 1)}|H|\leq \La$, by Lemma 3.5 of \cite{[LW]} for any $r\in (0, r')$ we have
\beq
\frac {\Area_{g_i(t_0)}(C_{x_i}(B_{r}(x_i)\cap \Si_{i, t_0}))}{\pi {r}^2}\leq e^{\La r'} \frac {\Area_{g_i(t_0)}(C_{x_i}(B_{r'}(x_i)\cap \Si_{i, t_0}))}{\pi {r'}^2}\leq 1+3\dd, \label{lem:singular:015}
\eeq  where we used (\ref{lem:singular:011}).  We rescale the surface by
\beq
\td \Si_{i, s}=Q_i(\Si_{i, t_0+Q_i^{-2}s}-x_i),\quad \forall\; s\in (-(1+t_0)Q_i^2, (1-t_0)Q_i^2).\nonumber
\eeq
Then  $\{\td \Si_{i, s}, -1<s<1\}$ is a sequence of mean curvature flow with
  \beq
\max_{\td \Si_{i, s}\times (-1, 1)}|H|\leq Q_i^{-1}\La\ri 0.\nonumber
\eeq
By the choice of $Q_i$ we have
 \beq
\sup_{C_0(\td \Si_{i, 0}\cap B_{\frac 12Q_ir_0}(0))}|A|\leq 1.
\eeq
By Theorem 3.8 of \cite{[LW]}, there exists a universal constant $\ee$ such that
\beq
\sup_{C_0(\td \Si_{i, s}\cap B_{\frac 14Q_ir_0}(0))}|A|\leq \frac 1{\ee},\quad \forall \;s\in (-1, 1).\nonumber
\eeq Thus, by the compactness of mean curvature flow (cf. Theorem 2.6 of \cite{[LW]})   the surface $C_{0}(\td \Si_{i, 0}\cap B_{\frac 12 Q_ir_0}(0))$ converges in smooth topology to a complete smooth minimal surface $\td \Si_{\infty}$ with \beq \sup_{\td \Si_{\infty}}|A|\leq 1,\quad |A|(0)=1.  \label{lem:singular:010}\eeq
Since (\ref{lem:singular:015}) implies that
\beq
\frac {\Area_{\td g_i(0)}(C_0(B_r(0)\cap \td \Si_{i, 0}))}{\pi r^2}\leq 1+3\dd, \quad \forall\; r\in (0, Q_ir'), \label{lem:singular:012}
\eeq
we have
\beq
\frac {\Area_{\td g_\infty(0)}(C_0(B_r(0)\cap \td \Si_{\infty}))}{\pi r^2}\leq 1+3\dd, \quad \forall\; r>0. \label{lem:singular:013}
\eeq
By Lemma 3.6 of \cite{[LW]}, there exists a universal constant $\dd_0>0$ such that if we choose $\dd=\frac {\dd_0}4$, then $\td \Si_{\infty}$ must be a plane, which contradicts (\ref{lem:singular:010}).

Therefore, $m\geq 2$ in (\ref{lem:singular:005}) for any $r\in (0, r')$ where $r'=r'(\dd, \Si_{\infty, t_0}, R, \La, x_0, \cS_{t_0})$.
By Lemma \ref{lem:ratio} we can find $r_1=r_1(\dd, \Si_{\infty, t_0}, R, \La, x_0, \cS_{t_0})\in (0, r')$ such that for any $r\in (0, r_1)$
\beq
 \frac {\Area_{g_{\infty}(t_0)}(B_r(x_0)\cap \Si_{\infty, t_0})}{\pi r^2}\geq 1-\dd. \label{lem:singular:004}
\eeq
Since $m\geq 2$, for any $r\in (0, r_1)$ we have
\beq
\lim_{i\ri +\infty}\frac {\Area_{g_i(t_0)}(C_{x_i}(B_r(x_i)\cap \Si_{i, t_0}))}{\pi r^2}=m \frac {\Area_{g_{\infty}(t_0)}(B_r(x_0)\cap \Si_{\infty, t_0})}{\pi r^2}\geq 2(1-\dd).\label{lem:singular:014}
\eeq
  The lemma is proved.
\end{proof}

Using the boundedness of the mean curvature and Lemma \ref{lem:singular}, we  show that the singular set $\cS$ consists of locally finitely many Lipschitz curves.

\begin{lem}\label{prop:weakcompactness2}Fix large $R>0.$ Under the assumption of Proposition \ref{prop:weakcompactness}, the singular set $\cS$ is the union of locally finitely many space-time singular curves, i.e,
  $$\cS\cap \Big(B_R(0)\times (-1, 1)\Big)=\bigcup_{k=1}^{l}\Big\{( \xi_k(t), t)\,\Big|\,t\in (-1, 1), \;\xi_k(t)\in B_R(0)\cap \cS_t\Big\},$$ where  $\cS_t$ is defined in Proposition \ref{prop:weakcompactness} and  $\{\xi_k(t)\}_{k=1}^{l}$ are $\La'$-Lipschitz curves, i.e,
\beq
|\xi_k(t_1)-\xi_k(t_2)|\leq \La' |t_1-t_2|,\quad \forall\; t_1, t_2\in (-1, 1).
\eeq
Here $\La'$ depends only on the constant $\La$  in (\ref{eq:A203}).

\end{lem}

\begin{proof}For any $t_1\in (-1, 1)$ and any $p_{t_1}\in \cS_{t_1}\cap B_R(0)$, we show that there exists a Lipschitz curve in $\cS$ passing through $p_{t_1}$. Since $p_{t_1}$ is singular,  by Lemma \ref{lem:singular} we can find a sequence of points $p_{i, t_1}\in \Si_{i, t_1}$ and  $r'=r'(\Si_{\infty, t_1}, R, \La, p_{t_1}, \cS_{t_1})>0$ such that $p_{i, t_1}\ri p_{t_1}$ and for any $r\in (0, r')$,
\beq \lim_{i\ri +\infty}\frac {\Area(C_{p_{i, t_1}}(B_{r}(p_{i, t_1})\cap \Si_{i, t_1}))}{\pi r^2}\geq \frac 74. \label{X:001}\eeq
We choose $\eta_0>0$ and $M_0=200$ such that 
\beq
e^{-2\La^2 \eta_0}\Big(1+\frac 2{M_0}\Big)^{-2}\geq\frac 67,\quad M_0 \La \eta_0<r'.\label{Z:b4}
\eeq
Let $t_2\in (t_1-\eta_0, t_1+\eta_0)\cap (-1, 1)$ and $r_1=M_0\La |t_2-t_1|$. Then we have $r_1<M_0\La \eta_0<r'$.  
By Lemma 3.4 in Li-Wang \cite{[LW]}  we have
\beqn &&\frac {\Area(C_{p_{i, t_2}}(B_{r_2}(p_{i, t_2})\cap \Si_{i, t_2}))}{\pi {r_2}^2}\nonumber\\&\geq& e^{-\La^2|t_2-t_1|}\Big(1+\frac {2\La}{r_1}|t_2-t_1|\Big)^{-2}\frac {\Area(C_{p_{i, t_1}}(B_{r_1}(p_{i, t_1})\cap \Si_{i, t_1}))}{\pi r_1^2}\nonumber\\
&\geq&e^{-\La^2\eta_0}\Big(1+\frac {2}{M_0}\Big)^{-2}\frac {\Area(C_{p_{i, t_1}}(B_{r_1}(p_{i, t_1})\cap \Si_{i, t_1}))}{\pi r_1^2}, \label{X:002}
\eeqn where $r_2=r_1+2\La|t_2-t_1|$. Combining (\ref{X:001})-(\ref{X:002}), we have
\beq
\lim_{i\ri +\infty}\frac {\Area(C_{p_{i, t_2}}(B_{r_2}(p_{i, t_2})\cap \Si_{i, t_2}))}{\pi {r_2}^2}\geq \frac 32,  \label{X:003}
\eeq where $p_{i, t_2}=\x_{i, t_2}(\x_{i, t_1}^{-1}(p_{i, t_1})). $ Since the mean curvature is uniformly bounded along the flow,  all points $\{p_{i, t_2}\}_{i=1}^{\infty}$ lie in a bounded ball centered at $p_{t_1}$. Thus,  we can find a subsequence of $\{p_{i, t_2}\}_{i=1}^{\infty}$ such that it converges to a point, which we denoted by $p_{t_2}$. Since $C_{p_{i, t_2}}(B_{r_2}(p_{i, t_2})\cap \Si_{i, t_2})$ converges locally smoothly to $B_{r_2}(p_{ t_2})\cap \Si_{\infty, t_2}$ away from singularities,
by Step 2 of Lemma \ref{lem:multiplicity} we have
\beq
\lim_{i\ri +\infty}\frac {\Area(C_{p_{i, t_2}}(B_{r_2}(p_{i, t_2})\cap \Si_{i, t_2}))}{\pi {r_2}^2}=m\frac {\Area(B_{r_2}(p_{ t_2})\cap \Si_{\infty, t_2})}{\pi {r_2}^2}, \label{Z:a1}
\eeq where $m\in \NN$. Note that $B_R(0)\cap \Si_{\infty, t}$ has bounded geometry for any $t\in (-1, 1)$, we can find a uniform $r_2'>0$ such that for any $(p, t)\in (\Si_{\infty, t}\cap B_R(0))\times (-1, 1)$ and any $r\in (0, r_2')$,
\beq
\frac {\Area(B_{r}(p)\cap \Si_{\infty, t})}{\pi {r}^2}\leq \frac 54.\label{Z:a2}
\eeq
Moreover,  we can choose $\eta_0$ small such that 
\beq r_2=r_1+2\La|t_2-t_1|=(M_0+2)\La |t_2-t_1|\leq (M_0+2)\La \eta_0<r_2'. \label{z:b3}\eeq
Combining (\ref{X:003})-(\ref{z:b3}), we have $m\geq 2$ in (\ref{Z:a1}). Thus, $C_{p_{i, t_2}}(B_{r_2}(p_{i, t_2})\cap \Si_{i, t_2})$ converges locally smoothly to $B_{r_2}(p_{ t_2})\cap \Si_{\infty, t_2}$ with multiplicity $m\geq 2$.
This implies that  $B_{r_2}(p_{t_2})\cap \Si_{\infty, t_2}$ contains a singular point, which we denoted by $q_{t_2}$.
Here we used the fact that if $B_{r_2}(p_{t_2})\cap \Si_{\infty, t_2}$ contains no singular points, then $C_{p_{i, t_2}}(B_{r_2}(p_{i, t_2})\cap \Si_{i, t_2})$ will converge smoothly to $B_{r_2}(p_{t_2})\cap \Si_{\infty, t_2}$ with multiplicity one.

Note that
$$|p_{i, t_1}-p_{i, t_2}|\leq \int_{t_1}^{t_2}\,|H|\,dt\leq \La|t_1-t_2|. $$
Taking the limit $i\ri +\infty$ we have
$|p_{t_1}-p_{t_2}|\leq \La |t_1-t_2|.$
Thus, for any $t_2\in (t_1-\eta_0, t_1+\eta_0)\cap (-1, 1)$ we have
\beqs
|q_{t_2}-p_{t_1}|&\leq & |q_{t_2}-p_{t_2}|+|p_{t_2}-p_{t_1}|
\leq r_2+\La |t_1-t_2|\\
&\leq&(M_0+3)\La |t_1-t_2|.
\eeqs
Therefore, $p_{t_1}$ lies in a $\La'$-Lipschitz curve in $\cS$ with $\La'=203\La$. Since for any $t\in (-1, 1)$  the set $\cS_t$ is
locally finite by Proposition \ref{prop:weakcompactness}, the singular curves are locally finite. The lemma is proved.

\end{proof}

\section{The rescaled mean curvature flow}

In this section, we will show the smooth convergence of rescaled mean curvature flow under uniform mean curvature bound. As is pointed out in the introduction,  we have no long-time pseudolocality of the flow and the singularities don't move along straight lines. When the multiplicity of the convergence is greater than one, in order to show the $L$-stability of the limit self-shrinker we need  an estimate on  the asymptotical behavior of the positive solution near the singular set (cf. Lemma \ref{claim2} and Lemma \ref{lem:B002}), and the proof of this estimate will be delayed to Section \ref{sec:analysis}.

\begin{theo}
\label{theo:removable}
Let $\{(\Si^2, \x(t)), 0\leq t<+\infty\}$ be a closed smooth embedded rescaled mean curvature flow
\beq \Big(\pd {\x}t\Big)^{\perp}=-\Big(H-\frac 12\langle \x, \n\rangle\Big)\n\label{eq:BB001}\eeq
 satisfying
\beq d(\Si_t, 0)\leq D,\quad \hbox{and}\quad \max_{\Si_t}|H(p, t)|\leq \La \label{eq:BB002}\eeq
for two constants $D, \La>0$.  Then for any $t_i\ri +\infty$ there exists a subsequence of $\{\Si_{t_i+t}, -1<t<1\}$ such that it  converges in smooth topology to a complete smooth self-shrinker with multiplicity one as $i\ri +\infty$.
\end{theo}

We sketch the proof of Theorem \ref{theo:removable}.  First, we show the weak compactness for any sequence of the rescaled mean curvature flow in Lemma \ref{lem:A001}. Suppose that the multiplicity is at least two.
By using the decomposition of spaces(cf. Definition~\ref{def:GD18_1}) we can select a special sequence $\{t_i\}$ in Lemma \ref{lem:A002} for each $\ee>0$. This special sequence is needed to control the upper bound of the function $w_i$ away from the singular set by using the parabolic Harnack inequality (cf. Lemma \ref{lem:A004a}).  Then we can take the limit for the function $w_i$ and obtain a positive function $w$ with uniform  bounds on any compact set away from the singular set(cf. Lemma ~\ref{lem:A004b}). The function $w$ satisfies the linearized mean curvature flow equation. To study the growth behavior of $w$ near the singular set, we take a sequence of $\ee_i\ri 0$ and for each $\ee_i$ we repeat the above process to get  a sequence of functions $\{w_{i, k}\}_{k=1}^{\infty}$. After choosing a diagonal sequence and taking the limit, we get a  function $w$ with good growth estimates near the singular set (cf. Proposition \ref{lem:A003d}) by assuming Theorem \ref{theoA:main} in the next section.  The bounds of $w$ imply  the $L$-stability of the limit self-shrinker (cf. Lemma~\ref{lem:stable1}), and this step also relies on Theorem \ref{theoA:main}. However, the  limit self-shrinker  is not $L$-stable by Colding-Minicozzi's theorem (cf. Theorem \ref{theo:CM1}) and we obtain a contradiction.

\subsection{Convergence away from singularities}

We recall Ilmanen's local Gauss-Bonnet formula in \cite{[Il]} to control the $L^2$ norm of the second fundamental form. Let $\Si$ be a smooth surface with smooth boundary $\p\Si.$ We denote by  $e(\Si)$ the genus of $\Si$ which is the genus of the closed surface obtained by capping off the boundary components of $\Si$ by disks.

\begin{lem}\label{lem:ilmanen}(cf. Ilmanen \cite{[Il]})
Let $R>1$ and let $\Si$ be a surface properly immersed in $B_R(p)$. Then for any $\ee>0$ we have
\beqs
(1-\ee)\int_{\Si\cap B_1(p)}\,|A|^2\,d\mu&\leq& \int_{\Si\cap B_R}\,|H|^2\,d\mu+8\pi e(\Si\cap B_R(p))\\&&+ \frac {24\pi R^2}{\ee(R-1)^2} \sup_{r\in [1, R]}\frac {\Area(\Si\cap B_r(p))}{\pi r^2}.
\eeqs

\end{lem}

For simplicity, we introduce the following definition.

\begin{defi}\label{defi:soliton}
Let $\rho: \RR^+\ri \RR^+$ be an increasing positive function. For any $ N>0$, we denote by $\cC( N, \rho)$ the space of all smooth embedded  self-shrinkers $\Si^2 \subset \RR^3$ satisfying the properties that for any $r>0$ and $p\in \Si$,
$$\int_{\Si\cap B_r(0)}\,|A|^2\leq \rho(r)\quad \hbox{and}\quad  \Area(B_r(p)\cap \Si)\leq \pi Nr^2.$$
\end{defi}

Note that  the space $\cC( N, \rho)$ is compact in the smooth topology by Colding-Minicozzi \cite{[CM1]}, and the distance from the origin to any self-shrinker in $\RR^3$ is at most $2$ by avoidance principle(cf. Corollary 3.6 of \cite{[Eckbook]}).
The total curvature bound in Definition~\ref{defi:soliton} can also be derived from genus bound by exploiting Lemma~\ref{lem:ilmanen}.

The following result shows that   the rescaled mean curvature flow converges  locally smoothly to a self-shrinker away from singularities.

\begin{lem}\label{lem:A001}Under the assumption of Theorem \ref{theo:removable}, for any sequence $t_i\ri +\infty$, there is a smooth self-shrinker $\Si_{\infty}\in \cC( N, \rho) $ and a space-time set $\cS\subset \Si_{\infty}\times\RR$ satisfying the following properties.
\begin{enumerate}
\item[(1).] For any $T>1$, there is a subsequence, still denoted by $\{t_i\}$,  such that $\{\Si_{t_i+t}, -T<t<T\}$ converges in smooth topology, possibly with multiplicities, to  $\Si_{\infty}$  away from  $\cS$;

\item[(2).] For any $R>0$,  $\cS\cap (B_R(0)\times (-T, T))$ consists of finite many $\si$-Lipschitz curves with Lipschitz constant $\si$ depending only on $\La, T$ and $R$;

\item[(3).] The convergence in part (1) is also in (extrinsic) Hausdorff distance;

\item[(4).] The limit self-shrinker $\Si_{\infty}$ is independent of the choice of $T$. In other words, for different $T$ we can choose two different subsequences of $\{t_i\}$ such that the corresponding flows in part (1) have  the same limit self-shrinker $\Si_{\infty}$.
\end{enumerate}

\end{lem}

\begin{proof}We divide the proof into the following steps.

\emph{Step 1. The area ratio along the flow (\ref{eq:BB001}) is uniformly bounded from above.} In fact, we rescale the flow (\ref{eq:BB001}) by \beq
s=1-e^{-t},\quad \hat \Si_s=\sqrt{1-s}\,\Si_{-\log(1-s)} \label{eq:MCF2}
\eeq such that $\{\hat \Si_s, 0\leq s<1 \}$   is a mean curvature flow satisfying the equation (\ref{eq:MCF}). By Lemma 2.9 of Colding-Minicozzi \cite{[CM2]} and Lemma 2.3 in Li-Wang \cite{[LW]}, we have that the area ratio of (\ref{eq:BB001}) is uniformly bounded from above.\\

\emph{Step 2. For any large $R$ , the energy of $\Si_t\cap B_R(0)$ is uniformly bounded  along the flow (\ref{eq:BB001}).} In fact, by Lemma \ref{lem:ilmanen} we have
\beqn \int_{\Si_t\cap B_R(0)}\,|A|^2\,d\mu_t&\leq &2\int_{\Si_t\cap B_{2R}(0)}\,|H|^2\,d\mu_t+C(N, e(\Si) )\nonumber
\\&\leq& 8\pi N\La^2R^2 + C(N, e(\Si) ), \label{eq:AA005} \eeqn
where $N$ denotes the upper bound of the area ratio. Therefore, for any $t>0$ the energy of $\Si_t\cap B_R(0)$ is   bounded by a constant $C(N, \La, R, e(\Si)). $\\

\textit{Step 3. For each sequence $t_i \ri + \infty$, we  obtain a refined sequence converging to a limit self-shrinker.}
 For any sequence $t_i\ri +\infty,$ we  rescale the flow $\Si_t$ by
\beq
s=1-e^{-(t-t_i)},\quad \td \Si_{i, s}=\sqrt{1-s}\;\Si_{ t_i-\log(1-s)} \label{eq:BB003}
\eeq  such that for each $i$ the flow   $\{\td \Si_{i, s}, 1-e^{t_i}\leq s< 1\}$ is a mean curvature flow satisfying (\ref{eq:MCF}) with the following properties:
\begin{enumerate}
  \item[$(a)$.]  For any small $\la>0$, the mean curvature of $\td \Si_{i, s}$ satisfies
  $$\max_{\td \Si_{i, s}\times [1-e^{t_i}, 1-\la]}|\td H_i|(p, s)\leq \td \La:= \frac {\La}{\sqrt{\la}};$$
\item[$(b)$.] For any large $R$, the energy of $\td \Si_{i, s}\cap B_R(0)$ is uniformly bounded ;

  \item[$(c)$.] The area ratio is uniformly bounded from above;
  \item[$(d)$.]  The area ratio is uniformly bounded from below;
  \item[$(e)$.] There exists a constant $D'>0$ such that $d(\td \Si_{i, s}, 0)\leq D'$ for any $i$.
\item[$(f)$.] We have
\beq \lim_{i\ri +\infty}\int_{-T}^{1-\la}\,dt\int_{\td \Si_{i, s}}\, e^{-\frac {|\td \x_i|^2}{4(1-s)}}\Big|\td H_i-\frac 1{2(1-s)}\langle \td \x_i, \n \rangle\Big|^2\,d\td \mu_{i, s}=0. \label{eq:G006}\eeq
\end{enumerate}
In fact, Property $(a)$ and $(e)$ follow from the assumption (\ref{eq:BB002}), and Property $(b)$ follows from (\ref{eq:AA005}).
Property $(c)$  follows  from  Step 1,  and Property $(d)$ follows from Lemma 3.5 in Li-Wang \cite{[LW]}. To prove Property $(f),$   by Huisken's monotonicity formula along the rescaled mean curvature flow (\ref{eq:BB001}) we have
\beq
\frac d{dt}\int_{\Si_t}\,e^{-\frac {|\x|^2}4}\,d\mu_t=-\int_{\Si_t}\,e^{-\frac {|\x|^2}4}\Big|H-\frac 12\langle\x, \n\rangle\Big|^2\,d\mu_t.
\eeq This implies that
$$\int_0^{\infty}\,dt\int_{\Si_t}\,e^{-\frac {|\x|^2}4}\Big|H-\frac 12\langle\x, \n\rangle\Big|^2\,d\mu_t<+\infty.$$
Let $T, \la>0$ with $-T<1-\la$.
For any $t_i\ri +\infty$, we have
\beq
\lim_{t_i\ri +\infty}\int_{t_i-\log (1+T)}^{t_i-\log \la}\,dt \int_{\Si_t}\,e^{-\frac {|\x|^2}4}\Big|H-\frac 12\langle\x, \n\rangle\Big|^2\,d\mu_t=0. \label{eq:G007}
\eeq
Then (\ref{eq:G006}) follows from   (\ref{eq:BB003}) and (\ref{eq:G007}).
 Therefore, by Definition \ref{defi:refined} for any $T>0$, small $\la>0$ and any $s_0\in [-T+1, -\la]$ the sequence $\{\td \Si_{i, s_0+\tau}, -1<\tau<1\}$ is a refined sequence. By Proposition \ref{prop:weakcompactness} and Lemma \ref{prop:weakcompactness2} a subsequence of  $\{\td \Si_{i, s}, -T<s<1-\la\}$  converges in smooth topology, possibly with multiplicities, to a self-shrinker flow $\{\td \Si_{\infty, s}, -T<s<1-\la\}$ away from a space-time, $\td \La'$-Lipschitz singular set $\td \cS$ with $\td \La'=203\td \La$.\\

\emph{Step 4.}
Let $t'=t-t_i$ and  $\Si_{i, t'}=\Si_{t_i+t'}$.
Since $\{\td \Si_{i, s}, -T<s<1-\la\}$ converges locally smoothly to $\{\td \Si_{\infty, s}, -T<s<1-\la\}$  away from $\td \cS$, by (\ref{eq:BB003}) the flow $\{\Si_{i, t'}, -\log(1+T)<t'<-\log \la \}$ also converges locally smoothly to a self-shrinker  $  \Si_{\infty}$ satisfying
$$H-\frac 12\langle\x, \n\rangle=0$$
 away from a space-time singular set
$\cS$ with
\beq
\cS_{t'}=\frac 1{\sqrt{1-s}}\td \cS_{s}.
\eeq    Here $s=1-e^{-t'}$.
Now we show the Lipschitz property of $\cS$.
By (\ref{eq:BB003}), for  any curve $\xi(t')$ of $\cS $, we can find a curve $\xi(s)$ of $\td \cS$ such that
\beq
\td \xi(s)=\sqrt{1-s}\,\xi(t'),\quad t'=-\log(1-s).
\eeq
Since $\td \xi(s)$ is $\td \La'$-Lipschitz, we have
\beq
|\td \xi(s_1)-\td \xi(s_2)|\leq \td \La' |s_1-s_2|,\quad \forall\; s_1, s_2\in (-T, 1-\la),  \label{eq:lem34:001}\eeq
which implies that
\beq
|e^{-\frac {t_1}2}\xi(t_1)-e^{-\frac {t_2}2}\xi(t_2)|\leq \td \La' |e^{-\frac {t_1}2}-e^{-\frac {t_2}2}|.
\eeq
Suppose that $|\xi(t)|\leq R.$
For any $t_1', t_2'$ with $|t_1'-t_2'|\leq 1$ we have
\beqn
|\xi(t_1')-\xi(t_2')|&=&|\xi(t_1')-e^{\frac {t_1'-t_2'}2}\xi(t_2')|+|e^{\frac {t_1'-t_2'}2}-1||\xi(t_2')|\nonumber\\
&\leq&\td \La'|1-e^{\frac {t_1'-t_2'}2}|+|e^{\frac {t_1'-t_2'}2}-1||\xi(t_2')|\nonumber\\
&\leq&(\td \La'+R)|t_1'-t_2'|, \label{eq:lem34:002}
\eeqn where we used the inequality
$$|e^x-1|\leq 2|x|,\quad \forall\, x\in [-1, 1]. $$
Note that the Lipschitz constant in (\ref{eq:lem34:002}) is given by
$\si=\td \La'+R. $
Thus, if we consider the convergence of $\{\Si_{t_i+t}, -T<t<T\}$ as in part (1), then $\cS\cap (B_R(0)\times (-T, T))$ consists of Lipschitz curves with Lipschitz constant $\si. $
The convergence is also  in extrinsic Hausdorff distance by Proposition \ref{prop:weakcompactness} and the limit self-shrinker is independent of the choice of $T$ by the argument of Claim 4.3 of \cite{[LW]}.
 The lemma is proved.

\end{proof}

 \subsection{Decomposition of spaces}

In this subsection, we follow the argument in \cite{[LW]} to decompose the space and define an  almost ``monotone  decreasing" quantity, which will be used to select time slices such that the limit self-shrinker is $L$-stable. First, we
   decompose the space as follows.

 \begin{defi} \label{def:GD18_1} (\cite{[LW]}) Fix large $R>0$ and small $\ee>0.$
 \begin{enumerate}
                                   \item[(1).] We define the set $\mathbf{S}=\mathbf{S}(\Sigma_t, \epsilon, R)=\{y \in \Sigma_t\; |\; |y|<R,\; |A|(y, t)>\epsilon^{-1}  \}.$
                                   \item[(2).]  The ball $B_{R}(0)$ can be decomposed into three parts as follows:
  \begin{itemize}
  \item the high curvature part $\mathbf{H}$, which is defined by  $$\mathbf{H}=\mathbf{H}(\Sigma_t, \epsilon, R)=\Big\{x\in \RR^3\; \Big|\; |x|<R,  d(x,\mathbf{S}) <\frac {\epsilon}2 \Big\}.$$
        \item  the thick part $\mathbf{TK}$, which is defined by
       \begin{align*}
          \mathbf{TK}&=\mathbf{TK}(\Sigma_t, \epsilon, R)\\&=\Big\{x\in \RR^3\; \Big| \;|x|<R, \; \textrm{there is a continuous curve $\gamma \subset B_{R}(0) \backslash (\mathbf{H} \cup \Sigma_t)$} \\
                             &\qquad \qquad \textrm{ connecting $x$ and some $y$ with} \; B(y,\epsilon) \subset B_{R}(0) \backslash (\mathbf{H} \cup \Sigma_t)\Big\}.
       \end{align*}
  \item the thin part $\mathbf{TN}$, which is defined by  $\mathbf{TN}=\mathbf{TN}(\Sigma_t, \epsilon, R)=B_{R}(0) \backslash (\mathbf{H} \cup \mathbf{TK})$.\\

  \end{itemize}
                                 \end{enumerate}
 \end{defi}

As is pointed out in \cite{[LW]}, the high curvature part $\mathbf{H}$ is the neighborhood of points with large
second fundamental form, and the thin part $\mathbf{TN}$ is the domain between
the top and bottom sheets. Moreover, the thick part $\mathbf{TK}$ is the union of path connected components of the domain ``outside" the sheets. The readers are referred to \cite{[LW]} for more explanation on the definition.

As in Colding-Minicozzi \cite{[CM1]}, we define the $L$-stability of a self-shrinker.

\begin{defi}\label{defi:SpaceSelfshrinker}
 For any $R>0$, a complete smooth self-shrinker $\Si^n\subset \RR^{n+1}$ is called $L$-stable in the ball $B_R(0)$, if for any  function $\varphi\in W_0^{1, 2}(B_R(0))$, we have
\beq
\int_{\Si}\,-\varphi L_{\Si}\varphi \,e^{-\frac {|x|^2}{4}}\geq 0, \label{eq:F002}
\eeq where $L_{\Si}$ is the operator on $\Si$ defined by
$$L_{\Si}=\Delta-\frac 12\langle x, \Na (\cdot)\rangle+|A|^2+\frac 12. $$ The subindex $\Si$ in $L_{\Si}$ will be omitted when it is clear in the context.
We say $\Si$ is not  $L$-stable in the ball $B_R(0)$ if (\ref{eq:F002}) doesn't hold for some $\varphi\in W_0^{1, 2}(B_R(0)).$ We call that $\Si$ is $L$-stable in $\RR^{n+1}$ if $\Si$ is  $L$-stable in the ball $B_R(0)$ of $\RR^{n+1}$ for any $R>0$.
\end{defi}

Recall  Colding-Minicozzi's result:

\begin{theo}\label{theo:CM1}(cf. \cite{[CM1]}\cite{[CM2]})There are no $L$-stable smooth complete self-shrinkers without boundary and with polynomial volume growth in $\RR^{n+1}$.

\end{theo}

As a   corollary of Theorem \ref{theo:CM1}, we have the following result.

\begin{lem}\label{lem:A002b}Let  $ N>0$ and $\rho$ an increasing positive function. There exists $R_0=R_0( N, \rho)>0$ such that any self-shrinker $\Si\in \cC( N, \rho)$ is not $L$-stable in the ball $B_{R_0}(0)$.

\end{lem}
\begin{proof}For otherwise, we can find a sequence $R_i\ri +\infty$ and self-shrinkers $\Si_i\in  \cC(N, \rho)$ such that $\Si_i$ is $L$-stable in the ball $B_{R_i}(0)$. By smooth compactness of $\cC(N, \rho)$ in \cite{[CM1]}, a subsequence of $\{\Si_i\}$ converges smoothly to a self-shrinker $\Si_{\infty}\in \cC(N, \rho) $. By Theorem \ref{theo:CM1},  $\Si_{\infty}$ is not $L$-stable in a ball $B_{R_0}(0)$ for some $R_0>0$. This implies that there exists a smooth function $\varphi_{\infty}\in C_0^{\infty}(\Si_{\infty}\cap B_{R_0}(0))$ such that
\beq
\int_{\Si_{\infty}}\,-\varphi_{\infty} L_{\Si}\,\varphi_{\infty} \,e^{-\frac {|x|^2}{4}}< 0.\label{eq:lem38:001}\eeq
Since $\Si_i$ converges smoothly to $\Si_{\infty}$, we define the map
$f_i: \Si_{\infty}\cap B_{R_0+1}(0)\ri \Si_i $ by
\beq
f_i(x)=x+u_i(x)\n(x),\quad \forall\; x\in \Si_{\infty}\cap B_{R_0+1}(0),\nonumber
\eeq where $\n(x)$ denotes the normal vector field of $\Si_{\infty}$ and $u_i(x)$ is the graph function of $\Si_i$ over $\Si_{\infty}.$ Let $\Om=\Si_{\infty}\cap B_{R_0+1}(0)$ and $\Om_i=f_i(\Om) \subset \Si_i$.  We assume that $i$ is large such that $\Om_i\subset \Si_i\cap B_{R_0+2}(0)$.
Note that $f_i$ converges smoothly to the identity map on $\Om $ as $i\ri +\infty$ and for large $i$ its inverse map $f_i^{-1}: \Om_i\ri\Om$ exists and is also smooth. Moreover, $f_i^{-1}$ also converges smoothly to the identity map on $\Om $ as $i\ri +\infty$. We define the function $\varphi_i:=(f_i^{-1})^*\varphi_{\infty}\in C_0^{\infty}(\Om_i)$ and we can extend $\varphi_i$ to $\Si_i$ such that $\varphi$ is zero on $\Si_i\b\Om_i$.  Then by (\ref{eq:lem38:001}) the
function $\varphi_i\in C_0^{\infty}(\Si_i)$ satisfies
\beq \lim_{i\ri+\infty} \int_{\Si_{i}}\,-\varphi_{i} L_{\Si_i}\,\varphi_{i} \,e^{-\frac {|x|^2}{4}}=\int_{\Si_{\infty}}\,-\varphi_{\infty} L_{\Si}\,\varphi_{\infty} \,e^{-\frac {|x|^2}{4}}<0.\nonumber\eeq
Thus, for large $i$ we have
\beq
\int_{\Si_{i}}\,-\varphi_{i} L_{\Si_i}\,\varphi_{i} \,e^{-\frac {|x|^2}{4}}<0. \label{eqn:205}
\eeq
Note that $\Supp(\varphi_i)\subset \Om_i\subset \Si_i\cap B_{R_0+2}(0)$ for large $i$. Thus, the inequality (\ref{eqn:205}) contradicts our assumption that $\Si_i$ is  $L$-stable in the ball $B_{R_i}(0)$ and $R_i\ri +\infty$. The lemma is proved.

\end{proof}

\begin{lem}\label{lem:rx}Let $R, N>0$ and  $\rho$ an increasing positive function.  For any $\Si\in \cC(N, \rho)$ and $x\in \Si$, we define $r_{\Si}(x)$ the supreme of the radius $r$ such that
\beq
B_r(x+r\n(x))\cap \Si=\emptyset, \quad B_r(x-r\n(x))\cap \Si=\emptyset,
\eeq where $\n(x)$ denotes the normal vector of $\Si$ at $x$.
 Then there exists $\ee_0(R,  N, \rho)>0$ such that
for any $\Si\in \cC(N, \rho)$ and $x\in \Si\cap B_R(0)$ we have
$$  r_{\Si}(x)\geq \ee_0.  $$
\end{lem}
\begin{proof}We divide the proof into several steps.

\emph{Step 1. }
For otherwise, we can find a sequence of $\Si_i\in \cC(N, \rho)$ and   points $x_i\in \Si_i\cap B_R(0)$ with $\dd_i:=r_{\Si_i}(x_i)\ri 0$. By the smooth compactness of $\cC( N, \rho)$, there is a subsequence of $\{\Si_i\}$ converging smoothly to a self-shrinker $\Si_{\infty}$ in $\cC( N, \rho)$. We assume that $x_i\ri x_{\infty}\in \Si_{\infty}\cap B_{R+1}(0)$. By the embeddedness of $\Si_{\infty}$, we have $\dd:=r_{\Si_{\infty}}(x_{\infty})>0.$ Since $\Si_{\infty}$ is smooth and embedded, there exists $r'>0$ such that $B_{r'}(x_{\infty})\cap \Si_{\infty}$ has only one component and
\beq
\inf_{y\in B_{r'}(x_{\infty})\cap \Si_{\infty}} r_{\Si_{\infty}}(y)\geq \frac {\dd}{2}. \label{eq:lem39:001}
\eeq
Moreover, we choose $r'$ sufficiently small such that $B_{r'}(x_{\infty})\cap \Si_{\infty}$ is almost flat by Lemma \ref{lem:graph1}.
Let
\beq
\Om(r', \frac {\dd}2):=\bigcup_{y\in B_{r'}(x_{\infty})\cap \Si_{\infty}}\Big(B_{\frac {\dd}2}(
y+\frac {\dd}2\n_{\Si_{\infty}}(y))\cup B_{\frac {\dd}2}(
y-\frac {\dd}2\n_{\Si_{\infty}}(y))\Big). \label{eq:lem39:002}
\eeq
Then (\ref{eq:lem39:001}) implies that $\Om(r', \frac {\dd}2)\cap \Si_{\infty}=\emptyset.$  By the smooth convergence of $\Si_{i}$ to $\Si_{\infty}$, for large $i$ we have
\beq
\Om(\frac {r'}2, \frac {\dd}4)\cap (\Si_i\b B_{r'}(x_{\infty}))=\emptyset. \label{eq:lem39:006}
\eeq
By the construction of $\Om(r', \frac {\dd}2)$, we have
\beq
B_{\frac {\dd}4}(y)\subset \Om(\frac {r'}2, \frac {\dd}4)\cup (\Si_{\infty}\cap B_{r'}(x_{\infty})),\quad \forall\; y\in B_{\frac {r'}4}(x_{\infty})\cap \Si_{\infty}. \label{eq:lem39:003}\eeq\\

\emph{Step 2. }
 Since $x_i\ri x_{\infty}$,  we can choose $r'$ sufficiently small such that for all large $i$ the projection  of $x_i$ to $\Si_{\infty}$ lie in the ball $B_{\frac {r'}2}(x_{\infty})$. This can be done since   $B_{r'}(x_{\infty})\cap \Si_{\infty}$ is almost flat. We denote by $y_i$ the projection  of $x_i$ to $\Si_{\infty}$ and we have $y_i\in B_{\frac {r'}2}(x_{\infty})\cap \Si_{\infty}$. Let $s_i\in \RR$ such that $y_i+s_i\n_{\Si_{\infty}}(y_i)=x_i$.  Combining this with $x_i\ri x_{\infty}$, we have \beq
B_{2\dd_i}(x_i\pm 2\dd_i\n_{\Si_i}(x_i))\subset B_{4\dd_i}(x_i)\subset B_{4\dd_i+|s_i|}(y_i). \label{eq:lem39:004}
\eeq
On the other hand,    $|s_i|\ri0$ and for large $i$ we have
\beq
B_{4\dd_i+|s_i|}(y_i)\subset B_{\frac {\dd}4}(y_i). \label{eq:lem39:005}
\eeq
Combining (\ref{eq:lem39:003})-(\ref{eq:lem39:005}), we have
\beq
B_{2\dd_i}(x_i\pm 2\dd_i\n_{\Si_i}(x_i))\subset  \Big(\Om(\frac {r'}2, \frac {\dd}4)\cup (\Si_{\infty}\cap B_{r'}(x_{\infty}))\Big). \label{eq:lem39:007}
\eeq\\

\emph{Step 3. } We  show that
\beq
B_{2\dd_i}(x_i\pm 2\dd_i\n_{\Si_i}(x_i))\cap \Si_i=\emptyset. \label{eq:lem39:008}
\eeq
Let $\Si_i=\Si_i^{(1)}\cup \Si_i^{(2)}$, where $\Si_i^{(1)}$ and $\Si_i^{(1)}$ are defined by
\beq
\Si_i^{(1)}=\Si_i\cap B_{r'}(x_{\infty}),\quad \Si_i^{(2)}=\Si_i\b B_{r'}(x_{\infty}).
\eeq
By the smooth convergence of $\Si_{i}$ to $\Si_{\infty}$ and the choice of $r'$ such that $B_{r'}(x_{\infty})\cap \Si_{\infty}$ is almost flat, we have that for large $i$  $B_{r'}(x_{\infty})\cap \Si_{i}$ is also almost flat. Consequently, for large $i$ we have
\beq
B_{2\dd_i}(x_i\pm 2\dd_i\n_{\Si_i}(x_i))\cap \Si_i^{(1)}=\emptyset.  \label{eq:lem39:009}
\eeq
On the other hand, (\ref{eq:lem39:006}) and (\ref{eq:lem39:007}) imply that
\beqn
B_{2\dd_i}(x_i\pm 2\dd_i\n_{\Si_i}(x_i))\cap \Si_i^{(2)}&\subset& \Big(\Om(\frac {r'}2, \frac {\dd}4)\cup (\Si_{\infty}\cap B_{r'}(x_{\infty}))\Big)\cap \Si_i^{(2)}\nonumber\\
&=&  \Om(\frac {r'}2, \frac {\dd}4)\cap \Si_i^{(2)}=\emptyset,  \label{eq:lem39:010}
\eeqn where we used the fact that $\Si_i^{(2)}\cap B_{r'}(x_{\infty})=\emptyset$. Thus,   (\ref{eq:lem39:008}) follows from (\ref{eq:lem39:009})-(\ref{eq:lem39:010}). Note that (\ref{eq:lem39:008}) contradicts the definition of $\dd_i=r_{\Si_i}(x_i)$. The lemma is proved.
\end{proof}

A direct corollary of Lemma \ref{lem:rx}
is the following result.

\begin{lem}\label{lem:TN}Let $R,  N>0$ and  an increasing positive function $\rho$. Then  there exists a constant $\ee_0(R,  N, \rho)>0$ such that for any $\ee\in (0, \ee_0)$ we have  \beq |\mathbf{TN}( \Sigma, \ee, R)| =0,\quad \forall\;\Si\in \cC( N, \rho). \label{eq:TN}\eeq
Here the notation $|\Omega|$ denotes the volume of $\Omega$ with respect to the standard metric on $\RR^3$.
\end{lem}
\begin{proof}We choose $\ee_0$ the same constant in Lemma \ref{lem:rx}. Thus, (\ref{eq:TN}) follows from Lemma \ref{lem:rx} and the definition of $\mathbf{TN}.$
\end{proof}

Using Lemma \ref{lem:TN} we show that the quantity $|\mathbf{TN}|$ along the flow will tend to zero.
 \begin{lem} \label{lma:GD18_1} Fix $R,  N>0$, and an increasing positive function $\rho$.   Under the assumption of Theorem \ref{theo:removable}, there exists a constant $\ee_0(R,  N, \rho)>0$ such that for any  $\epsilon\in (0, \ee_0)$, we have
   $$
       \displaystyle \lim_{t \to \infty}   |\mathbf{TN}( \Sigma_{t}, \ee, R)| =0.
   $$
 \end{lem}
 \begin{proof} By Lemma \ref{lem:A001}, for any $t_i\ri \infty$ there exists a subsequence, still denoted by $\{t_i\}$, such that it converges locally smoothly to a limit self-shrinker $\Si_{\infty}\in \cC( N, \rho)$ away from the singular set $\cS_0\subset \RR^3$. For any $\ee>0$, by Definition \ref{def:GD18_1} we have \beq
\mathbf{TN}(\Sigma_{t_i}, \ee, R)\ri \mathbf{TN}(\Sigma_{\infty}, \ee, R)\b B_{\frac {\ee}2}(\cS_0),\nonumber
\eeq where $B_{\ee}(\cS_0)=\cup_{p\in \cS_0}B_{\ee}(p).$
Therefore, by Lemma \ref{lem:TN} we have
$$\lim_{t_i\ri +\infty}|\mathbf{TN}( \Sigma_{t_i}, \ee, R)|\leq \lim_{t_i\ri +\infty}|\mathbf{TN}(  \Sigma_{\infty}, \ee, R)|=0,
$$ where $\ee\in (0, \ee_0)$ and $\ee_0$ is the constant in Lemma \ref{lem:TN}.  The lemma is proved.

 \end{proof}

As in  Lemma 4.7 of \cite{[LW]}, we have
\begin{lem}\label{lma:GD18_3}Fix $ R>0$ and $\tau\in (0, 1).$ Let $\{t_i\}$ be any sequence as in Lemma \ref{lem:A001}.
 If the multiplicity of the convergence in Lemma \ref{lem:A001} is  more than one, then for any $\ee>0$, there exists $i_0>0$ such that for any $i\geq i_0$ we have $$\inf_{t\in [ t_i-\tau, t_i]}|\mathbf{TN}( \Si_{t}, \ee, R)|>0.$$
\end{lem}
\begin{proof}Since $\Si_t$ is embedded and $\{\Si_{t_i+t}, -\tau\leq t\leq \tau\}$ converges locally smoothly to the limit self-shrinker $\Si_{\infty}$, all components of $(\Si_{t}\cap B_{R}(0))\b \mathbf{H}(\ee, \Si_{t}, R)$ with $t\in [t_i-\tau, t_i]$ lie in the $\frac {\ee}2$-neighborhood of  $\Si_{\infty}$. By the definition of $\mathbf{TN}$, for any $t\in [t_i-\tau, t_i]$ the quantity $\mathbf{TN}(\epsilon, \Si_{t}, R) $ is nonempty and we have $|\mathbf{TN}(\epsilon, \Si_{t}, R)|>0.$

\end{proof}

Using Lemma \ref{lma:GD18_1} and Lemma \ref{lma:GD18_3}, we have the following result as in
Lemma 4.8 of \cite{[LW]}.

\begin{lem}\label{lem:A002}Let $R, \ee, \tau>0$ and $f(t, \ee)=\inf_{s\in [t-\tau, t]}|\mathbf{TN}(\Si_s, \ee, R)|$. For any $t_0>0$ and $l>0$, we can find  a sequence $\{t_i\}$ with $t_{i+1}> t_i+l$ such that for any $i\in \NN$,
\beq  \sup_{t\in [t_i, t_i+l]} f( t, \ee)\leq 2f(t_i, \ee).\label{eqn:206} \eeq

\end{lem}
\begin{proof} By Lemma \ref{lma:GD18_3}, we can find $s_1>t_0+l$ with $f(s_1, \ee)>0$. We search for time $t \in [s_1, s_1+l]$ satisfying $f(t, \ee)> 2f(t_i, \ee)$. If no such time exists, then we set
  $t_1=s_1$. Otherwise, we choose such a time and denote it by $s_1^{(1)}$. Then search the time interval $[s_1^{(1)}, s_1^{(1)}+l]$. Inductively, we search
  $[s_1^{(k)}, s_1^{(k)}+l]$. If we have
  $$ \sup_{t \in [s_1^{(k)}, s_1^{(k)}+l]} f(t, \ee)\leq 2 f(s_1^{(k)}, \ee),$$
  then we denote $t_1=s_1^{(k)}$ and stop the searching process. Otherwise, choose a time $s_1^{(k+1)} \in [s_1^{(k)}, s_1^{(k)}+l]$ with more than doubled
   value and continue the process. Note that
   $$ f(s_1^{(k)}, \ee)\geq 2^k f(s_1, \ee)\ri \infty \quad \textrm{as} \quad k \to \infty.$$ Since $\lim_{t\ri +\infty}f(t, \ee)=0$ by Lemma \ref{lma:GD18_1},
   this process must stop in finite steps, and we can find $k_1$ such that
$$ \sup_{t \in [s_1^{(k_1)}, s_1^{(k_1)}+l]} f(t, \ee)\leq 2 f(s_1^{(k_1)}, \ee).$$
We denote by $t_1=s_1^{(k_1)}$. After we find $t_1$, set $s_2^{(0)}=t_1 +l+1$ and continue the previous process to find time in $[s_2^{(0)}, s_2^{(0)}+l]$ such that $f(t, \ee)> 2f(s_2^{(0)}, \ee)$. Similarly, for some $k$ we have
$$\sup_{[s_2^{(k)}, s_2^{(k)}+l]} f(t, \ee)\leq  2f(s_2^{(k)}, \ee).$$
 Then we define $t_2=s_2^{(k)}$. Inductively, after we find $t_l$, we set $s_{l+1}^{(0)}=t_{l}+l+1$. Then we start the process to search time in $[s_{l+1}^{(0)}, s_{l+1}^{(0)}+l]$ with $f(t, \ee)>2f(s_{l+1}^{(0)}, \ee)$.   This process is well defined.
 Repeating this process and we can find a sequence of times $\{t_i\}$ such that for any $t_i$ the inequality (\ref{eqn:206}) holds. The lemma is proved.

\end{proof}

\subsection{Construction of auxiliary functions}\label{sec:construction}

In this subsection, we  construct  functions which will be used to show the $L$-stability of the limit self-shrinker. We fix $R, T>1$ in this section. For any sequence $t_i\ri +\infty$, by Lemma \ref{lem:A001} a subsequence of $\{\Si_{i, t}, -T<t<T\}$ converges in smooth topology to a self-shrinker $\Si_{\infty}$ away from a locally finite, $\si$-Lipschitz singular set $\cS\subset \RR^3\times (-T, T).$ We denote by $\cS_t=\{x\in \RR^3\,|\,(x, t)\in \cS\}$ the singular set in $\RR^3$ at time $t$. By Lemma \ref{lem:multiplicity}, we assume that the multiplicity of the convergence  is a constant $N_0\geq 2$.  As in \cite{[LW]},  we construct some functions as follows:
\begin{enumerate}
  \item[(1).] Let $\ee>0$ and large $R>0$. We define
\beq
\Om_{\ee, R}(t)=(\Si_{\infty}\cap B_{R}(0))\backslash B_{\ee}(\cS_t) \label{eqn:Omega2}
\eeq and for any time interval $I\subset (-T, T)$ we define
\beq
\Om_{\ee, R}(I)=\cap_{t\in I}\Om_{\ee, R}(t),\quad \cS_I=\cup_{t\in I}\cS_t. \label{eqn:defi}\eeq
For any $\ee>0$,  the surface $\Si_{i, t}\cap B_R(0)$ is  a union of graphs over the set $\Om_{ \ee, R}(t)$ for large $t_i$ and  any  $t\in (-T, T)$.
  \item[(2).] Let $u_i^+(x, t)$ and $u_i^-(x, t)$ be the graph functions representing the top and bottom sheets ( which we denote by $\Si_{i, t}^+$ and $\Si_{i, t}^-$ respectively) over $\Si_{\infty}\cap B_R(0)$. The readers are referred to \cite{[LW]} for the details on the construction of $u_i^+(x, t)$ and $u_i^-(x, t)$.
 By the convergence property of the flow $\{(\Si_{i, t}, \x_i(t)), -T< t< T\}$,  for any $\ee>0$ and large $R$  there exists $i_0>0$ such that for any $i\geq i_0$ and any $t\in (-T, T)$ the functions $u_i^+(x, t)$ and $u_i^-(x, t)$ are well-defined on $\Om_{\ee, R}(t)$.
By the calculation  in Appendix \ref{appendixC}, the function
\beq
u_i(x, t)=u_i^+(x, t)-u_i^-(x, t), \label{eq:u}
\eeq which we call the \emph{height difference function} of $\Si_{i, t}$ over $\Si_{\infty}$,  satisfies the equation
\beq
\pd {u_i}t=\Delta_0u_i-\frac 12\langle x, \Na u_i\rangle+|A|^2 u_i+\frac {u_i}2+a_i^{pq}u_{i, pq}+b_i^pu_{i, p}+c_iu_i \label{eq:F008}
\eeq for any $(x, t)\in \Om_{\ee, R}(I)\times I$. Here $\Delta_0$ denotes the Laplacian operator on $\Si_{\infty}$.
 The coefficients $a_i^{pq}, b_i^p$ and $c_i$ are small on $\Om_{\ee, R}(I)\times I$ as $t_i$ large and tend to zero as $t_i\ri +\infty$.
  \item[(3).]
Fix a point $x_0\in (\Si_{\infty}\cap B_R(0))\backslash  \cS_1$. We choose a sequence of points $\{x_i\}_{i=1}^{\infty}\subset (\Si_{\infty}\backslash  \cS_1)\cap B_R(0) $ with $x_i\ri x_0$.
  Then for sufficiently small $\ee>0$  we have $x_0\in\Om_{\ee, R}(1)$ and $\{x_i\}_{i=1}^{\infty}\subset\Om_{\ee, R}(1).$
  For any $t\in (-T, T)$ and $x\in \Om_{\ee, R}(t)$ we define the normalized height difference function
   \beq w_i(x, t)=\frac {u_i(x, t)}{u_i(x_i, 1)},  \label{eq:F001}\eeq
Then $w_i(x, t)$ is a positive function with $w_i(x_i, 1 )=1$ and by (\ref{eq:F008}) $w_i(x, t)$ satisfies the equation on $ \Om_{\ee, R}(I)\times (I)$ for any $I\subset (-T, T)$
\beq
\pd {w_i}t=\Delta_0w_i-\frac 12\langle x, \Na w_i\rangle+|A|^2 w_i+\frac {w_i}2+a_i^{pq}w_{i, pq}+b_i^pw_{i, p}+c_iw_i.  \label{eq:F007}
\eeq
Note that the construction of the function  $w_i $ is slightly different from that of \cite{[LW]}. In (\ref{eq:F001}) we choose a sequence of points $\{x_i\}\subset \Si_{\infty}\b \cS_1$ to normalize the function $u_i$, while in \cite{[LW]} we choose a fixed point $x_0$. The reason why we choose such a  normalization is that we need the inequality (\ref{eq:w7})  in Lemma \ref{claim1} below.

\end{enumerate}

As in \cite{[LW]}, we have the following result which implies that for large $t_i$ the integral of $u_i$ is comparable to the volume $|\mathbf{TN}|$.

\begin{lem}\label{lem:equi} (cf. \cite{[LW]}) Fix $\ee, R$ and $T$ as above.  For any sequence $\{t_i\}$ chosen in Lemma \ref{lem:A001}, there exists $t_T>0$ such that  for any $t\in (-T, T)$ and $t_i>t_T$ we have
\beq
  \frac 12\int_{\Om_{\ee, R}(t)}\, u_i(x, t) \,d\mu_{\infty}\leq  |\mathbf{TN}( \Sigma_{i, t}, \ee, R)|\leq  2 \int_{\Om_{\frac {\ee}5, R}(t)}\, u_i(x, t) \,d\mu_{\infty},\label{eq:J001}
\eeq where $d\mu_{\infty}$ denotes the volume form of $\Si_{\infty}$.

\end{lem}

The proof of Lemma \ref{lem:equi} is similar to that of Lemma 4.13 of \cite{[LW]}. Note that the  coefficients $2$ and $\frac 12$ are chosen to absorb the error term caused by the second fundamental of $\Sigma_{\infty}$.

Since $w_i$ satisfies the parabolic equation (\ref{eq:F007}), we have the following parabolic Harnack inequality by using Theorem \ref{theo:B5} in Appendix \ref{appendixA}.

\begin{lem}\label{lem:harnack}For any $-T<a<s<t<b<T$, any $\ee>0$, $x\in \Om_{\ee, R}(s)$ and $y\in \Om_{\ee, R}(t)$, there exists a constant $C=C(\ee, R,  s-a, t-s, \Si_{\infty}, \cS_{[a, b]})$ such that
\beq
w_i(x, s)\leq Cw_i(y, t).
\eeq
\end{lem}
\begin{proof}We divides the proof into several steps:

\emph{Step 1.} Since $\cS_t\cap B_R(0)$ consists of finitely many points, we can choose sufficiently small $\dd_0(\Si_{\infty}, \cS_{[a, b]})>0$  such that for any $s\in [a, b]$,
\beq
\Om_{2\ee, R}(s)\subset \Om_{\ee, R}(t),\quad \Om_{\frac 12\ee, R+2}(s)\subset \Om_{\frac {1}5\ee, R+2}(t),\quad \forall \; t\in [s-\dd_0, s+\dd_0]\cap [a, b].\label{eqF:001}
\eeq Let $N$ be a positive integer satisfying
\beq
N>\max\Big\{\frac {5(b-a)}{\dd_0}, \;\frac {b-a}{s-a},\; \frac {5(b-a)}{t-s}\Big\}. \label{eq:N}
\eeq Set
\beq
 \tau_k=a+\frac {b-a}{N}k,\quad \forall\; k\in \{0, 1, \cdots, N\}. \label{eqF:002}
\eeq
Then $\tau_0=a$ and $\tau_N=b$. By (\ref{eq:N}) we have $s\geq \tau_1.$
 Note that (\ref{eqF:001}) and (\ref{eqF:002}) imply that for any $k= 1, 2,  \cdots, N-1$ we have
\beq
\Om_{\frac 12\ee, R+2}(\tau_k)\subset \Om_{\frac {\ee}5, R+2}(t),\quad \forall\; t\in [\tau_{k-5}, \tau_{k+5}]\cap [a, b].
\eeq\\

\emph{Step 2.} Let
\beq
 \Om':=\Om_{\ee, R}(\tau_k),\quad \Om'':=\Om_{\frac 23\ee, R+1}(\tau_k),\quad \Om:=\Om_{\frac 12\ee, R+2}(\tau_k).
\eeq Then we have $\Om'\subset \Om''\subset \Om.$ Clearly, $\Om''$ has a positive distance $\dd=\dd(\ee)$ away from the boundary of $\Om.$ Sine $\bar \Om'$ is compact, we can cover $\Om'$ by finite many balls contained in $\Om''$ with radius $r=\frac {\ee}{100}$ and the number of these balls is bounded by a constant depending only on $\ee, R$ and $\Si_{\infty}$. Since $w_i$ satisfies the parabolic equation (\ref{eq:F007}), applying Theorem \ref{theo:B5} in Appendix \ref{appendixA} for the function $w_i$, the domains $\Om', \Om'', \Om$ and the interval $[\tau_{k-1}, \tau_{k+1}]$, we have
\beq
w_i(x, \tau_k)\leq C w_i(y, \tau_{k+1}),\quad \forall\; x, y\in \Om_{\ee, R}(\tau_k),\label{eqG:001}
\eeq where $C=C(\ee, R, b-a, N, \Si_{\infty}, \cS_{[a, b]})$ is a constant independent of $i$. Moreover, since $S_{[a, b]}\cap B_R(0)$ consists of finitely many Lipschitz curves,  there exists a sequence of points $\{z_k\}$ such that
\beq
z_k\in \Om_{2\ee, R}([\tau_{k-1}, \tau_k])\cap \Om_{2\ee, R}([\tau_k, \tau_{k+1}])\neq \emptyset.\label{eq:zk}
\eeq\\

\emph{Step 3.}
For $s, t\in (a, b)$ with $s<t$, there exist integers $k_s$ and $k_t$ such that $s\in [\tau_{k_s}, \tau_{k_s+1})$ and $t\in (\tau_{k_t}, \tau_{k_t+1}]$. Note that (\ref{eq:N}) implies
\beq
t-s\geq \frac {5(b-a)}{N}.\label{eq:L004}
\eeq On the other hand, (\ref{eqF:002}) implies that
\beq
t-s\leq\tau_{k_t+1}-\tau_{k_s}=\frac {b-a}{N}(k_t+1-k_s).\label{eq:L005}
\eeq Combining (\ref{eq:L005}) with (\ref{eq:L004}), we have
\beq
k_t-k_s\geq 4.\label{eq:P004}
\eeq
Thus,  (\ref{eqG:001}), (\ref{eq:P004})and (\ref{eq:zk}) implies that
\beq
w_i(z_{k_s+2}, \tau_{k_s+2})\leq Cw_i(z_{k_s+3}, \tau_{k_s+3})\leq \cdots\leq C^N w_i(z_{k_t-1}, \tau_{k_t-1}),\label{eq:P002y}
\eeq where $C=C(\ee, R, b-a, N, \Si_{\infty}, \cS_{[a, b]})$. \\

\emph{Step 4.}
Set
\beq
\Om'=\Om_{\ee, R}(s),\quad \Om''=\Om_{\frac 23\ee, R+1}(s),\quad \Om=\Om_{\frac 12\ee, R+2}(s).
\eeq
Then by (\ref{eqF:001}) we have
\beq
\Om'\subset \Om''\subset \Om=\Om_{\frac 12\ee, R+2}(s)\subset \Om_{\frac 15\ee, R+2}(s'),\quad \forall\; s'\in [\tau_{k_s-2}, \tau_{k_s+2}],
\eeq where we used the fact that
$[\tau_{k_s-2}, \tau_{k_s+2}]\subset [s-\dd_0, s+\dd_0]. $
Note that by (\ref{eq:zk}) and (\ref{eqF:001}), we have
\beq z_{k_s+2}\in \Om_{2\ee, R}(\tau_{k_s+2})\subset \Om_{\ee, R}(s).\label{eq:zk1}\eeq
As in \emph{Step 2}, $\Om''$ has a positive distance $\dd=\dd(\ee)$ from the boundary of $\Om$, and  we can cover $\Om'$ by finite many balls contained in $\Om''$ with radius $r=\frac {\ee}{100}$ and the number of these balls is bounded by a constant depending only on $\ee, R$ and $\Si_{\infty}$.
Applying Theorem \ref{theo:B5}  for such $\Om', \Om'', \Om$ and the interval $[\tau_{k_s-2}, \tau_{k_s+2}]$ and using (\ref{eq:zk1}),
 we have
\beq
w_i(x, s)\leq Cw_i(z_{k_s+2}, \tau_{k_s+2}),\quad \forall\; x\in \Om_{\ee, R}(s),\label{eq:P001}
\eeq where $C=C(\ee, R, b-a, N, \Si_{\infty}, \cS_{[a, b]})$. \\

\emph{Step 5.} Set
\beq
\Om'=\Om_{\ee, R}(t),\quad \Om''=\Om_{\frac 23\ee, R+1}(t),\quad \Om=\Om_{\frac 12\ee, R+2}(t).
\eeq
Then by (\ref{eqF:001}) we have
\beq
\Om'\subset \Om''\subset \Om=\Om_{\frac 12\ee, R+2}(t)\subset \Om_{\frac 15\ee, R+2}(t' ),\quad \forall\; t'\in [\tau_{k_t-2}, \tau_{k_t+2}],
\eeq where we used the fact that
$[\tau_{k_t-2}, \tau_{k_t+2}]\subset [t-\dd_0, t+\dd_0]. $
Note that by (\ref{eq:zk}) and (\ref{eqF:001}), we have
\beq z_{k_t-1}\in \Om_{2\ee, R}(\tau_{k_t-1})\subset \Om_{\ee, R}(t).\label{eq:zk1y}\eeq
Applying Theorem \ref{theo:B5} as in \emph{Step 4}  for such $\Om', \Om'', \Om$ and the interval $[\tau_{k_t-2}, \tau_{k_t+2}]$ and using (\ref{eq:zk1y}),
 we have
\beq
w_i(z_{k_t-1}, \tau_{k_t-1})\leq C w_i(y, t),\quad \forall\; y\in \Om_{\ee, R}(t), \label{eq:P003}
\eeq
where $C=C(\ee, R, b-a, N, \Si_{\infty}, \cS_{[a, b]})$.
Combining this with (\ref{eq:P001}) and (\ref{eq:P002y}), we have
\beq
w_i(x, s)\leq C w_i(y, t),\quad \forall\; x\in \Om_{\ee, R}(s),\quad  y\in \Om_{\ee, R}(t),
\eeq where $C=C(\ee, R, b-a, N, \Si_{\infty}, \cS_{[a, b]})$.
The lemma is proved.

\end{proof}

For any fixed $\ee, R$ and $ T$, the following result shows that we can find a sequence $\{t_i\}$ such that the  functions $w_i$ are uniformly bounded on a compact set away from singularities. Note that we have   no estimates of $w_i$ near the singularities.

\begin{lem}\label{lem:A004a}Fix $\ee, \tau\in (0, \frac 12)$ and $R, T$ large. Let $\{t_i\}$ be the sequence chosen in Lemma \ref{lem:A002} for such $\ee, \tau, R$ and $l=T$. For any time interval $I=[a, b]\subset [-1, T-2]$ and a compact set $K\subset\subset (\Si_{\infty}\cap B_R(0))\b \cS_I$, there exists a constants $C=C( K, \Si_{\infty},  \cS_{[-2, b+2]})>0$ such that the function $w_i$ defined by (\ref{eq:F001}) satisfies
\beqn
0<w_i(x, t)<C, \quad \forall\; (x, t)\in K\times I.
\eeqn
Moreover, if $a\in [2, T-2])$ there exists $C'=C'(K, \Si_{\infty}, \cS_{[0, a+1]})>0$ independent of $b$ such that \beq
w_i(x, a)\geq C'.
\eeq

\end{lem}

\begin{proof}By the assumption, we can assume that $K\subset \Om_{\ee', R}(I)$ and $\{x_i\}\subset \Om_{\ee', R}(1)$ for some $\ee'\in (0, \ee)$,  where $\{x_i\}$ is the sequence in (\ref{eq:F001}). Note that $w_i(x_i, 1)=1.$ We divide the rest of the proof into several steps.\\

{\it Step 1.
 $w_i$ is bounded on $K\times I$ for the time interval  $I=[-1, \frac 12]$ and any $K$ above. }
Applying Lemma \ref{lem:harnack} for $a=-2$ and $b=2$ we have
\beq
w_i(x, t)\leq C(\ee', R, \Si_{\infty}, \cS_{[-2, 2]})w_i(x_i, 1)=C(\ee', R, \Si_{\infty}, \cS_{[-2, 2]}),\quad \forall\; (x, t)\in K\times I.\label{eq:A200}
\eeq

{\it Step 2.  $w_i$ is  bounded from above on $K\times I$ for any  $I=[a, b]\subset (0, T-2)$ and $K$ above.}
For any $t\in [a, b]\subset (0, T-2)$,  we have $t':=t+1\in (1, T-1)$. By Lemma \ref{lem:equi} and Lemma \ref{lem:A002}  for large $i$ we have
\beqn
\inf_{s\in [t'-\tau, t']}\int_{\Om_{\ee, R}(s)}\, w_i(x, s)\,d\mu&\leq& \frac {2}{u_i(x_i, 1)}\inf_{s\in [t'-\tau, t']}|\mathbf{TN}(\Si_{t_i+s}, \ee, R)|\nonumber\\
&\leq &\frac {4}{u_i(x_i, 1)}\inf_{s\in [-\tau, 0]}|\mathbf{TN}(\Si_{t_i+s}, \ee, R)|\nonumber\\
&\leq &4\inf_{s\in [-\tau, 0]} \int_{ \Om_{\frac {\ee}5, R}(s)}\,w_i(x, s)\,d\mu.\label{eqF:006}
\eeqn
Moreover,  by (\ref{eq:A200}) we have
\beq
w_i(x, 0)\leq C(\ee, R, \Si_{\infty}, \cS_{[-2, 2]}), \quad \forall\; x\in  \Om_{\frac {\ee}5, R}(0),\nonumber
\eeq which implies that
\beq
\int_{\Om_{\frac {\ee}5, R}(0) }\, w_i(x, 0)\leq C(\ee, R, \Si_{\infty}, \cS_{[-2, 2]}). \label{eqF:005}
\eeq
Combining (\ref{eqF:005}) with (\ref{eqF:006}), we have
\beq
\inf_{s\in [t'-\tau, t']}\int_{\Om_{\ee, R}(s)}\, w_i(x, s)\,d\mu\leq C(\ee, R, \Si_{\infty}, \cS_{[-2, 2]}). \label{eqF:007}
\eeq This implies that for any $t\in (1, T-1)$ there exists $s(t)\in [t-\tau, t]$ such that
\beq
\int_{\Om_{\ee, R}(s(t))}\, w_i(x, s(t))\,d\mu\leq C(\ee, R, \Si_{\infty}, \cS_{[-2, 2]}).\label{eqG:003}
\eeq
On the other hand, Lemma \ref{lem:harnack} implies that for any $ x\in K\subset \Om_{\ee', R}([a, b]), t\in [a, b]$, and $y\in \Om_{\ee, R}(s(t+1))\subset \Om_{\ee', R}(s(t+1))$ we have
\beq
w_i(x, t)\leq C(\ee', R, \Si_{\infty}, \cS_{[a-1, b+2]}) w_i(y, s(t+1)),\label{eqF:008}
\eeq where we used the fact that $\tau\in (0, \frac 12)$ and
$$s(t+1)\geq t+1-\tau\geq t+\frac 12. $$
Integrating the right-hand side of (\ref{eqF:008}) and using (\ref{eqG:003}), we have
\beqs
w_i(x, t)&\leq& C(\ee', R, \Si_{\infty}, \cS_{[a-1, b+2]}) \int_{\Om_{\ee, R}(s(t+1))}\,w_i(y, s(t+1))\\&\leq& C(\ee', R, \Si_{\infty}, \cS_{[-2, b+2]}),\quad \forall\; t\in [a, b].
\eeqs

{\it Step 3. $w_i(x, t)$ is bounded from below on $K\times I$ for any  $I=[a, b]\subset [2, T-2]$ and $K$ above.}
 By Lemma \ref{lem:harnack}, for any $(x, t)\in K\times I$ we have
\beq
w_i(x, t)\geq C(\ee', R, \Si_{\infty}, \cS_{[0, t+1]}) w_i(x_i, 1).\label{eq:X001}
\eeq In particular, for $t=a$ the constant in (\ref{eq:X001}) depends only on $\ee', R, \Si_{\infty} $ and $\cS_{[0, a+1]}$.
Thus, the lemma is proved.

\end{proof}

\begin{lem}\label{lem:A004b}The same assumption as in Lemma \ref{lem:A004a}.  As $t_i\ri +\infty$, we can take a subsequence of the functions  $w_i(x, t)$ such that it converges in $C^2$ topology on any compact subset $K\subset\subset (\Si_\infty\cap B_{R}(0))\b \cS_{I}$, where $I=[a, b]\subset [-1, T-2]$,  to a positive function $w(x, t)$ with $w(x_0, 1)=1$ and satisfying
\beq \pd wt=\Delta_0 w+|A|^2 w-\frac 12\langle x, \Na w\rangle+\frac 12 w,\quad \forall\;(x, t)\in K \times I. \label{eq:D004a}\eeq

\end{lem}
\begin{proof} Since $w_i$ is positive by definition and $w_i$ is uniformly bounded from above by Lemma \ref{lem:A004a}, by the interior estimates of the parabolic equation we have the space-time $C^{2, \al}$ estimates of $w_i$(cf. Theorem 4.9 of \cite{[Lieb]}), and the estimates are independent of $i$.
 Therefore, as $i\ri +\infty,$ the function $w_i$ converges to a limit function $ w$ in $C^2$ topology on $K\times [a, b]$ with $w(x_0, 1)=1$ and $w$ is  positive by the strong maximal principle.
   The lemma is proved.

\end{proof}

\subsection{The  auxiliary functions near the singular set  }
In this subsection, we show that there exists a refined sequence such that the limit auxiliary function has uniform  estimates across the singular set. Recall that by Lemma \ref{lem:A004b} the function $w$ is uniformly bounded on any compact set away from the singular set and $w$ has no estimates near the singularities. In this section, we will use Lemma \ref{lem:A002} repeatedly for a sequence  $\{\ee_i\}$ decreasing to zero, and  after taking a diagonal subsequence we can construct a auxiliary function   which has uniform estimates across the singular set.

\begin{lem}\label{lem:item}Let $R>1, \tau\in (0, \frac 12)$, and $\{\ee_i\}$ be a sequence of positive numbers with $\ee_i\ri 0.$ For any $i\in \NN$, there exists a sequence $\{t_{i, k}\}_{k=1}^{\infty}$ with $t_{i, k+1}>t_{i, k}+i$ satisfying the following properties.
\begin{enumerate}
\item[(1).]For any $k\in\NN$ \beq \sup_{s\in [0, \,i]}f(t_{i, k}+s, \ee_i)\leq 2f(t_{i, k}, \ee_i), \label{eq:w5}\eeq
where $f(t, \ee)=\inf_{s\in [t-\tau, t]}|\mathbf{TN}(\Si_s, \ee, R)|$.

\item[(2).]For any $T>0$,   $\{\Si_{t_{i, k}+s}, -T<s<T\}$ converges locally smoothly to a self-shrinker $\Si_{i, \infty}\in \cC( N, \rho)$ away from the space-time singular set $\cS_{i}$ as $k\ri +\infty$.

\item[(3).]For large $k$ the surface $\Si_{t_{i, k}+s}$ can be written as a union of graphs over $\Si_{i, \infty}$ away from the singular set $\cS_{i, s}$.  We denote by $\td u_{i, k}^+(x, s), \td u_{i, k}^-(x, s)$ the graph functions of the top and bottom sheets of $\Si_{t_{i, k}+s}$ over $\td \Om_{i, \ee, R}(s)$, where
\beq
\td \Om_{i, \ee, R}(s)=(\Si_{i, \infty}\cap B_R(0))\b B_{\ee}(\cS_{i, s}). \label{eqn:omega}
\eeq
  Let $\td u_{i, k}(x, s)= \td u_{i, k}^+(x, s)-\td u_{i, k}^-(x, s)$ be the height difference function of $\Si_{t_{i, k}+s}$ over $\td \Om_{i, \ee, R}(s). $ These functions are constructed as in Section \ref{sec:construction}. By Lemma \ref{lem:equi},
  we can choose $k_i$ large such that for any $k\geq k_i$ and $s \in (-T, T)$,
 \beq
  \frac 12 \int_{\td \Om_{i, \ee_i, R}(s)}\,\td u_{i, k}(x, s)\leq\Big|\mathbf{TN}( \Sigma_{t_{i, k}+ s}, \ee_i, R)\Big|\leq 2\int_{\td \Om_{i, \frac {\ee_i}5, R}(s)}\,\td u_{i, k}(x, s).  \label{eq:w001}
 \eeq

\item[(4).] By the smooth compactness of $\cC( N, \rho)$ in  \cite{[CM1]}, we assume that $\Si_{i, \infty}$ in item (2) converges smoothly to $\Si_{ \infty}\in \cC( N, \rho)$.

\item[(5).] For any $i\in \NN$, there exists $k_i>0$ satisfying the following property. For any $\{s_i\}_{i=1}^{\infty}$ with $s_i>k_i$,  $\{\Si_{t_{i, s_i}+s}, -T<s<T\}$ converges locally smoothly to the same self-shrinker $\Si_{\infty}$ as in item (4) away from the space-time singular set $\cS_{\infty}$. Moreover, the singular set $\cS_{i}$ in item (2) converges to $\cS_{\infty}$ in Hausdorff distance.
\end{enumerate}

\end{lem}

\begin{proof}  Applying  Lemma \ref{lem:A002} for $\ee_i$ and $l=i$,  we have (\ref{eq:w5}). Item (2) follows from Lemma \ref{lem:A001}, and item (3) follows from Lemma \ref{lem:equi}. It is clear that item (4) follows from Colding-Minicozzi's compactness theorem \cite{[CM1]}.

To prove item (5), we first note that the convergence in item (2) is also in  Hausdorff distance by Lemma \ref{lem:A001}, for any $i$ there exists   $k_i>0$ such that for any $k\geq k_i$ and $s\in (-2, 2)$ we have
 \beq d_H\Big(\Si_{t_{i, k}+s}\cap B_R(0), \Si_{i, \infty}\cap B_R(0)\Big)\leq \frac 1i, \quad
d_H\Big(\mathbf{S}(\Si_{t_{i, k}+s}, \ee_i, R), \cS_{i, s}\cap B_R(0)\Big)\leq \frac 1{i}, \label{eq:C001}
\eeq where $d_H$ denotes the  Hausdorff distance.  By item (4), we assume that $\Si_{i, \infty}$ converges smoothly to $\Si_{ \infty}\in \cC( N, \rho)$. By Lemma \ref{lem:A001}
 for any sequence of times $\{s_i\}_{i=1}^{\infty}$ with $s_i>k_i$ the surfaces $\{\Si_{t_{i, s_i}+s}, -T<s<T\}$ converge locally smoothly to a self-shrinker, which is denoted by $\hat \Si_{\infty}$,  away from a singular set $ \cS_s\subset \Si_{\infty}$ as $i\ri +\infty$. Moreover,  as $i\ri +\infty$,
\beqs &&d_H\Big(\hat \Si_{\infty}\cap B_R(0), \Si_{\infty}\cap B_R(0)\Big)\\&\leq& d_H\Big(\hat \Si_{\infty}\cap B_R(0), \Si_{t_{i, s_i}+s}\cap B_R(0) \Big)+d_H\Big(\Si_{t_{i, s_i}+s}\cap B_R(0), \Si_{i, \infty}\cap B_R(0)\Big)\\&&+d_H\Big(\Si_{i, \infty}\cap B_R(0), \Si_{\infty}\cap B_R(0)\Big)\\&\leq& d_H\Big(\hat \Si_{\infty}\cap B_R(0), \Si_{t_{i, s_i}+s}\cap B_R(0)\Big)+\frac 1i+ d_H\Big(\Si_{i, \infty}\cap B_R(0), \Si_{\infty}\cap B_R(0)\Big)\\&\ri& 0,
\eeqs where we used (\ref{eq:C001}). Thus, $\hat \Si_{\infty}$ coincides with $\Si_{\infty}$. Moreover,
since $\mathbf{S}(\Si_{t_{i, s_i}+s}, \ee_i, R)$ converges to $\cS_s\cap B_R(0)$ as $i\ri +\infty$,  we have
\beqs &&d_H\Big(\cS_{i, s}\cap B_R(0), \cS_s\cap B_R(0)\Big)\\&\leq& d_H\Big(\mathbf{S}(\Si_{t_{i, s_i}+s}, \ee_i, R), \cS_s\cap B_R(0)\Big)+d_H\Big(\mathbf{S}(\Si_{t_{i, s_i}+s}, \ee_i, R), \cS_{i, s}\cap B_R(0)\Big)\\&\leq& d_H\Big(\mathbf{S}(\Si_{t_{i, s_i}+s}, \ee_i, R), \cS_s\cap B_R(0)\Big)+\frac 1i\\&\ri&  0,
\eeqs where we used (\ref{eq:C001}). Thus, $\cS_{i, s}\cap B_R(0)$ converges to $\cS_{\infty}\cap B_R(0)$ as $i\ri +\infty.$ The lemma is proved.\\

\end{proof}

\begin{lem}\label{claim1}Under the same assumptions as in Lemma \ref{lem:item}, we can choose $x_0\in (\Si_{\infty}\b \cS_1)\cap B_R(0)$ and $\{x_{i, k}\}\subset (\Si_{i, \infty}\b \cS_{i, 1})\cap B_R(0)$ satisfying the following properties.
\begin{enumerate}
\item[(1). ] $x_{i, k}\ri x_0$ as $i\ri +\infty$ and $k\ri +\infty$.

\item[(2). ]  For each $i$,  there exists $k_i>0$   such that for any $k\geq k_i$,
\beq
\td u_{i, k}(x_{i, k}, 1)\leq 2 u_{i, k}(x_0, 1),\quad \forall\;  k\geq k_i. \label{eq:w7}
\eeq
 Here   $ u_{i, k}$ denotes the height difference function of $\Si_{t_{i, k}+s}$ over $\Si_{\infty}$.

\end{enumerate}
\end{lem}
\begin{proof}
Choose $x_0\in (\Si_{\infty}\b \cS_1)\cap B_R(0)$ and we denote by $ l_{x_0}$ the normal line of $\Si_{\infty}$  passing through the point $x_0$. Then the set $\Si_{i, k}\cap l_{x_0}$ is nonempty for large $i$ and $k$. Since $\Si_{i, k}$ can be viewed as a union of multiple graphs over $\Si_{i, \infty}$ away from singularities, we assume that   $l_{x_0}$ intersects with the bottom sheet of $\Si_{i, k}$ at the point $y_{i, k} $, and the projection of $y_{i, k}$ on $\Si_{i, \infty}$ is $x_{i, k}\in \Si_{i, \infty}$.
 We denote by $l_{x_{i, k}}$ the normal line of $\Si_{i, \infty}$ passing through the point $x_{i, k}$.   Since $x_0\not\in \cS_1$, we have $x_{i, k}\not\in \cS_{i, 1}$ for large $i$ and $k$.
 By the construction of $x_{i, k}$, it is clear that $x_{i, k}$ converges to $x_0$ as $i\ri +\infty$ and $k\ri +\infty.$

 Fix $\te_0\in (0, \frac {\pi}2)$.
Since $\Si_{i, \infty}$ converges smoothly to $\Si_{\infty}$,  the angle between the two lines $l_{x_0}$ and $l_{x_{i, k}}$ will lie in $[0, \te_0)$ for  large $i$ and there is a uniform $r_0>0$ independent of $i$ such that
$$|A|(x)\leq \frac 1{r_0}, \quad \forall\; x\in \Si_{i, \infty}\cap B_{r_0}(x_0).  $$
 We assume that $\Si'$ is $\Si_{i, \infty}$, $\Si'_{u_1}$ is the top sheet of $\Si_{i, k}$, $\Si'_{u_2}$ is the bottom sheet of $\Si_{i, k}$ and   $P$ is the point $x_{i, k}$ as above. Then we apply Lemma \ref{lem:Z001} below for such $\Si', \Si'_{u_1}, \Si'_{u_2}$ and the point $P$ and we can get that    the functions $\td u_{i, k}$ and $u_{i, k}$ satisfy (\ref{eq:w7}) for large $k$. The lemma is proved.

\end{proof}

By Lemma \ref{lem:A004b} for each $i$ the function $\td w_{i, k}$ converges in $C^2$ to the limit function $\td w_{i, \infty}$ on any $K\times I$ with $I\subset [-1, T-2]$ and $K\subset\subset (\Si_{i, \infty}\cap B_R(0))\b \cS_{i, I}$,   and $x_{i, k}\ri x_{i, 0}$ as $k\ri +\infty$. Moreover, $\td w_{i, \infty}(x_{i, 0}, 1)=1$. Note that $\td w_{i, \infty}$ satisfies the equation (\ref{eq:D004a}), and the function
$\hat w_i=\td w_{i, \infty}e^{-\frac {|x|^2}{8}}$ satisfies the equation
\beq \pd {\hat w_i}t=\Delta \hat w_i+\Big(|A|^2+\frac 34-\frac 1{16}|x|^2\Big)\hat w_i,\quad \forall\; (x, t)\in K\times I. \label{eq:fi}\eeq
We would like to show that $\hat w_i$ satisfies the parabolic  Harnack inequality with uniform constants independent of $i$. Note that here we need to use Theorem \ref{theo:LYB3} in Appendix \ref{appendixB} instead of Theorem \ref{theo:B5} in Appendix \ref{appendixA}. The reason is that $\hat w_i$ are functions defined on subdomains in $\Si_{i, \infty}$, which varies when $i$ is different. The constants in the Harnack inequality of  Theorem \ref{theo:B5} depend on the manifold $\Si_{i, \infty}$ and it is difficult to show that the constants are independent of $i$. However, we can use Theorem \ref{theo:LYB3} to avoid this difficulty since the constants can be explicitly written down by Theorem \ref{theo:LY}. We note that Theorem \ref{theo:LYB3} cannot be used for the equation (\ref{eq:F007}) of $w_i$ and we have to use Theorem \ref{theo:B5} in the proof of Lemma \ref{lem:harnack}.

\begin{lem}\label{lem:harnack2}Let $\hat w_i=\td w_{i, \infty}e^{-\frac {|x|^2}{8}}$. For any $-T<a<s<t<b<T$, any $\ee>0$, $x\in \Om_{i, \ee, R}(s)$ and $y\in \Om_{i, \ee, R}(t)$, there exists a constant $C=C(\ee, R,  s-a, t-s, \Si_{\infty}, \cS_{[a, b]})$ independent of $i$ such that
\beq
\hat w_i(x, s)\leq C \hat w_i(y, t).
\eeq
\end{lem}
\begin{proof}The lemma follows from the combination of the proof of Lemma \ref{lem:harnack} and Theorem \ref{theo:LYB3}. For the readers' convenience, we give the detailed proof here.

By Lemma \ref{lem:item}, $\cS_i$ converges to $\cS$ in the Hausdorff topology. Since $\cS_t\cap B_R(0)$ consists of finitely many points, we can choose $\dd_0(\Si_{\infty}, \cS_{[a, b]})>0$ small such that for any $s\in [a, b]$,
\beq
\Om_{\frac 32\ee, R}(s)\subset \Om_{\ee, R}(t),\quad \Om_{\frac 13\ee, R+2}(s)\subset \Om_{\frac {1}5\ee, R+2}(t),\quad \forall \; t\in [s-\dd_0, s+\dd_0]\cap [a, b].\label{eqF:001x}
\eeq
Thus, for large $i$ we have
\beq
\Om_{i, 2\ee, R}(s)\subset \Om_{i, \ee, R}(t),\quad \Om_{i, \frac 12\ee, R+2}(s)\subset \Om_{i, \frac {1}5\ee, R+2}(t),\quad \forall \; t\in [s-\dd_0, s+\dd_0]\cap [a, b].\label{eqF:001y}
\eeq
Let $N$ be a positive integer satisfying
\beq
N>\max\Big\{\frac {5(b-a)}{\dd_0}, \;\frac {b-a}{s-a},\; \frac {5(b-a)}{t-s}\Big\}. \label{eq:Nx}
\eeq Set
\beq
 \tau_k=a+\frac {b-a}{N}k,\quad \forall\; k\in \{0, 1, \cdots, N\}. \label{eqF:002x}
\eeq
Then $\tau_0=a$ and $\tau_N=b$. By (\ref{eq:Nx}) we have $s\geq \tau_1.$
 Note that (\ref{eqF:001x}) and (\ref{eqF:002x}) imply that for any $k= 1, 2,  \cdots, N-1$ we have
\beq
\Om_{i, \frac 12\ee, R+2}(\tau_k)\subset \Om_{i, \frac {\ee}5, R+2}(t),\quad \forall\; t\in [\tau_{k-5}, \tau_{k+5}]\cap [a, b].
\eeq
Let
\beq
 \Om'_i:=\Om_{i, \ee, R}(\tau_k),\quad \Om''_i:=\Om_{i, \frac 23\ee, R+1}(\tau_k),\quad \Om_i:=\Om_{i, \frac 12\ee, R+2}(\tau_k).
\eeq Then we have $\Om'_i\subset \Om''_i\subset \Om_i.$ By Lemma \ref{lem:item}, $\Si_{i, \infty}$ converges smoothly to $\Si_{\infty}$, $\cS_i$ converges to $\cS$, the domains $\Om'_i, \Om''_i, \Om_i$ converge to $\Om', \Om'', \Om$ respectively, where
\beq
\Om':=\Om_{ \ee, R}(\tau_k),\quad \Om'':=\Om_{ \frac 23\ee, R+1}(\tau_k),\quad \Om:=\Om_{ \frac 12\ee, R+2}(\tau_k).
\eeq
 Note that $\hat w_i$ satisfies the equation (\ref{eq:fi}), which is exactly the same as the equation (\ref{eq:fib}) in the appendix \ref{appendixB}. Thus, we can apply Theorem \ref{theo:LYB3} in appendix \ref{appendixB} for the function $w_i$, the domains $\Om'_i, \Om''_i, \Om_i$ and the interval $[\tau_{k-1}, \tau_{k+1}]$ to obtain
\beq
\hat w_i(x, \tau_k)\leq C \hat w_i(y, \tau_{k+1}),\quad \forall\; x, y\in \Om_{i, \ee, R}(\tau_k),\label{eqG:001x}
\eeq where $ C=C(\ee, R, b-a, N, \Si_{\infty}, \cS_{[a, b]}) $ is a constant independent of $i$. Moreover, there exists a sequence of points $\{z_k\}$ such that
\beq
z_k\in \Om_{i, 2\ee, R}([\tau_{k-1}, \tau_k])\cap \Om_{i, 2\ee, R}([\tau_k, \tau_{k+1}])\neq \emptyset.\label{eq:zkx}
\eeq
For $s, t\in (a, b)$ with $s<t$, there exist integers $k_s$ and $k_t$ such that $s\in [\tau_{k_s}, \tau_{k_s+1})$ and $t\in (\tau_{k_t}, \tau_{k_t+1}]$. Then we have
\beq
k_t-k_s\geq 4.\label{eq:P004x}
\eeq as in Lemma \ref{lem:harnack}.
Set
\beq
\Om'_i=\Om_{i, \ee, R}(s),\quad \Om''_i=\Om_{i, \frac 23\ee, R+1}(s),\quad \Om=\Om_{i, \frac 12\ee, R+2}(s).
\eeq
Applying Theorem \ref{theo:LYB3} for such $\Om', \Om'', \Om$ and the interval $[\tau_{k_s-2}, \tau_{k_s+2}]$ as in Lemma \ref{lem:harnack},
 we have
\beq
\hat w_i(x, s)\leq C \hat w_i(z_{k_s+2}, \tau_{k_s+2}),\quad \forall\; x\in \Om_{i, \ee, R}(s).\label{eq:P001x}
\eeq
Moreover, (\ref{eqG:001x}) and (\ref{eq:zkx}) implies that
\beq
\hat w_i(z_{k_s+2}, \tau_{k_s+2})\leq C\hat w_i(z_{k_s+3}, \tau_{k_s+3})\leq \cdots\leq C^N \hat w_i(z_{k_t-1}, \tau_{k_t-1}),\label{eq:P002}
\eeq where we used (\ref{eq:P004x}). Similar to the proof of (\ref{eq:P001x}), we have
\beq
\hat w_i(z_{k_t-1}, \tau_{k_t-1})\leq C \hat w_i(y, t),\quad \forall\; y\in \Om_{i, \ee, R}(t). \label{eq:P003x}
\eeq Combining this with (\ref{eq:P001x})-(\ref{eq:P003x}), we have
\beq
\hat w_i(x, s)\leq C \hat w_i(y, t),\quad \forall\; x\in \Om_{i, \ee, R}(s),\quad  y\in \Om_{i, \ee, R}(t).\label{eq:lem320:002}
\eeq The constants $C$ in (\ref{eq:P001x})-(\ref{eq:lem320:002}) depend on $\ee, R, b-a, N, \Si_{\infty}$ and $\cS_{[a, b]}$.
The lemma is proved.

\end{proof}

The next result shows that the normalized height difference function $\td w_{i, k}$ has   uniformly $L^1$ estimate away from the singular set near $t=0$, and the estimate doesn't depend on $i$. The proof of this result relies on the growth estimates of $\td w_{i, \infty}$ near the singular set, which is given in Theorem \ref{theoA:main} in the next section.

\begin{lem}\label{claim2}Fix $\tau\in (0, \frac 12)$.  Under the same assumptions as in Lemma \ref{lem:item}, for each $i$ we can choose $k_i$ sufficiently large  such that for any $k\geq k_i$ the normalized height difference function
\beq
\td w_{i, k}(x, s)=\frac {\td u_{i, k}(x, s)}{\td u_{i, k}(x_{i, k}, 1)}, \label{eq:w}
\eeq where the points $\{x_{i, k}\}$ are chosen as in Lemma \ref{claim1},
  satisfies the inequality
\beq
\inf_{s\in [-\tau, 0]}\int_{\Si_{i, \infty}\cap \td \Om_{i, \frac {\ee_i}5, R}}\,\td w_{i, k}(x, s)\leq 2W_0. \label{eq:w8}
\eeq Here $W_0$ is a constant independent of $i$.
\end{lem}
\begin{proof}
Fix large $R>0.$  Since $\Si_{i, \infty}$ converges to $\Si_{\infty}$ smoothly,  there exist uniform constants $\rho_0, \Xi_0>0$ such that for any large $i$ we have $B_R(0)\cap \Si_{i, \infty}\in \cM_{k_0, 2}(\rho_0, \Xi_0)$. Here the set $\cM_{k_0, 2}(\rho_0, \Xi_0)$ is defined in Definition \ref{defi:003}. Note that by Lemma \ref{lem:A004b} for each $i$ the function $\td w_{i, k}$ converges in $C^2$ to the limit function $\td w_{i, \infty}$ away from $\cS_i$  and $x_{i, k}\ri x_{i, 0}$ as $k\ri +\infty$.
Applying Theorem \ref{theoA:main} to the function
$\hat w_i=\td w_{i, \infty}e^{-\frac {|x|^2}{8}}$, we obtain that
there exist  uniform constants $C=C(\rho_0, \Xi_0, R)$  and $r_1(\rho_0, \Xi_0, R)>0$ such that
\beq
\|\td w_{i, \infty}\|_{L^1((\Si_{i, \infty}\cap B_{R}(0))\times [-\frac 12, 0] )}\leq C(R, \rho_0, \Xi_0)\|\td w_{i, \infty}\|_{L^1(K_i)}, \label{eq:H100}
\eeq where $K_i$ is a compact set defined by
\beq
K_i:=\Big\{(x, t)\in (\Si_{i, \infty}\cap B_{R+1}(0))\times \Big[-1, \frac 12\Big]\;\Big|\;\min_{p\in \cS_{i, t}\cap B_{R+1}(0)}d_{g_i}(x, p)\geq r_1\Big\},
\eeq where $d_{g_i}$ denotes the intrinsic distance function of $(\Si_{i, \infty}, g_i)$. For any $t\in (-T, T)$, we define
$$
K_{i, r}(t)=\Big\{x\in \Si_{i, \infty}\,\Big|\;\min_{p\in \cS_{i, t}\cap B_{R+1}(0)}d_{g_i}(x, p)\geq r\Big\}.
$$
Since $(\Si_{i, \infty}, g_i)$ converges smoothly to $(\Si_{\infty}, g_{\infty})$ and $\cS_i$ converges to $\cS_{\infty}$ by Lemma \ref{lem:item}, for any $t\in (-T, T)$ $K_{i, r}(t)$ converges smoothly to a limit set, which we denote by $K_{\infty, r}(t)\subset \Si_{\infty}$.  By part (5) of Lemma \ref{lem:item}, $K_{\infty, r}(t)\cap \cS_t=\emptyset$. Note that $K_{\infty, r}(t)$ is defined with respect to the metric $g_\infty$ while $\Om_{r, R}(t)$ is with respect to the Euclidean metric in $\RR^3.$  Let
\beq
r_1':=\frac 12\min\Big\{d(x, p)\;\Big|\;x\in K_{\infty, r_1}(t), p\in \cS_{t}, t\in [-2, 2]\Big\}>0,
\eeq where $d(x, p)$ denotes the Euclidean distance in $\RR^3$.
Thus, we have
\beq
K_{\infty, r_1}(t)\subset \Om_{r_1', R+1}(t)\subset \Si_{\infty},\quad \forall\,t\in [-2, 2].
\eeq Since $K_{i, r_1}(t)$ and $\td \Om_{i,  r_1', R+1}(t)$ converge to $K_{\infty, r_1}(t)$ and $ \Om_{  r_1', R+1}(t)$ respectively for each $t$, for large $i$ we have
\beq
K_{i, r_1}(t)\subset \td \Om_{i, \frac 12 r_1', R+2}(t),\quad \forall\; t\in [-2, 2]. \label{eq:Kit}
\eeq
Applying Lemma \ref{lem:harnack2} for $\td \Om_{i, \frac 12 r_1', R+2}(t)$ and $[-2, 2]$, we have
\beq
\td w_{i, \infty}(x, t)\leq  C(r_1', R, \Si_{\infty}, \cS_{[-2, 2]}) \td  w_{i, \infty}(x_{i, 0}, 1)=C(r_1', R, \Si_{\infty}, \cS_{[-2, 2]}),\quad \forall\;(x, t)\in K_{i},\label{eq:H101}
\eeq where we used the fact that $\td w_{i, \infty}(x_{i, 0}, 1)=1$.
 Integrating both sides of (\ref{eq:H101}) on $K_i$, we have
\beqn
\|\td w_{i, \infty}\|_{L^1(K_i)}&\leq& C(r_1', R, \Si_{\infty}, \cS_{[-2, 2]})\Area_{g_i}(\Si_{i, \infty}\cap B_{R+1}(0))\nonumber\\&\leq& C(r_1', R, \Si_{\infty}, \cS_{[-2, 2]}, N),\label{eq:H103y}
\eeqn where we used the upper bound of area ratio in Lemma \ref{lem:A001} in the last inequality.
Combining (\ref{eq:H100}) with (\ref{eq:H103y}), we have
\beq
\|\td w_{i, \infty}\|_{L^1((\Si_{i, \infty}\cap B_{R}(0))\times [-\frac 12, 0] )}\leq C( R, \Si_{\infty}, \cS_{[-2, 2]}, N, \rho_0, \Xi_0).  \label{eq:H102}
\eeq
Thus, the $L^1$ norm of $\td w_{i, \infty}$   is uniformly bounded.
 Since $\td w_{i, k}$ converges to $\td w_{i, \infty}$ on any compact set away from singularities as $k\ri +\infty$, we can choose $k_i$ large such that for any $k\geq k_i$,
\beqs \inf_{s\in [-\tau, 0]}\int_{\Si_{i, \infty}\cap \td \Om_{i, \frac {\ee_i}5, R}}\,\td w_{i, k}(x, s)&\leq&2\inf_{s\in [-\tau, 0]}\int_{\Si_{i, \infty}\cap \td \Om_{i, \frac {\ee_i}5, R}}\,\td w_{i, \infty}(x, s)\\&\leq& \frac 2{\tau} \int_{-\tau}^0dt\int_{\Si_{i, \infty}\cap B_R(0)}\,\td w_{i, \infty}(x, t)\\&\leq&C( R, \Si_{\infty}, \cS_{[-2, 2]}, N, \rho_0, \Xi_0, \tau). \eeqs
where we used the inequality (\ref{eq:H102}). Thus, the inequality (\ref{eq:w8}) is proved.

\end{proof}

Combining Lemma \ref{lem:item}, Lemma \ref{claim1} with Lemma \ref{claim2}, we have the following result.

\begin{lem}\label{lem:A003c}Let $R>0$ and $\tau\in (0, \frac 12)$. There is a sequence of times  $t_i \to \infty$, a self-shrinker  $\Si_{ \infty}\in \cC( N, \rho)$,  a locally finite singular set $\cS$, and a constant $W$ satisfying the following properties.
\begin{enumerate}
  \item[(1).] For any $T>1$, there exists a subsequence $\{t_{i_k}\}_{k=1}^{\infty}$ of $\{t_i\}$ such that $\{\Si_{t_{i_k}+s}, -T<s<T\}$ converges locally smoothly to  $\Si_{ \infty}\in \cC( N, \rho)$ away from $\cS$;
  \item[(2).] Let $x_0\in \Si_{\infty}\b\cS_1$. We define the    functions $u_i$  as in (\ref{eq:u}) and $w_i$ by
\beq
w_i(x, t)=\frac {u_i(x, t)}{u_i(x_0, 1)}.
\eeq
For any $\ee>0$ and large $t_i$,  we have the inequality
\beq
\inf_{s\in [t-\tau, t]}\int_{\Omega_{\ee, R}(s)}\,w_i(x, s)\leq W,\quad \forall\; t\in [2, T),\label{eq:w1}
\eeq where $W$ is a constant independent of $\epsilon$, $i$ and $T$.

\item[(3).]For any  $I=[a, b]\subset [-1, T-2]$ and $K\subset \subset (\Si_{\infty}\cap B_R(0))\b \cS_I$, there exists a constant $C=C(\ee, K, \cS_I, a, b)$  such that
\beqn
0<w_i(x, t)<C, \quad \forall\; (x, t)\in K\times I.\label{eq:X002}
\eeqn
Moreover, if $a\in [2, T-2]$ there exists $C'=C'(K, \Si_{\infty},  \cS_{[0, a+1]})$ independent of $b$ such that \beq
w_i(x, a)\geq C',\quad \forall\; x\in K.\label{eq:X003}
\eeq

\end{enumerate}

\end{lem}

\begin{proof}
 Fix a sequence of $\ee_i\ri 0$. We choose $t_i=t_{i, k_i}$ with $k_i$ large such that Lemma \ref{claim1} and Lemma \ref{claim2} hold. Note that $u_i(x, s)=u_{i, k_i}(x, s)$ is the height difference function of $\Si_{t_{i, k_i}+s}$ over $\Si_{\infty}$.
 Then for any  $T>1$ the sequence $\{\Si_{t_i+s}, -T<s<T\}$ converges locally smoothly to   $\Si_{ \infty}\in \cC( N, \rho)$ away from $\cS$. Note that the limit self-shrinker $\Si_{\infty}$ is independent of the choice of $T$ by Lemma \ref{lem:A001}. For any $\ee>0$, we have $\ee_i\in (0, \ee)$ for large $i$.   Moreover,     for large $t_i$ we have
\beqn
 \inf_{s\in [t-\tau, t]}\int_{\Omega_{\ee, R}(s)}\,w_i(x, s)
&\leq & \frac {2}{u_i(x_0, 1)} \inf_{s\in [t-\tau, t]}\Big|\mathbf{TN}( \Sigma_{t_i+ s}, \ee_i, R)\Big|\nonumber\\
&\leq &\frac {4}{u_i(x_0, 1)} \inf_{s\in [-\tau, 0]}\Big|\mathbf{TN}( \Sigma_{t_i+ s}, \ee_i, R)\Big|,\label{eq:w6}
\eeqn
where we used Lemma \ref{lem:equi} in the first inequality and (\ref{eq:w5}) in the second inequality.
Note that (\ref{eq:w001}) implies that
$$\Big|\mathbf{TN}( \Sigma_{t_i+ s}, \ee_i, R)\Big|\leq  2 \int_{ \td \Om_{i, \frac {\ee_i}5, R}}\,\td u_{i, k_i}(x, s).$$
Thus, we have
\beqn \inf_{s\in [t-\tau, t]}\int_{\Omega_{\ee, R}(s)}\,w_i(x, s)
&\leq&\frac {4\td u_{i, k_i}(x_{i, k_i}, 1) }{u_i(x_0, 1)}\cdot \frac {1}{\td u_{i, k_i}(x_{i, k_i}, 1)} \inf_{s\in [-\tau, 0]}\Big|\mathbf{TN}( \Sigma_{t_i+ s}, \ee_i, R)\Big|\nonumber\\
&\leq &\frac {8\td u_{i, k_i}(x_{i, k_i}, 1) }{u_i(x_0, 1)}\cdot\frac {1}{\td u_{i, k_i}(x_{i, k_i}, 1)}\inf_{t\in [-\tau, 0]}\int_{ \td \Om_{i, \frac {\ee_i}5, R}}\,\td u_{i, k_i}(x, s)\nonumber\\
&\leq & \frac {8\td u_{i, k_i}(x_{i, k_i}, 1) }{u_i(x_0, 1)}\cdot\inf_{s\in [-\tau, 0]}\int_{ \td \Om_{i, \frac {\ee_i}5, R}}\,\td w_{i, k_i}(x, s)\nonumber\\
&\leq& 16W_0\cdot \frac {\td u_{i, k_i}(x_{i, k_i}, 1) }{u_i(x_0, 1)}\leq 32W_0,\label{eq:w002}
\eeqn where we used (\ref{eq:w8}) in the fourth inequality and  (\ref{eq:w7}) in the last inequality. As in the proof of Lemma \ref{lem:A004a}, (\ref{eq:w002}) implies a uniform upper bound of $w_i$ on $K$, and we also have the lower bounds  (\ref{eq:X002})-(\ref{eq:X003}) of $w_i$. The lemma is proved.

\end{proof}

\begin{prop}\label{lem:A003d}Under the same assumption as in Lemma \ref{lem:A003c}, $w_i$ converges in $C^2$ to a positive function $w(x, t)$ satisfying the equation (\ref{eq:D004a}) on $(\Si_{\infty}\times (0, \infty))\b \cS$ with $w(x_0, 1)=1$ and
\beq
\inf_{s\in [t-\tau, t]}\int_{\Si_{\infty}\cap B_R(0)}\, w(x, s)\leq \,W \label{eq:w2},\quad \forall\; t\in [1, \infty).
\eeq Moreover, for any $a\in  [2, \infty) $  there exists a constant $C=C(a,  \Si_{\infty}, \cS_{[1, a+1]}, K)>0$ such that the function $w(x, t)$  satisfies
\beq
\int_{\Si_{\infty}\cap B_R(0)}\,w(x, a)\geq C. \label{eq:w3}
\eeq
\end{prop}
\begin{proof}  For any $I\subset [1, T-2]$ and $K\subset\subset (\Si_{\infty}\cap B_R(0))\b \cS_I$, by Lemma \ref{lem:A003c} and the interior estimates of the parabolic equations (cf. Theorem 4.9 of \cite{[Lieb]}), we have the space-time $C^{2, \al}$ estimates of $w_i$ on $K\times I$.
Taking the limit $i\ri +\infty$, $w_i$ converges in $C^2$ to a limit function $w(x, t)$ on $K\times I$ with the estimate (\ref{eq:X002})-(\ref{eq:X003}). Moreover,  (\ref{eq:w2}) holds on $I$  by (\ref{eq:w1}) and (\ref{eq:w3}) holds on $ I\cap [2, \infty)$ by  (\ref{eq:X003}). Since $\Si_{\infty}$ is independent of the choice of $T$ and the estimates of $w$ are independent of $T$, by taking $T\ri +\infty$ we obtain a function, still denoted by $w$, on $(\Si_{\infty}\times (0, \infty))\b\cS$ with the estimates (\ref{eq:w2})-(\ref{eq:w3}). The proposition is proved.

\end{proof}

The following result was used in the proof  of Lemma \ref{claim1}.

\begin{lem}\label{lem:Z001}Let $\Si\subset\RR^3
$ be a  surface properly embedded in $B_{r_0}(x_0)$ with
\beq |A|(x)\leq \frac 1{r_0}\,\quad x\in B_{r_0}(x_0)\cap \Si. \eeq
  Assume that $\Si_{u_i}$ is the graph of a functions $u_i$ over $\Si$ for $i=1, 2$ and $\Si_{u_1}\cap \Si_{u_2}=\emptyset$.   Let $P\in \Si$, $l_P$ the  normal of $\Si$ at the point $P$, $G=l_P\cap \Si_{u_1}$ and $Q=l_P\cap \Si_{u_2}$.
  For any $\te\in [0, \frac {\pi}2)$, we denote by $l_{\te}$ the line which  passes through $Q$ and has angle $\te$ with the line $l_P$. Let $B=\Si_{u_1}\cap l_{\te}$. Then there are two  constants $\ee\in (0, 1)$ and $\te_0>0$ both depending only on $r_0$  such that if $\te\in (0, \te_0)$ and  \beq \|u_1\|_{C^1(\Si\cap B_{r_0}(x_0))}+\|u_2\|_{C^1(\Si\cap B_{r_0}(x_0))}\leq \ee, \label{eq:R009}\eeq then we have
\beq
 |GQ|\leq 2|BQ|.
\eeq

\end{lem}

\begin{proof}[Proof of Lemma \ref{lem:Z001}] Without loss of generality, we assume that the tangent plane of $\Si$ at $P$ is the plane $\pi:=\{(x_1, x_1, 0)\,|\, x_1, x_2\in \RR\}$ and the point $P$ is the origin $O$ of $\RR^3.$ Let $\hat B_{\dd_0}(0)=\{(x_1, x_2, x_3)\;|\; x_1^2+x_2^2<\dd_0^2\,\}$.  By Lemma \ref{lem:graph1}, there exists $\dd_0=\dd_0(r_0)>0$ such that $\Si\cap \hat B_{\dd_0}(0)$ can be written as a graph of a function $f$ over the plane $\pi,$
\beq \Si\cap \hat B_{\dd_0}(0)=\{(x_1, x_2, f(x_1, x_2))\;|\;|x|<\dd_0\}, \label{eq:S001}\eeq
where $x=(x_1, x_2) $, and the graph function $f$ satisfies \beq
f(0)=0, \quad Df(0)=0,\quad |\Na f|(y)\leq C_0|y|.\label{eq:R003}\eeq
Here $C_0$ depends only on $r_0$.  Note that the coordinates of $G$ and $Q$ are give by
$G=(0, 0, u_1(0))$ and $Q=(0, 0, u_2(0))$
respectively. For the point $B\in\Si_{u_1}\cap l_{\te}$, we define the point $E\in \Si$ the projection of $B$ onto $\Si$, which means
 \beq \overrightarrow{OE}+u_1(E)\n(E)=\overrightarrow{OB}, \label{eq:R004}\eeq
 where $\n(E)$ is the unit normal vector of $\Si$ at $E$.

We claim that there exist  $\ee_0=\ee_0(\dd_0)\in (0, 1)$ and  $\te_0=\arctan 3>0$ such that if $\te\in (0, \te_0)$ and (\ref{eq:R009}) holds for some  $\ee\in (0, \ee_0)$, then  $E\in \Si\cap \hat B_{\dd_0}(0).$ In fact, we assume that $\te_0=\arctan \frac {\dd_0}{4\ee}$. Then for any $\te\in (0, \te_0)$, we have
$|\overrightarrow{OB}|\leq 2\ee+\frac {\dd_0}2$. Combining this with (\ref{eq:R004}) we have
$$ |\overrightarrow{OE}|\leq |\overrightarrow{OB}|+|u_1(E)|\leq 3\ee+\frac {\dd_0}2\leq \frac 34\dd_0, $$
where we choose $\ee\in (0, \frac 1{12}\dd_0).$ Therefore, by (\ref{eq:S001}) we have $E\in \Si\cap \hat B_{\dd_0}(0).$ The claim is proved.

 Assume that $E=(y, f(y))\in \Si\cap \hat B_{\dd}(0)$ with $y=(y_1, y_2)$ and $\dd\in (0, \dd_0)$.   Note that the normal vector at $E$ is give by
$$\n(E)=\frac {(-\partial_{y_1}f(y), -\partial_{y_2}f(y), 1)}{\sqrt{1+|\Na f(y)|^2}},$$ and
by (\ref{eq:R004}) the coordinates of $B=(B_1, B_2, B_3)$ are given by
\beqn
B_1&=&y_1-\frac {u_1(y)\partial_{y_1}f(y)}{\sqrt{1+|\Na f(y)|^2}},\label{eq:R001}\\
B_2&=&y_2-\frac {u_1(y)\partial_{y_2}f(y)}{\sqrt{1+|\Na f(y)|^2}},\label{eq:R002}\\
B_3&=&f(y)+\frac {u_1(y)}{\sqrt{1+|\Na f(y)|^2}},
\eeqn where we write $u_1(y)=u_1(y_1, y_2, f(y_1, y_2))$ for simplicity.
Since $B_1^2+B_2^2=|QB|^2\sin^2 \te$, using (\ref{eq:R001})-(\ref{eq:R002}) we have
$$y_1^2+y_2^2+\frac {u_1(y)^2|\Na f(y)|^2}{1+|\Na f(y)|^2}-2\frac {u_1(y)\langle y, \Na f(y)\rangle}{\sqrt{1+|\Na f(y)|^2}} = |QB|^2\sin^2 \te,$$
where $\langle y, \Na f(y)\rangle=y_1\partial_{y_1}f(y)+ y_2\partial_{y_2}f(y)$. Combining this with (\ref{eq:R003}), we have
\beqs
|QB|^2\sin^2 \te&\geq& y_1^2+y_2^2-2\frac {u_1(y)\langle y, \Na f(y)\rangle}{\sqrt{1+|\Na f(y)|^2}}\\
&\geq&(1-2C_0|u_1(y)|)(y_1^2+y_2^2)\geq (1-2C_0\ee)(y_1^2+y_2^2).
\eeqs Thus, if $\ee\in (0, \frac 1{2C_0})$, we have
\beq |y|^2=  y_1^2+y_2^2\leq \frac {|QB|^2\sin^2\te}{1-2C_0\ee}.\label{eq:R008}
\eeq

Since $l_{\te}$ has the angle $\te$ with the line $l_P$, we assume that the unit direction vector of $l_{\te}$ is $\vec{v}=(v_1, v_2, \cos\te)$. Thus,
we have
\beqn
|QB|&=&|\langle\overrightarrow{QB}, \vec{v}\rangle|=|B_1 v_1+B_2 v_2+(B_3-u_2(0))\cos\te|\nonumber\\
&\geq&|(B_3-u_2(0))\cos\te|-|B_1 v_1|-|B_2 v_2|. \label{eq:R005}
\eeqn
Note that by (\ref{eq:R003})
\beqn
|B_3-u_2(0)|&=&\Big|f(y)+\frac {u_1(y)}{\sqrt{1+|\Na f(y)|^2}}-u_2(0)\Big| \nonumber\\
&\geq&|u_1(0)-u_2(0)|-|u_1(0)-u_1(y)|-|u_1(y)|\cdot \Big|\frac 1{\sqrt{1+|\Na f(y)|^2}}-1\Big|-|f(y)|\nonumber\\
&\geq& |u_1(0)-u_2(0)|-\max_{B_{\dd}(0)}|\Na u_1|\cdot |y|-C_0(1+\max_{B_{\dd}(0)}|u_1|)|y|^2\nonumber\\
&\geq&|u_1(0)-u_2(0)|-C_1|y|, \label{eq:R006}
\eeqn where $C_1=\ee+C_0(1+\ee)\dd_0, $
 and by (\ref{eq:R003}), (\ref{eq:R001})-(\ref{eq:R002}) we have
\beq
|B_1|\leq (1+C_0\ee)|y|,\quad |B_2|\leq (1+C_0\ee)|y|.\label{eq:R007}
\eeq
Combining (\ref{eq:R008})-(\ref{eq:R007}) we have
\beqs
|QB|&\geq& |u_1(0)-u_2(0)|\cos\te-C_1 |y|\\
&\geq&|u_1(0)-u_2(0)|\cos\te-C_1 \sin\te |QB|.
\eeqs  This implies that
\beq \frac {|GQ|}{|QB|}=\frac {|u_1(0)-u_2(0)|}{|QB|}\leq \frac {1+C_1\sin\te}{\cos\te}\leq 2,\nonumber\eeq
if we choose $\te$ sufficiently small.  Thus, the lemma is proved.
\end{proof}

\subsection{The $L$-stability of the limit self-shrinker}

In this subsection, we show that the limit self-shrinker is $L$-stable. The rough idea is similar to that of \cite{[LW]}, but the details  are much more complicated. Compared with \cite{[LW]}, the singularities here no longer move along straight lines, we  cannot choose time large enough such that a given compact set doesn't contain the singularities(cf. Lemma 4.13 of \cite{[LW]}). Therefore, we have to choose a cutoff function near the singularities and analyze the asymptotical behavior of the positive solution near the singular set.
The analysis of the asymptotical behavior is very difficult and we delay the arguments in the next section.

The main result in this subsection is the following lemma.
\begin{lem}   \label{lem:stable1} Fix $ R>1$.  Let $\{t_i\}$ be the sequence of times and $\Si_{ \infty}\in \cC( N, \rho)$ the self-shrinker   in Lemma \ref{lem:A003c}.
  Then we have
\begin{align}
 -\int_{\Si_{\infty}}\,(\psi L\psi)e^{-\frac {|x|^2}{4}} \geq 0,
  \label{eqn:GB24_6}
\end{align} for any smooth  function $\psi\in C_0^{\infty}(\Si_{\infty, R})$.
\end{lem}

Let $w$ be the function obtained in Proposition \ref{lem:A003d} and $v=\log w$. Then $v$ is a function satisfying
$$\pd vt=\Delta_0 v+|A|^2+\frac 12-\frac 12\langle x, \Na v\rangle+|\Na v|^2,\quad \forall\; (x, t)\in (\Si_{\infty}\times (0, \infty))\b \cS.$$
 Let $I=[a, b]\subset (0, \infty)$. We assume that $\phi(x, t)$ is a function satisfying the properties that for any $t\in I$ we have  \beq \phi(\cdot, t)\in W^{1, 2}_0(\Si_{\infty, R}),\quad \overline{\Supp(\phi(\cdot, t))}\cap \cS_t=\emptyset. \label{eq:T001}\eeq Then for any $t\in I$, we have
\beqs
0&=&\int_{\Si_{\infty}}\,\div\Big(\phi^2e^{-\frac {|x|^2}{4}}\Na v\Big)\\
&=&\int_{\Si_{\infty}}\,\Big(2\phi\langle\Na \phi, \Na v\rangle+\Big(\pd vt-\frac 12-|A|^2-|\Na v|^2\Big)\phi^2\Big)e^{-\frac {|x|^2}{4}}\\
&\leq&\int_{\Si_{\infty}}\,\Big(|\Na \phi|^2-\frac 12\phi^2-|A|^2\phi^2+\pd vt\phi^2\Big)
e^{-\frac {|x|^2}{4}}.
\eeqs  This implies that for any $t\in I,$
\beqs
-\int_{\Si_{\infty}}\,(\phi L\phi)e^{-\frac {|x|^2}{4}}&\geq& -\int_{\Si_{\infty}}\,\pd vt\phi^2 e^{-\frac {|x|^2}{4}}\\
&=&-\frac d{dt}\int_{\Si_{\infty}}\,v\phi^2 e^{-\frac {|x|^2}{4}}+\int_{\Si_{\infty}}\,2v\phi\pd {\phi}t e^{-\frac {|x|^2}{4}}.
\eeqs
Integrating both sides with respect to $t\in I$, we have
\beqn &&
-\int_a^b\,\int_{\Si_{\infty}}\,(\phi L\phi)e^{-\frac {|x|^2}{4}}\nonumber\\&\geq&\int_{\Si_{\infty}}\,v\phi^2 e^{-\frac {|x|^2}{4}}\Big|_{t=a}-\int_{\Si_{\infty}}\,v\phi^2 e^{-\frac {|x|^2}{4}}\Big|_{t=b}+\int_a^b\,\int_{\Si_{\infty}}\,2v\phi\pd {\phi}t e^{-\frac {|x|^2}{4}}. \label{eq:B004}
\eeqn

To get the inequality (\ref{eqn:GB24_6}), the main difficulty is to estimate the last term of (\ref{eq:B004}). Using a cutoff function inspired by \cite{[LW]}, we will see that the last term of (\ref{eq:B004}) depends on the asymptotical behavior of $w$ near the singular set.

We now construct the cutoff function near the singular set.
Let $\{\xi_1(t), \xi_2(t), \cdots, \xi_l(t)\}(t\in I)$ be $\si$-Lipschitz curves on $\Si_{\infty}$. We denote by
$$\Ga_k=\{(\xi_k(t), t)\;|\;t\in I\}\subset \Si_{\infty}\times I,\quad \Ga=\cup_{k=1}^l\Ga_k.$$
Choose $0<\dd<\rho<1$. We define  the function on $\RR$
$$ \eta(s)=\left\{
     \begin{array}{ll}
       \frac {\log \rho}{\log |s|}, & \quad 0<|s|<\rho, \\
       1, &  \quad  |s|\geq \rho
     \end{array}
   \right.
$$
and the function $\bb(s)\in C^{\infty}(\RR)$ such that $\bb(s)=0$ for $|s|<\frac {\dd}2$, $\bb(s)=1$ for $|s|\geq \dd$, $0\leq \bb(s)\leq 1$ and $|\Na \bb|\leq \frac 3{\dd}$.
We define the function on $\Si_{\infty}\times I$,
$$f_{\dd, \rho}(x, t)=\prod_{k=1}^l\Big(\eta(\r_k(x, t))\bb(\r_k(x, t))\Big)\in W^{1, 2}((\Si_{\infty}\times I)\b \Ga),$$
where \beq  \r_k(x, t)=d_g(x, \xi_k(t)). \label{eq:rk2} \eeq
For any $\psi(x)\in C_0^{\infty}(\Si_{\infty, R})$, we define
\beq \phi(x,  t)=\psi(x)f_{\dd, \rho}(x, t).\label{eq:D003}\eeq
Then $\phi(x, t)$ satisfies the properties (\ref{eq:T001}).
With loss of generality, we assume that $\sup_{\Si_{\infty}}|\psi|\leq 1.$ Then we have
\begin{lem}\label{lem:B001}For  any small $\ee>0$ we have \beqn -\int_{\Si_{\infty}}\,(\phi L\phi)e^{-\frac {|x|^2}{4}}&\leq& - \int_{\Si_{\infty}} \psi L(\psi) e^{-\frac {|x|^2}{4}}+\Psi(\ee, \rho, \dd\,|\,\Si_{\infty, R}),\label{eq:D001}\eeqn
where  $\Psi$ depends on $\rho, \dd, \ee$ and the geometry of $\Si_{\infty, R}$ and satisfies
\beq \lim_{\ee\ri 0} \lim_{\rho\ri 0}\lim_{\dd\ri0}\Psi(\ee, \rho, \dd \,|\,\Si_{\infty, R})=0.\label{eq:D002}\eeq
\end{lem}
\begin{proof}Since the function $\phi(x, t)=\psi(x)f_{\dd, \rho}(x, t)$ satisfies
$$|\Na \phi|^2\leq (1+\ee)f_{\dd, \rho}^2|\Na \psi|^2+\Big(1+\frac 1{\ee}\Big)\psi^2|\Na f_{\dd, \rho}|^2,$$
we have
\beqn &&
-\int_{\Si_{\infty}}\,\phi L(\phi) e^{-\frac {|x|^2}{4}}=\int_{\Si_{\infty}}\,\Big(|\Na \phi|^2-(\frac 12+|A|^2)\phi^2\Big)e^{-\frac {|x|^2}{4}}\nonumber\\
&\leq&\int_{\Si_{\infty}}\,\Big(|\Na \psi|^2-(\frac 12+|A|^2)\psi^2\Big)e^{-\frac {|x|^2}{4}}+\int_{\Si_{\infty}}\,\Big( (1+\ee)f_{\dd, \rho}^2-1\Big)|\Na \psi|^2e^{-\frac {|x|^2}{4}}\nonumber\\
&&+\int_{\Si_{\infty}}\,\Big(\frac 12+|A|^2\Big)(1-f_{\dd, \rho}^2)\psi^2 \,e^{-\frac {|x|^2}{4}}+\Big(1+\frac 1{\ee}\Big)\int_{\Si_{\infty}}\,\psi^2|\Na f_{\dd, \rho}|^2\,e^{-\frac {|x|^2}{4}}\nonumber\\
&:=&I_0+I_1+I_2+I_3. \label{eq:G002}
\eeqn
Note that $|f_{\dd, \rho}|\leq 1$ and $\lim_{\rho\ri 0}\lim_{\dd\ri 0}f_{\dd, \rho}(x, t)=1$ for any $(x, t)\in (\Si\times I)\b \Ga$.  The Lebesgue dominated convergence theorem implies that
\beq
\lim_{\ee\ri 0}\lim_{\rho\ri 0}\lim_{\dd\ri 0}I_1=0,\quad \lim_{\rho\ri 0}\lim_{\dd\ri 0}I_2=0.\label{eq:G003}
\eeq
We next estimate $I_3$. Let $f_k(x, t)=\eta(\r_k(x, t))\bb(\r_k(x, t))$.
We define
$$\Xi_R:=\inf\Big\{\Xi>0\,\Big|\, \frac 1{\Xi} \dd_{ij}\leq g_{ij}(x)\leq \Xi \dd_{ij}, \forall\, x\in B_{R}(0)\cap \Si_{\infty} \Big\},$$
where $g_{ij}$ is the induced metric on $\Si_{\infty}$.
Note that
\beq
\int_{\Si_{\infty}}\, |\Na f_k|^2\,e^{-\frac {|x|^2}{4}}\leq
2\int_{\Si_{\infty}}\, \Big(\bb^2|\Na \eta|^2+\eta^2|\Na \bb|^2\Big)\,e^{-\frac {|x|^2}{4}}. \label{eqAA:001}
\eeq
We estimate
\beqn
\int_{\Si_{\infty}}\,\bb(\r_k)^2|\Na (\eta(\r_k))|^2\,e^{-\frac {|x|^2}{4}}&\leq&\int_{\frac {\dd}2\leq \r_k\leq \rho}\,(\eta'(\r_k))^2\leq C\int_{\frac {\dd}2}^{\rho}\,\frac {(\log \rho)^2}{s(\log s)^4}\,ds\nonumber\\
&\leq&C\Big(\frac 1{|\log \rho|}+\frac {(\log \rho)^2}{|\log \frac {\dd}2|^3}\Big), \label{eqAA:002}
\eeqn where $C$ is a constant depending on the metric $g$. Moreover,
\beqn
\int_{\Si_{\infty}}\, \eta^2|\Na \bb|^2\,e^{-\frac {|x|^2}{4}}&\leq&
\int_{\frac {\dd}2\leq \r_k\leq \dd}\,\eta(\r_k)^2(\bb'(\r_k))^2|\Na \r_k|^2\nonumber \\
&\leq&C\int_{\frac {\dd}2}^{\dd}\,\frac {(\log \rho)^2}{(\log s)^2}\cdot \frac 4{\dd^2}\cdot s\,ds\nonumber \\
&\leq &C\int_{\frac {\dd}2}^{\dd}\,\frac {(\log \rho)^2}{s(\log s)^2}\,ds\leq C\frac {(\log \rho)^2}{|\log \dd|}, \label{eqAA:003}
\eeqn where $C$ is a constant depending on the metric $g$.
Combining (\ref{eqAA:002})(\ref{eqAA:002}) with (\ref{eqAA:001}), we have
\beqs
\int_{\Si_{\infty}}\, |\Na f_k|^2\,e^{-\frac {|x|^2}{4}}&\leq&
2\int_{\Si_{\infty}}\, \Big(\bb^2|\Na \eta|^2+\eta^2|\Na \bb|^2\Big)\,e^{-\frac {|x|^2}{4}}\\&\leq&  C \Big(\frac {1}{|\log \rho|}+ \frac {|\log \rho|^2}{|\log \dd|}\Big).\nonumber
\eeqs
Since $|\psi|\leq 1$ and $|f_k|\leq 1$,  we have
\beqs
\int_{\Si_{\infty}}\,\psi^2|\Na f_{\dd, \rho}|^2\,e^{-\frac {|x|^2}{4}}&\leq &l\int_{\Si_{\infty}}\,\sum_{k=1}^l\,|\Na f_k|^2\,e^{-\frac {|x|^2}{4}}\\
&\leq &C(l, g)\Big(\frac 1{|\log \rho|}+ \frac {|\log \rho|^2}{|\log \dd|}\Big).
\eeqs
Therefore, we have
\beq  \lim_{\rho\ri 0}\lim_{\dd\ri 0}I_3=0. \label{eq:G004}\eeq
Combining (\ref{eq:G003}) (\ref{eq:G004}) with (\ref{eq:G002}), we have (\ref{eq:D001}) and (\ref{eq:D002}).

\end{proof}

\begin{lem}\label{lem:B003} For the function $\phi$ defined by (\ref{eq:D003}), we have
\beqn
\int_{\Si_{\infty}}\,2v\phi\pd {\phi}t e^{-\frac {|x|^2}{4}}&\geq& -2\si (\log \rho)^2\sum_{k=1}^l\,\Big(\int_{A^{(k)}_t(\frac {\dd}2, \dd)\cap F^{(k)}_t(\frac {\dd}2)}\,\frac {3|v|}{\dd |\log \r_k|^2}\nonumber\\&&
+\int_{A^{(k)}_t(\frac {\dd}2, \rho)\cap F^{(k)}_t(\frac {\dd}2)}\,\frac {|v|}{\r_k|\log \r_k|^3}
\Big), \label{eq:D013}
\eeqn where $F^{(k)}_t(\dd)$ and $A^{(k)}_t(\dd, \rho)$ are defined by
\beqn
F^{(k)}_t(\dd)&=&\bigcap_{i\neq k}\{x\in \Si_{\infty}\,|\, \r_i(x, t)\geq \dd\},\label{eq:G008}\\
A^{(k)}_t(\dd, \rho)&=&\{x\in \Si_{\infty}\,|\, \dd<\r_k(x, t)<\rho\}.\label{eq:G009}
\eeqn

\end{lem}
\begin{proof}We use the same notations as in the proof of Lemma \ref{lem:B001}.
Direct calculation shows that
\beqs &&
\Big|\int_{\Si_{\infty}}\,2v\phi\pd {\phi}t e^{-\frac {|x|^2}{4}}\Big|\\&\leq& \sum_{k=1}^l\,\int_{\Si_{\infty}\cap F^{(k)}_t(\frac {\dd}2)}\,2|v| f_k\Big|\pd {f_k}t\Big| e^{-\frac {|x|^2}{4}} \\
&\leq&
\sum_{k=1}^l\,\int_{\Si_{\infty}\cap F^{(k)}_t(\frac {\dd}2)}\,2|v|\eta(\r_k)\bb(\r_k)\Big(|\eta'(\r_k)|\bb(\r_k)
+|\bb'(\r_k)|\eta(\r_k)\Big)\Big|\pd {\r_k}t\Big|\,e^{-\frac {|x|^2}{4}}. \label{eq:D010}
\eeqs
Note that for {a.e.} $t\in I$
$\Big|\pd {\r_k}t\Big|\leq |\xi_k'(t)|\leq \si $, and  we assumed that $\sup_{\Si_{\infty}}|\psi|\leq 1.$ Therefore, using the definition of $\eta$ and $\bb$  we have
\beqn  &&\int_{\Si_{\infty}\cap F^{(k)}_t(\frac {\dd}2)}\,2|v|\eta(\r_k)\bb^2(\r_k)|\eta'(\r_k)|\cdot \Big|\pd {\r_k}t\Big|\,e^{-\frac {|x|^2}{4}} \nonumber\\&\leq &2\si  \int_{\Si_{\infty}\cap F^{(k)}_t(\frac {\dd}2)}\,
\eta(\r_k)\bb(\r_k)^2|v|\cdot |\eta'(\r_k)|\nonumber\\&\leq&2\si \int_{A^{(k)}_t (\dd, \rho)\cap F^{(k)}_t(\frac {\dd}2)}\,|v|\frac {(\log \rho)^2}{|\log \r_k|^3\r_k} \label{eq:D011}
\eeqn and
\beqn &&
 \int_{\Si_{\infty}\cap F^{(k)}_t(\frac {\dd}2)}\,2|v|\eta(\r_k)^2\bb(\r_k)  |\bb'|\cdot \Big| \pd {\r_k}t\Big|\,e^{-\frac {|x|^2}{4}}\nonumber \\&\leq& 2\si\int_{\Si_{\infty}\cap F^{(k)}_t(\frac {\dd}2)}\,|v|\eta^2(\r_k)\bb(\r_k)|\bb'(\r_k)|\nonumber\\&\leq& \frac {6\si}{\dd}\,\int_{A^{(k)}_t (\frac {\dd}2, \dd)\cap F^{(k)}_t(\frac {\dd}2)}\,|v|\Big(\frac {\log \rho}{\log \r_k}\Big)^2. \label{eq:D012}
\eeqn Combining (\ref{eq:D010})-(\ref{eq:D012}), we have (\ref{eq:D013}).

\end{proof}

\begin{lem}\label{lem:B002}For any $t>0$,  we have
\beqn \lim_{\rho\ri 0}\lim_{\dd\ri 0}\,(\log \rho)^2\int_a^b\,dt\int_{ A^{(k)}_t(\dd, \rho)\cap F^{(k)}_t(\frac {\dd}2)}\,\frac {|v|}{\r_k|\log \r_k|^3}&=&0,\label{eq:E003}\\ \lim_{\rho\ri 0}\lim_{\dd\ri 0}\frac {(\log \rho)^2}{\dd}\int_a^b\,dt\int_{A^{(k)}_t (\frac {\dd}2, \dd)\cap F^{(k)}_t(\frac {\dd}2)}\,\frac {|v|}{|\log \r_k|^2}&=&0. \label{eq:E004}\eeqn

\end{lem}

\begin{proof}Since $w(x, t)$ satisfies (\ref{eq:D004a}) away from the singular set,  the function $f(x, t)=w(x, t)e^{-\frac {|x|^2}{8}}$ satisfies the equation
$$\pd ft=\Delta f+\Big(|A|^2+\frac 34-\frac 1{16}|x|^2\Big)f.$$
 By Theorem \ref{theoA:main}, we have
 \beqn
  \lim_{\rho\ri 0}\lim_{\dd\ri 0}\int_{a}^{b}\,dt\int_{ A^{(k)}_t (\dd, \rho)\cap F^{(k)}_t(\frac {\dd}2)}\,\frac {f}{\r_k|\log \r_k|}\,&=&0,\label{eq:E001}\\
\lim_{\dd\ri 0}\frac 1{\dd}\int_{a}^{b}\,dt\int_{A^{(k)}_t (\frac {\dd}2, \dd)\cap F^{(k)}_t(\frac {\dd}2)}\,f\,&<&+\infty. \label{eq:E002}
  \eeqn
Since near the singular curve $\xi(t)$,  the function $w$ is large and we have $v=\log w\leq w. $ Thus, (\ref{eq:E001})-(\ref{eq:E002}) also hold for $v$, and this directly implies (\ref{eq:E003})-(\ref{eq:E004}). The lemma is proved.

\end{proof}

Combining the above results, we can show Lemma \ref{lem:stable1}.

\begin{proof}[Proof of Lemma \ref{lem:stable1}]

Combining Lemma \ref{lem:B001}, Lemma \ref{lem:B003} with the inequality (\ref{eq:B004}), we have
\beqn &&
- (b-a)\int_{\Si_{\infty}} \psi L(\psi) e^{-\frac {|x|^2}{4}} +\Psi(\ee, \rho, \dd \,| \, \Si_{\infty, R})  (b-a)\nonumber\\&\geq& \int_{\Si_{\infty}}\,v\phi^2 \,e^{-\frac {|x|^2}{4}}\Big|_{t=a}-\int_{\Si_{\infty}}\,v\phi^2 \,e^{-\frac {|x|^2}{4}}\Big|_{t=b}\nonumber\\
&&-2\si (\log \rho)^2\sum_{k=1}^l\,\Big(\int_a^b\,\int_{A^{(k)}_t (\frac {\dd}2, \dd)\cap F^{(k)}_t(\frac {\dd}2)}\,\frac {3|v|}{\dd |\log \r_k|^2}\nonumber\\&&
+\int_a^b\,\int_{A^{(k)}_t (\dd, \rho)\cap F^{(k)}_t(\frac {\dd}2)}\,\frac {|v|}{\r_k|\log \r_k|^3}
\Big).\nonumber\\\label{eq:B005}
\eeqn
Taking  $\dd\ri 0$ and next $\rho\ri 0$, and then $\ee\ri 0$ in (\ref{eq:B005}), we get
\beq -\int_{\Si_{\infty}} \psi L(\psi) e^{-\frac {|x|^2}{4}}\geq \frac 1{b-a}\Big(\int_{\Si_{\infty}}\,v\psi^2 \,e^{-\frac {|x|^2}{4}}\Big|_{t=a}-\int_{\Si_{\infty}}\,v\psi^2 \,e^{-\frac {|x|^2}{4}}\Big|_{t=b}\Big). \label{eq:E005}\eeq
Note that by Propostion \ref{lem:A003d} $w(x, t)$ is a function on $(\Si_{\infty}\times (0, \infty))\b \cS$ with uniform estimates (\ref{eq:w2}) and (\ref{eq:w3}). Thus, there is a sequence $b_i\ri +\infty$ such that
$$\int_{\Si_{\infty}}\,v\psi^2 \,e^{-\frac {|x|^2}{4}}\,d\mu_{\infty}\Big|_{t=b_i}\leq \int_{\Si_{\infty}}\,w\,d\mu_{\infty} \Big|_{t=b_i}\leq W$$
for a constant $W$.
Therefore, by taking $b_i\ri +\infty$ and $a=2$ in (\ref{eq:E005}) we get (\ref{eqn:GB24_6}). The lemma is proved.

\end{proof}

\subsection{Proof of Theorem \ref{theo:removable}}

In this subsection, we show Theorem \ref{theo:removable}. First, using Lemma \ref{lem:rx} and the compactness result of Colding-Minicozzi \cite{[CM1]}   we have the following result.

\begin{lem}\label{lem:G1}Let $R,  N>0$ and $\rho $ an increasing positive function.
For any $\dd>0$, there exists a constant $\xi=\xi(R,  N, \rho, \dd)>0$  such that for any $\Si\in \cC( N, \rho)$ and any $r\in (0, \xi]$ we have
\beq
1-\dd\leq \frac {\Area(\Si\cap B_{r}(x))}{\pi r^2}\leq 1+\dd,\quad \forall \; x\in B_R(0)\cap \Si.\label{lem:G1:001}
\eeq
\end{lem}
\begin{proof}We show that there exists a constant $C_R=C(R,  N, \rho)>0$ such that for all $\Si\in \cC( N, \rho)$ we have
$ \sup_{\Si\cap B_{R+1}(0)}|A|\leq C_R$.
For otherwise, we can find a sequence $\Si_i\in \cC( N, \rho)$ such that \beq \sup_{\Si_i\cap B_{R+1}(0)}|A|\ri +\infty. \label{lem:G1:002}\eeq  On the other hand, by the compactness theorem of Colding-Minicozzi \cite{[CM1]}, $\Si_i$ converges smoothly to $\Si_{\infty}\in \cC( N, \rho)$, which has bounded $|A|$ on any compact set. This contradicts (\ref{lem:G1:002}).

Since $\sup_{\Si\cap B_{R+1}(0)}|A|$ is uniformly bounded by $C_R$, (\ref{lem:G1:001}) follows directly from  Lemma \ref{lem:ratio}. The lemma is proved.

\end{proof}

Using the uniform upper bound of the area ratio and Lemma \ref{lem:A001}, we have the following result.

\begin{lem}\label{lem:area}Under the same assumption as in Lemma \ref{lem:A001}, if $\{\Si_{t_i}\}$ converges locally smoothly to $\Si_{\infty}$ with multiplicity $m\in \NN$ away from a locally finite singular set $\cS_0$, then for any $x_i\in \Si_{t_i}\cap B_R(0)$ with $x_i\ri x_{\infty}\in \Si_{\infty}\cap B_{R+1}(0)$ and $r>0$ we have
\beq
\lim_{i\ri+\infty}\Area(\Si_{t_i}\cap B_r(x_i))=m\cdot\Area(\Si_{\infty}\cap B_r(x_\infty)).\label{eq:F006}
\eeq

\end{lem}
\begin{proof} Since $\cS_0$ is locally finite, without loss of generality we assume that $B_r(x_{\infty})\cap \Si_{\infty}$ consists of only one singular point $y_{\infty}$. For any $\ee>0$ by the smooth convergence of $\Si_{t_i}\cap (B_r(x_i)\b B_{\ee}(y_{\infty}))$ we have
\beq
\lim_{i\ri+\infty}\Area\Big(\Si_{t_i}\cap (B_r(x_i)\b B_{\ee}(y_{\infty}))\Big)=m\cdot\Area\Big(\Si_{\infty}\cap (B_r(x_\infty)\b B_{\ee}(y_{\infty}))\Big).\label{eq:F005}
\eeq Since the area ratio is uniformly bounded from above along the rescaled mean curvature flow, we have
$$\Area(\Si_{t_i}\cap B_{\ee}(y_{\infty}))\leq N \pi \ee^2\ri 0$$
as $\ee\ri 0.$ Taking $\ee\ri 0$ in both sides of (\ref{eq:F005}), we have (\ref{eq:F006}). The lemma is proved.

\end{proof}

Combining Lemma \ref{lem:G1}, Lemma \ref{lem:area} with Lemma \ref{lem:multiplicity}, we show that the area ratio is always close to an integer after a fixed time.

\begin{lem}\label{lem:G2}Fix large $R$ and small $\dd_0\in (0, \frac 12)$.  Under the same assumption as in Lemma \ref{lem:A001}, there exists $t_0>0$ such that for any $t>t_0$ we have
\beq
m(1-2\dd_0)< \frac {\Area(\Si_t\cap B_{\xi}(x))}{\pi \xi^2}< m(1+2\dd_0),\quad \forall \; x\in B_R(0)\cap \Si_t, \label{lem:G2:005}
\eeq where $m$ is a positive integer independent of $x$ and $t$. Here $\xi=\xi(R+1,  N, \rho, \dd_0)$ is the constants in Lemma \ref{lem:G1} with  $ N$ and $ \rho$ determined as in the assumption (\ref{eq:BB002}) and  Lemma \ref{lem:A001}.
\end{lem}
\begin{proof} We divide the proof into several steps.

\emph{Step 1.} We show that  there exists $t_0>0$ such that for any $t>t_0$ (\ref{lem:G2:005}) holds for some integer $m(x, t)$, which may depend on $x$ and $t$. For otherwise, there exist a sequence $t_i\ri+\infty$ and $x_i\in B_R(0)\cap \Si_{t_i}$ such that
\beq
\Big|\frac {\Area(\Si_{t_i}\cap B_{\xi}(x_i))}{\pi m \xi^2}-1\Big|\geq 2\dd_0,\quad \forall \;m\in \NN\cap [1, N_0]. \label{lem:G2:003}
\eeq By Proposition \ref{prop:weakcompactness}, by taking a subsequence if necessary  we assume that $\Si_{t_i}$ converges locally smoothly to a self-shrinker $\Si_{\infty}\in \cC( N, \rho)$ with multiplicity $m_0\in \NN$ and $x_i\ri x_{\infty}\in \Si_{\infty}\cap B_{R+1}(0). $ By the convergence of $\{\Si_{t_i}\}$ and Lemma \ref{lem:area}, we have
\beq
\lim_{i\ri +\infty}\frac {\Area(\Si_{t_i}\cap B_{\xi}(x_i))}{\pi \xi^2}=m_0 \frac {\Area(\Si_{\infty}\cap B_{\xi}(x_\infty))}{\pi  \xi^2}. \label{lem:G2:001}
\eeq Lemma \ref{lem:G1} implies that
\beq
1-\dd_0\leq \frac {\Area(\Si_{\infty}\cap B_{\xi}(x_\infty))}{\pi  \xi^2}\leq 1+\dd_0.\label{lem:G2:002}
\eeq
Combining (\ref{lem:G2:001}) with  (\ref{lem:G2:002}), for large $t_i$ we have
\beq
\Big|\frac {\Area(\Si_{t_i}\cap B_{\xi}(x_i))}{\pi m_0 \xi^2}-1\Big|\leq \frac 32\dd_0, \label{lem:G2:004}
\eeq which contradicts (\ref{lem:G2:003}).\\

\emph{Step 2.} We show that $m(x, t)$ is independent of $x$ and we can write $m(t)$ for short.  For otherwise, we can find a sequence $t_i\ri +\infty$ and $x_i, y_i\in \Si_{t_i}$ with $m(x_i, t_i)\neq m(y_i, t).$ Since $m(x, t)\in [1, N_0]$, by taking a subsequence if necessary we assume that $m(x_i, t_i)=m_1$ for all $i$. Thus, for any $i$ we have
\beq m(y_i, t_i)\neq m_1. \label{lem:G2:006}\eeq By Proposition \ref{prop:weakcompactness}, by taking a subsequence if necessary  we assume that $\Si_{t_i}$ converges locally smoothly to a self-shrinker $\Si_{\infty}\in \cC( N, \rho)$ with multiplicity $m_0\in \NN$, and
$$x_i\ri x_{\infty}, \quad y_i\ri y_{\infty},\quad x_{\infty}, y_{\infty}\in \Si_{\infty}\cap B_{R+1}(0). $$
By (\ref{lem:G2:004}), we have $m(x_i, t_i)=m_0=m(y_i, t_i)$, which contradicts (\ref{lem:G2:006}). \\

\emph{Step 3.} We show that $m(t)$ is independent of $t$. It suffices to show that for any $s\in (-\frac 12, \frac 12)$, we have $m(t)=m(t+s)$. For otherwise, we can find a sequence $t_i\ri +\infty$ and $s_i\in (-\frac 12, \frac 12)$ such that for all $i$,
\beq
m(t_i)\neq m(t_i+s_i). \label{lem:G2:007}
\eeq We follow the same argument as in Step 2.  Since $m(t_i)$ is uniformly bounded, by taking a subsequence if necessary we can assume that $m(t_i)=m_1$ for all $i$. By (\ref{lem:G2:007}), for all $i$ we have
\beq
m(t_i+s_i)\neq m_1. \label{lem:G2:008}
\eeq
Note that $m(t_i+s_i)$ is also bounded, we can assume that a subsequence of $\{m(t_i+s_i)\}$ converges to an integer $m_2$ with \beq m_2\neq m_1 \label{lem:G2:010}\eeq by (\ref{lem:G2:008}).
By Proposition \ref{prop:weakcompactness}, by taking a subsequence if necessary  we assume that $\{\Si_{t_i+s}, -1<s<1\}$ converges locally smoothly to a self-shrinker $\Si_{\infty}\in \cC( N, \rho)$ with multiplicity $m_0\in \NN$. The inequality (\ref{lem:G2:004}) implies that $m_0=m_1$. Since  the multiplicity $m_0$ of the  convergence is independent of time by Lemma \ref{lem:multiplicity},  we have $m_0=m_2$. Thus, we have $m_1=m_2$,  which contradicts (\ref{lem:G2:010}).

\end{proof}

Using Lemma \ref{lem:G2} and the results in previous sections, we show Theorem \ref{theo:removable}.

\begin{proof}[Proof of Theorem \ref{theo:removable}]Fix large $R>R_0$, where $R_0$ is the constant chosen in Lemma \ref{lem:A002b}.  We  choose a sequence $t_i\ri +\infty$ as in Lemma \ref{lem:A003c}.  Then there is a self-shrinker $\Si_{\infty}\in \cC( N, \rho)$    such that for any $T>1$ we can find a subsequence, still denoted by $\{t_i\}$, such that   $\{\Si_{t_i+t}, -T<t<T\}$ converges in smooth topology, possibly with multiplicities at most $N_0$, to $\Si_{\infty}$ away from a  singular set
$\cS$.  If the multiplicity of the convergence is greater than one,  Lemma \ref{lem:stable1} shows that the limit self-shrinker $\Si_{\infty}$ is $L$-stable in  the ball $B_{R}(0)$. This contradicts Lemma \ref{lem:A002b}. Therefore, the multiplicity is one and the convergence is smooth.

We next show that for any sequence of $s_i\ri +\infty$ there exists a subsequence such that  the multiplicity of the convergence is also one. For otherwise, there exists a sequence $s_i\ri +\infty$ such that $\Si_{s_i}$ converges locally smoothly to a self-shrinker $\Si_{\infty}'\in \cC( N, \rho)$ with multiplicity  $m'>1$. By Lemma \ref{lem:G2}, there exists $t_0>0$ such that for any $t>t_0$ we have
\beq
m(1-2\dd_0)< \frac {\Area(\Si_t\cap B_{\xi}(x))}{\pi \xi^2}< m(1+2\dd_0),\quad \forall \; x\in B_R(0)\cap \Si_t, \label{lem:G3:001}
\eeq where $m$ is a positive integer independent of $x$ and $t$. By taking $t=t_i\ri +\infty$ in (\ref{lem:G3:001}), we have $m=1$. On the other hand, taking $t=s_i\ri +\infty$ in (\ref{lem:G3:001}), we have $m=m'>1$, which is a contradiction. Thus, the theorem is proved.

\end{proof}

\section{Estimates  near the singular set}\label{sec:analysis}
In this section, we will study the asymptotical behavior of the function $w$ near the singular set. These estimates are used in the proof of Lemma \ref{claim2}  and Lemma \ref{lem:B002}. In \cite{[KT]}, Kan-Takahashi studied time-dependent singularities in semilinear parabolic equations along  one singular curve. Here we develop Kan-Takahashi's techniques to estimate the solution when the singular sets consists of  multiple singular curves.

First, we introduce the following notations. Throughout this section, we denote by $\cB_r(p)$ the (intrinsic) geodesic ball centered at $p$ in $(M, g)$ and $d_g(x, y)$ the distance from $x$ to $y$ with respect to the metric $g$.

\begin{defi}\label{defi:003} Let $(M, g)$ be a complete Riemannian manifold of dimension $m$. For any $k\in \NN,  \rho,  \Xi>0$, we define $\cM_{k, m}(\rho, \Xi) $ the set of  all subsets $A\subset (M, g)$ such that
\begin{enumerate}
  \item[(1).] for any $p\in A$, the harmonic radius at $p$ satisfies $r_h(p)\geq \rho$;
  \item[(2).] For any $p\in A$, the ball $\cB_\rho(p)$ has  harmonic coordinates $\{x_1, x_2, \cdots, x_m\}$ such that the metric tensor $g_{ij}$ in these coordinates satisfies
\beq
\Xi^{-1}\dd_{ij}\leq g_{ij}\leq \Xi \dd_{ij}, \quad \Big|\frac {\partial^{\al}g_{ij}}{\partial x^{\al}}\Big|\leq \Xi,\quad \hbox{on}\quad \cB_{\rho}(p),
\eeq for any multi-index $\al$  with $1\leq|\al|\leq k.$

\end{enumerate}

\end{defi}

The following theorem is the main result in this section, which gives the asymptotical behavior of a positive solution of a parabolic equation near a time-dependent  singular set.

\begin{theo}\label{theoA:main}
Let $(\Si^2, g)$ be a two-dimensional complete surface and $\{\xi_1, \xi_2, \cdots, \xi_l\}$ with $\xi_k: [T_1, T_2]\ri \Si$ be  $\si$-Lipschitz curves in $\Si$. Assume that $u(x, t)\in L^1_{\loc}((\Si\times (T_1, T_2))\b \cup_{k=1}^l\Ga_k)$ is a nonnegative solution of the equation:
\beq \pd ut=\Delta_g u+c(x, t) u, \label{eq:D004}\eeq
where  $c(x, t)\in L_{\loc}^{\infty}(\Si\times [T_1, T_2])$ and  $\Ga_k=\{(\xi_k(t), t)\}\subset \Si\times [T_1, T_2]$.
Assume that for any $k\in \{1, 2, \cdots, l\}$ and any $t\in [T_1, T_2]$ the ball $\cB_1(\xi_k(t))$ is in $\cM_{k_0, 2}(\rho_0, \Xi_0)$, where $k_0$ is an integer chosen as in Corollary \ref{cor:heat}. Then we have
\begin{enumerate}
  \item[(1)]  $u\in L^1_{\loc}(\Si\times(T_1, T_2)).$ More precisely, for any   $(t_1, t_2)\subset (T_1, T_2)$, there exists a constant $r_1=r_1(\rho_0, \Xi_0, l, t_1, t_2, T_1, T_2)>0$ such that
  \beq
  \|u\|_{L^1(Q_{r_1, t_1, t_2})}\leq C\|u\|_{L^1(K)},\label{eq:L002x}
  \eeq where $C$ is a constant depending on $\|c\|_{L^{\infty}(Q_{1, T_1, T_2})}, \rho_0, \Xi_0, \si,  t_1, t_2, T_1, T_2 $ and  $K$ is defined by $K=\overline{Q_{2r_1, T_1, T_2}}\b Q_{r_1, T_1, T_2}.$ Here $  Q_{r, t_1, t_2}$ is defined by
  \beq
  Q_{r, t_1, t_2}=\bigcup_{k=1}^l\{(x, t)\in \Si\times \RR, x\in \cB_r(\xi_k(t))\subset \Si, t\in (t_1, t_2)\}. \nonumber
  \eeq \item[(2)] For any $(t_1, t_2)\subset (T_1, T_2)$,  we have
  \beqn
  \lim_{R\ri 0}\lim_{\dd\ri 0}\int_{t_1}^{t_2}\,dt\int_{A^{(k)}_t (\frac {\dd}2, R)\cap F^{(k)}_t(\frac {\dd}2)}\,\frac {u}{\r_k|\log \r_k|}\,d\vol&=&0,\label{eqB:007}\\
\lim_{\dd\ri 0}\frac 1{\dd}\int_{t_1}^{t_2}\,dt\int_{A^{(k)}_t (\frac {\dd}2, \dd)\cap F^{(k)}_t(\frac {\dd}2)}\,u\,d\vol&<&+\infty,\label{eqB:008}
  \eeqn where $A_t^{(k)}$ and
 $F^{(k)}_t$ are defined by (\ref{eq:G008})-(\ref{eq:G009}) and $\r_k(x, t)=d(x, \xi_k(t))$.
\end{enumerate}

\end{theo}

We sketch the proof of Theorem \ref{theoA:main}. First, we show an asymptotical formula for the heat kernel on a Riemannian manifold in Theorem \ref{theo:heat}. Using this formula, we construct a special function $U_k(x, t)$ for each singular curve $\xi_k$ and a measure $\nu$,  and show that $U_k(x, t)$ behaves like $\log \frac 1{\r_k(x, t)}$ when $x$ is near $\xi_k$ and $\nu$ is the Lebesgue measure in Lemma \ref{lemA:001b}. Moreover, $U_k(x, t)$ satisfies the growth estimates (\ref{eqB:007})-(\ref{eqB:008}) by Lemma \ref{lem:estimate_U}, and  we use $U_k(x, t)$ to construct a function $ v_k$ in Lemma \ref{lemA:001b}, which satisfies the backward heat equation. The function $v_k$ is important to construct some cutoff functions(cf. Definition \ref{defi:002}).
When the singular curves are disjoint, using these cutoff functions we can  show (\ref{eq:L002x}) directly in Lemma \ref{lem:A003b}. When the singular curves are not disjoint, we   show the finiteness of a functional $I$ and use the functional $I$ to show the $L^1$ norm of $u$   (\ref{eq:L002x}) in Lemma \ref{lemA:003b}. By using the functional $I$, we get a positive linear functional $\mu_k$ for each singular curve $\xi_k$ in Lemma \ref{lemA:003}, and by Lemma \ref{lem:measure} $\mu_k$ is uniformly bounded even if the singular curves are not disjoint. Finally, we use $\mu_k$ to construct $U_k$ and show that $u$ is controlled by $U_k$. By the properties of $U_k$, we have that $u$ satisfies  the growth estimates (\ref{eqB:007})-(\ref{eqB:008}).

\subsection{Properties of the heat kernel}
In this subsection, we will  give the expansion of the heat kernel on  Riemannian manifolds. Let $(M, g)$ be a  complete Riemannian manifold (without boundary) of dimension $m$. Suppose that $p(x, y, t)$ is the heat kernel. Then $p(x, y, t)$ has the following asymptotical formula (cf. Theorem 11.1 of \cite{[Li]})
\beq p(x, y, t)\sim (4\pi t)^{-\frac m2}e^{-\frac {d_g^2(x, y)}{4t}} \label{eq:heat}\eeq
as $t\ri 0$ and $d_g(x, y)\ri 0$.   The next result gives more estimates on the asymptotical formula.

\begin{theo}\label{theo:heat}(cf. Theorem 11.1 of \cite{[Li]}, or Theorem 2.30 of \cite{[BGV]}) Let $\rho_0,  \Xi_0>0$ and integers $m\geq 2, k\geq 0$. There exists an integer $k_0=k_0(k)$ depending only on $k$ satisfying the following property.  Let $(M, g)$ be a complete Riemannian manifold of dimension $m$ and $x_0\in M$ with  $\cB_{\rho_0}(x_0)\in \cM_{k_0, m}(\rho_0, \Xi_0)$.  There exists a sequence of smooth functions $\{u_i(x, y)\}$ with $u_0(x, x)=1$ such that for any $x, y\in \cB_{\frac {\rho_0}2}(x_0)$ and $t\in (0, 1]$ we have
\beqn
\Big|p(x, y, t)-(4\pi t)^{-\frac m2}e^{-\frac {d_g^2(x, y)}{4t}}\sum_{i=0}^ku_i(x, y)t^i\Big|&\leq& C(\rho_0, \Xi_0, m)t^{k+1-\frac m2}, \label{eq:heat1}\\
\Big|\Na_xp(x, y, t)-\Na_x\Big((4\pi t)^{-\frac m2}e^{-\frac {d_g^2(x, y)}{4t}}\sum_{i=0}^ku_i(x, y)t^i\Big)\Big|&\leq& C(\rho_0, \Xi_0, m) t^{k-\frac m2}. \label{eq:heat2}
\eeqn

\end{theo}
\begin{proof} We follow the argument in  Theorem 11.1 of \cite{[Li]} to prove (\ref{eq:heat1})-(\ref{eq:heat2}). Define the function
$$G(x, y, t)=(4\pi t)^{-\frac m2}e^{-\frac {d_g^2(x, y)}{4t}}\sum_{i=0}^ku_i(x, y)t^i.$$
Direct calculation shows that
\beqs
\Big(\Delta_y-\pd {}t\Big)G&=&(4\pi t)^{-\frac m2}e^{-\frac {d_g^2(x, y)}{4t}}\times \Big(\Big(-\frac {r\Delta_y r}2+\frac {m-1}{2}\Big)\sum_{i=-1}^{k-1}\,u_{i+1}t^i\\
&& -r\sum_{i=-1}^{k-1}\langle \Na r, \Na u_{i+1} \rangle t^i+\sum_{i=0}^k(\Delta_y u_i)t^i-\sum_{i=0}^{k-1}(i+1) u_{i+1} t^i\Big).
\eeqs
For fixed $x$ and $y\in \cB_{\frac {\rho_0}2}(x)$, there exists a sequence of function $\{u_i(x, y)\}$ satisfying
\beqs
\Big(\frac {r\Delta_y r}2-\frac {m-1}{2}\Big)u_0+r\langle \Na r, \Na u_0 \rangle &=&0,\\
\Big(\frac {r\Delta_y r}2-\frac {m-1}{2}\Big)u_{i+1}+r \langle \Na r, \Na u_{i+1} \rangle +(i+1) u_{i+1}&=&\Delta_y u_i, \quad 0\leq i\leq k-1.
\eeqs This implies that
\beq
\Big(\Delta_y-\pd {}t\Big)G=(4\pi t)^{-\frac m2}e^{-\frac {d_g^2(x, y)}{4t}}\Delta_y u_k \,t^k. \label{eq:J007}
\eeq
As in the proof of Theorem 11.1 of \cite{[Li]}, we have that
\beqs
u_0(x, y)&=&C d_g(x, y)^{\frac {m-1}{2}} J^{-\frac 12},\\
u_{i+1}&=& d_g(x, y)^{\frac {m-3-2i}{2}} J^{-\frac 12}\int_0^r\, s^{\frac {2i+1-m}{2}}J^{\frac 12}\Delta u_i\, ds,
\eeqs where $C$ is a constant such that $u_0(x, x)=1$ and  $J(y)$ is the area element of the sphere of radius $d_g(x, y)$ at the point $y$.  There exists integer $k_0$ depending only on $k$  such that under the assumption  $\cB_{\rho_0}(x_0)\in \cM_{k_0, m}(\rho_0, \Xi_0)$,
 for any integer $i\in [0, k]$ we have \beq |u_i(x, y)|+ |\Na_y u_i(x, y)|+ |\Na_x\Na_y\Na_y u_i(x, y)|\leq C(\rho_0, \Xi_0, m), \quad \forall\;x, y\in \cB_{\frac {\rho_0}2}(x_0). \label{theo:heat:001}
 \eeq

Let $\rho=\frac {\rho_0}4$.  Now we choose a cutoff function $\eta(r)$ with $0\leq \eta\leq 1$  such that  $\eta(r)=1$ when $r\leq \rho$ and  $\eta(r)=0$ when $r\geq 2\rho$.  Define $\chi(x, y)=\eta(d_g(x, y))$ and  $F(x, y, t)=\chi(x, y)G(x, y, t)$.  If $d_g(x, y)\leq \rho$ and $t\leq 1$, the identity (\ref{eq:J007}) gives that
\beqs
\Big|\Big(\Delta_y-\pd {}t\Big) F\Big|&=&\Big|\Big(\Delta_y-\pd {}t\Big)G\Big|\leq C(\rho_0, \Xi_0) t^{k-\frac m2}e^{-\frac {d_g^2(x, y)}{4t}}\eeqs
and
\beqs
\Big|\Big(\Delta_y-\pd {}t\Big) \Na_x F\Big|&=&\Big|\Big(\Delta_y-\pd {}t\Big)\Na_x G\Big|\\
&=&\Big| (4\pi t)^{-\frac m2} t^k e^{-\frac {d_g^2(x, y)}{4t}} \Big(-\frac 1{2t}d\Na_xd\Delta_y u_k+\Na_x\Delta_y u_k\Big)\Big|\\
&\leq& C(\rho_0, \Xi_0) t^{k-1-\frac m2}e^{-\frac {d_g^2(x, y)}{4t}},
\eeqs where we used (\ref{theo:heat:001}) in the last inequality.
Similarly, for $\rho\leq d_g(x, y)\leq 2\rho$ we can also check that
 \beqs
\Big|\Big(\Delta_y-\pd {}t\Big) F\Big|&\leq&  C(\rho_0, \Xi_0) t^{-\frac m2-1}e^{-\frac {d_g^2(x, y)}{4t}},\\
\Big|\Big(\Delta_y-\pd {}t\Big) \Na_x F\Big|&\leq&  C(\rho_0, \Xi_0) t^{-2-\frac m2}e^{-\frac {d_g^2(x, y)}{4t}}.
\eeqs
Combining the above estimates, we have
\beqn &&
\Big|F(x, y, t)-p(x, y, t)\Big|\nonumber\\&=&\Big|\int_0^t\,ds\,\int_M\, p(z, y, t-s)\Big(\Delta_z-\pd {}s\Big) F(x, z, s)dz\Big|\nonumber\\
&\leq& C(\rho_0, \Xi_0)\int_0^t\, s^{k-\frac m2}ds \int_{\cB_{\rho}(x)}\,p(z, y, s)\,dz\nonumber\\
&&+C(\rho_0, \Xi_0)\int_0^t\,s^{-\frac m2-1}e^{-\frac {\rho^2}{4s}}ds\int_{\cB_{2\rho}(x)\b \cB_{\rho}(x)}\,p(z, y, s)\,dz\nonumber\\
&\leq&C(\rho_0, \Xi_0, m) t^{k+1-\frac m2}, \label{eq:A003}
\eeqn
where we used the fact that $\int_M\;p(x, y, t)d\vol_y\leq 1.$ Thus,  (\ref{eq:A003}) gives (\ref{eq:heat1}).  Similarly, we can show that
\beqn &&
\Big|\Na_xF(x, y, t)-\Na_x p(x, y, t)\Big|\nonumber\\&=&\Big|\int_0^t\,ds\,\int_M\, p(z, y, t-s)\Big(\Delta_z-\pd {}s)\Na_x F(x, z, s)dz\Big|\nonumber\\
&\leq& C(\rho_0, \Xi_0)\int_0^t\, s^{k-\frac m2-1}ds \int_{\cB_{\rho}(x)}\,p(z, y, s)\,dz\nonumber\\
&&+C(\rho_0, \Xi_0)\int_0^t\,s^{-2-\frac m2}e^{-\frac {\rho^2}{4t}}ds\int_{\cB_{2\rho}(x)\b \cB_{\rho}(x)}\,p(z, y, s)\,dz\nonumber\\
&\leq&C(\rho_0, \Xi_0, m) t^{k-\frac m2}. \label{eq:A0030a}
\eeqn
Thus, (\ref{eq:A0030a})  implies  (\ref{eq:heat2}).
 The theorem is proved.

\end{proof}

As a corollary, we have the following result in dimension two.

\begin{cor}\label{cor:heat}
 Fix $\rho_0,  \Xi_0>0$ and an integer $k_0=k_0(0)$ chosen as in Theorem \ref{theo:heat} for $k=0$.  Let $(\Si^2, g)$ be a complete surface and $x_0\in \Si$ with $\cB_{1}(x_0)\in \cM_{k_0, 2}(\rho_0, \Xi_0)$, there exists a constant $C(\rho_0, \Xi_0)>0$ such that for any $x, y\in \cB_{\frac {\rho_0}2}(x_0)$ and $t\in (0, 1]$ we have
\beqn
 p(x, y, t)&\leq& \Big(1+C(\rho_0, \Xi_0)d_g(x, y)\Big) p_0(x, y, t)+C(\rho_0, \Xi_0), \label{cor:heat:001}\\
p(x, y, t)&\geq&\Big(1-C(\rho_0, \Xi_0)d_g(x, y)\Big) p_0(x, y, t)-C(\rho_0, \Xi_0),\label{cor:heat:002}\\
|\Na_xp(x, y, t)|&\leq& \Big(1+C(\rho_0, \Xi_0)d_g(x, y)\Big) |\Na_x p_0(x, y, t)|+\frac {C(\rho_0, \Xi_0)}{t},\label{cor:heat:003}\\
|\Na_xp(x, y, t)|&\geq&\Big(1-C(\rho_0, \Xi_0)d_g(x, y)\Big) |\Na_x p_0(x, y, t)|-\frac {C(\rho_0, \Xi_0)}{t},\label{cor:heat:004}
\eeqn where $p_0(x, y, t)=\frac 1{4\pi t}e^{-\frac {d_g(x, y)^2}{4t}}. $

\end{cor}
\begin{proof}By (\ref{theo:heat:001}), for any $x, y\in \cB_{\frac {\rho_0}2}(x_0)$ we have
\beq
|u_0(x, y)-1|\leq \sup_{\cB_{\frac {\rho_0}2}(x_0)}|\Na_y u_0|\cdot\,d_g(x, y)\leq C(\rho_0, \Xi_0)\,d_g(x, y).\label{eq:K004}
\eeq Applying  Theorem \ref{theo:heat} for $k=0$ and using (\ref{eq:K004}), we have (\ref{cor:heat:001})-(\ref{cor:heat:004}).
The corollary is proved.

\end{proof}

\subsection{Properties of a  solution with time-dependent singularities}
In this subsection, we follow the arguments in Section 3 of \cite{[KT]} to discuss a solution of the linear equation on $(\Si, g)$
\beq
\pd {f(x, t)}t=\Delta f(x, t)+\dd_{\xi(t)}\otimes \nu,  \label{eq:Ut}
\eeq where $\Si$ is a complete two-dimensional surface.
Here we assume that $\xi: (\un T, \bar T)\ri \Si$ is a $\si$-Lipschitz curve with $-\infty< \un T<\bar T< +\infty
$, $\dd_{\xi(t)}$ denotes the Dirac function with the pole $\xi(t)$ and $\nu\in (C_0((\un T, \bar T)))'$.
For $0<r< +\infty$ and $\un T\leq \un t<\bar t\leq \bar T$, we set
\beqn
\Ga_{\un t, \bar t}&=&\{(\xi(t),t)\in \Si\times \RR, t\in (\un t, \bar t) \},\label{eqE:001}\\Q_{r, \un t, \bar t}&=&\{(x, t)\in \Si\times \RR, x\in \cB_r(\xi(t))\subset \Si, t\in (\un t, \bar t)\}. \label{eqE:002}
\eeqn
We say that $f(x, t)$ is a solution of (\ref{eq:Ut}) if for any $\varphi\in C_0^{\infty}(Q_{\infty, \un T, \bar T}),$
\beqn
\int_{Q_{\infty, \un T, \bar T}}\,f(x, t)\Big(-\pd {\varphi(x, t)}t-\Delta \varphi(x, t)\Big)\,d\vol\, dt=\int_{\un T}^{\bar T}\,\varphi(\xi(s), s)d\nu(s). \label{eq:E1}
\eeqn
 We define the function $U(x, t)$ by
\beq
U(x, t)=\int_{\un T}^t\,p(x, \xi(s), t-s)\,d\nu(s), \label{Sec4.2:003}
\eeq where $p(x, y, t)$ is the heat kernel of $(\Si, g)$. Then $U(x, t)$ satisfies (\ref{eq:E1})(cf. \cite{[KT]}).
Moreover, we define
\beq \Phi(x, y)=\frac 1{2\pi}\log \frac 1{d_{g}(x, y)},\quad \r(x, t)=d_g(x, \xi(t)). \label{eq:rk}\eeq

Following the  argument  in \cite{[KT]} and using Theorem \ref{theo:heat}, we have
\begin{lem}\label{lemA:001b}(cf. \cite{[KT]})  Let $\xi: (\un T, \bar T)\ri \Si$ be a $\si$-Lipschitz curve and $\un T<\un t<\bar t<\bar T$. Assume that for any $t\in (\un T, \bar T)$ the ball  $\cB_{1}(\xi(t))$ is in $ \cM_{k_0, 2}(\rho_0, \Xi_0) $ as in Corollary \ref{cor:heat}.
Then we have \begin{enumerate}
  \item[(1).] Assume that $\nu$ is the Lebesgue measure.  For any
  $\ee>0$, there exists $r_0=r_0(\ee, \si, \Xi_0, \rho_0, \un T, \bar T)$ such that if $\r(x, t)\leq r_0$ and $t\in (\un t, \bar t)$, then we have
\beqn
(1-\ee)\Phi(x, \xi(t)) &\leq& U(x, t)\leq (1+\ee)\Phi(x, \xi(t)),\label{eq:J003}\\
(1-\ee)|\Na \Phi(x, \xi(t))|&\leq& |\Na U(x, t)|\leq (1+\ee)|\Na \Phi(x, \xi(t))|.\label{eq:J004}
\eeqn

\item[(2).]   For any $\ga\in (\frac 12, 1)$, there exist  constants $r_0=r_0(\rho_0, \Xi_0, \si, \un T, \bar T, \ga)\in (0, 1) $ and a function $v\in C^{\infty}(\overline{Q_{1, \un t, \bar t}}\b \overline{\Ga_{\un t, \bar t}})$ satisfying
\beq
\pd vt+\Delta v=0,\quad \hbox{in}\quad Q_{1, \un t, \bar t}\b \Ga_{\un t, \bar t} \label{eqA:005}
\eeq
such that for all $(x, t)\in Q_{r_0, \un t, \bar t}\b \Ga_{\un t, \bar t}$ the following inequalities hold:
\beqn
\ga \log\frac 1{\r(x, t)}&\leq& v(x, t)\leq \log \frac 1{\r(x, t)},\label{eqA:003}\\
 \ga \;\r(x, t)^{-1}&\leq& |\Na v(x, t)|\leq  \r(x, t)^{-1}. \label{eqA:004}
\eeqn

\end{enumerate}

\end{lem}
\begin{proof}The proof is almost the same as  that of Proposition 3.1, Proposition 3.3 and Lemma 4.1 in \cite{[KT]}, and we sketch some details here.
For $r>0, \bb>0$ and $\dd>0$, we define
$$S_{\bb}(r)=\pi^{-1}\int_0^{\dd}\,(4s)^{-\frac {\bb}2} e^{-\frac {r^2}{4s}}\,ds.$$
Since $\xi$ is $\si$- Lipschitz continuous, we have
\beq
|\r(x, t)-\r(x, s)|\leq \si |t-s|. \label{eq:J005}
\eeq Thus, for any $c>0$ we have
\beq
\r(x, s)^2\leq (1+c)\r(x, t)^2+\Big(1+\frac 1c\Big)\si^2|t-s|^2.
\eeq
This implies that
\beq
\r(x, s)^2\geq \frac 1{1+c}\r(x, t)^2-\frac 1{c}\si^2|t-s|^2.\label{eqn:001}
\eeq
Combining this with Corollary  \ref{cor:heat}, we have
\beqs &&
\int_{t-\dd}^t\, p(x, \xi(s), t-s)\,ds\\&\leq& \Big(1+C(\rho_0, \Xi_0)(\r(x, t)+\si \dd)\Big)e^{\frac {\si^2\dd}{4c}}\int_{t-\dd}^t\,
\frac 1{4\pi (t-s)}\,e^{-\frac {\r(x, t)^2}{4(1+c)(t-s)}}\,ds+C(\rho_0, \Xi_0, \dd)\\
&=&\Big(1+C(\rho_0, \Xi_0)(\r(x, t)+\si \dd)\Big)e^{\frac {\si^2\dd}{4c}} S_2\Big(\frac {\r(x, t)}{\sqrt{1+c}}\Big)+C(\rho_0, \Xi_0, \dd).
\eeqs
Choosing the constant $c=\sqrt{\dd}$, we have
\beqs
U(x, t)&=&\Big(\int_{t-\dd}^t+\int_{\un  T}^{t-\dd}\Big)\, p(x, \xi(s), t-s)ds\\
&\leq &\Big(1+C(\rho_0, \Xi_0)(\r(x, t)+\si \dd)\Big)e^{\frac {\si^2 \sqrt{\dd}}4} S_2\Big(\frac {\r(x, t)}{\sqrt{1+\sqrt{\dd}}}\Big)+C(\rho_0, \Xi_0, \dd, \un T, \bar T).
\eeqs
Note that $\lim_{r\ri 0}(\log \frac 1r)^{-1}S_2(r)=\frac {1}{2\pi}$. Therefore, for any $\ee>0$ there exists $r_0=r_0(\ee, \si, \Xi_0, \rho_0, \un T, \bar T)$ such that for any  $x$ with $\r(x, t)\leq r_0$ we have
$$\frac {U(x, t)}{\Phi(x, \xi(t))}\leq 1+\ee. $$
Similarly, we can show that
$\frac {U(x, t)}{\Phi(x, \xi(t))}\geq 1-\ee $
when $\r(x, t)$ is small. Thus, (\ref{eq:J003}) is proved.  Similarly, we can use (\ref{cor:heat:003}) and (\ref{cor:heat:004}) of Corollary \ref{cor:heat} to estimate $|\Na U|$.

To prove part (2), we denote by $U(x, t; \xi, \nu)$ the function (\ref{Sec4.2:003}) constructed by $\xi(t)$ and the measure $\nu.$ We define $\td \xi(t)=\xi(\un T+\bar T-t)$ and let $\nu$ be the Lebesgue measure. Then the function $v(x, t)=k U(x, \un T+\bar T-t; \td \xi, \nu)$ satisfies the properties in part (2) by choosing some $k>0$.
See Lemma 4.1 of \cite{[KT]} for  details.
\end{proof}

Using Corollary \ref{cor:heat}, we have the following result.

\begin{lem}\label{lem:estimate_U} Let $\xi: (\un T, \bar T)\ri \Si$ be a $\si$-Lipschitz curve and $\un T<\un t<\bar t<\bar T$. Assume that for any $t\in (\un T, \bar T)$ the ball  $\cB_{1}(\xi(t))$ is in $ \cM_{k_0, 2}(\rho_0, \Xi_0) $ as in Corollary \ref{cor:heat}.  Let $\nu\in (C_0((\un T, \bar T)))'$ and $U(x, t)$ be the function defined by (\ref{Sec4.2:003}). Then for $\un T<t_1<t_2<\bar T$ we have
 \beqn
  \lim_{R\ri 0}\lim_{\dd\ri 0}\int_{t_1}^{t_2}\,dt\int_{A_t(\dd, R)}\,\frac {U(x, t)}{\r(x, t)|\log \r(x, t)|}&=&0,\label{eq:C003}\\
 \lim_{\dd\ri 0}\frac 1{\dd}\int_{t_1}^{t_2}\,dt\int_{A_t(\frac {\dd}2, \dd)}\,U(x, t)&=&0, \label{eq:C004}
  \eeqn where $A_t(\dd, R)=\{x\in \Si\, |\, \dd<\r(x, t)<R\}.$
\end{lem}

\begin{proof}We follow the arguments  in the proof of Proposition 3.3 in \cite{[KT]}. Without loss of generality, we can assume that the curve $\xi(s)(s\in (\un T, \bar T))$ is contained in $\cB_{\frac 12\rho_0}(x_0)$ for some $x_0\in \Si$ and $\cB_{\rho_0}(x_0)\in \cM_{k_0, 2}(\rho_0, \Xi_0)$.
Corollary \ref{cor:heat} gives that for any $x\in \cB_{\frac 12\rho_0}(x_0)$ and $t\in (\un T, \bar T)$,
\beqs U(x, t)&\leq &\int_{\un T}^t\,\Big(C(\rho_0, \Xi_0)p_0(x, \xi(s), t-s)+C(\rho_0, \Xi_0)\Big)\,d\nu\\
&= & C(\rho_0, \Xi_0)U_0(x, t)+C(\rho_0, \Xi_0)(\bar T-\un T),\eeqs
where $U_0$ is defined by
$$U_0(x, t)=\int_{\un T}^t\,p_0(x, \xi(s), t-s)\,d\nu.$$
Thus, it suffices to show (\ref{eq:C003})-(\ref{eq:C004}) for $U_0(x, t)$.

 For $t\in (\un T, \bar T)$ with $|D\nu|<+\infty$ we write
$$\nu((s, t])=D\nu(t)(t-s)-G(s),\quad T_1<s<t,$$
where $G(s)$ satisfies
$\lim_{s\ri t^-}\frac {G(s)}{t-s}=0$ for a.e. $t\in (\un T, \bar T)$.
Let $\la\in (0, t-\un t)$. Note that $U_0$ can be written as
\beqs
U_0(x, t)&=&\int_{\un T}^{t-\la}\, p_0(x, \xi(s), t-s)\,d\nu(s)+D\nu(t)\int_{t-\la}^t\,
p_0(x, \xi(s), t-s)\,ds\\&&+\int_{t-\la}^t\,p_0(x, \xi(s), t-s)dG(s)\\
&=:&I_1+I_2+I_3.
\eeqs
By Theorem \ref{theo:heat}  $I_1$ satisfies
$I_1\leq \frac 1{ 4\pi \la}\nu((\un T, t-\la))<+\infty$. Thus,
we have
\beqn &&
\int_{t_1}^{t_2}\,dt\int_{A_t(\dd, R)}\,\frac {I_1(x, t)}{\r(x, t)|\log \r(x, t)|}\,d\vol\nonumber\\&\leq& \frac 1{ 4\pi \la}\nu(( \un T, \bar T))\int_{t_1}^{t_2}\,dt\int_{A_t(\dd, R)}\,\frac 1{\r(x, t)|\log \r(x, t)|}\,d\vol\nonumber\\
&=&\frac 1{ 2 \la}(t_2-t_1)\nu((\un T, \bar T))\int_{\dd}^R\, \frac 1{|\log \r|}\,d\r\nonumber\\
&\leq&\frac 1{ 2\la}(t_2-t_1)\nu((\un T, \bar T))(R-\dd), \label{eqG:005}
\eeqn where we assumed that $R$ is  small such that $|\log \r|\geq 1$ for any $\r\in (0, R)$. Moreover, we have
\beq
\frac 1{\dd}\int_{t_1}^{t_2}\,dt\int_{A_t(\frac {\dd}2, \dd)}\,I_1(x, t)\,d\vol\leq\frac {C_1}{  \la}(t_2-t_1)\nu((\un T, \bar T))\dd,
\eeq where $C_1$ is a universal constant.
Next, we estimate $I_2$. Using Corollary  \ref{cor:heat}, (\ref{eqn:001}) and integration by parts we have
\beqn &&
\int_{t-\la}^t\,
p_0(x, \xi(s), t-s)\,ds\nonumber\\&\leq &\int_0^{\la}\,\frac 1{ 4\pi\tau}e^{-\frac {\r(x, s)^2}{4\tau}}\,d\tau\leq e^{\frac {\si^2\la}{4c}}\int_0^{\la}\,\frac 1{ 4\pi\tau}e^{- \frac {\r(x, t)^2}{4(1+c)\tau}}\,d\tau\nonumber\\&\leq&\frac 1{4\pi} e^{\frac {\si^2\la}{4c}}\Big(\log \frac {4(1+c)\la}{\r(x, t)^2}e^{-\frac {r^2}{4(1+c)\la}}+\int_{\frac {\r(x, t)^2}{4(1+c)\la}}^{\infty}\,e^{-z}\log z\,dz \,\Big)\nonumber\\
&\leq& C_2|\log \r(x, t)|+C_2, \label{eq:Z004}
\eeqn where we can choose $c=1$ and $C_2$ is a constant depending on $\si$ and $\la.$
Therefore, we have
\beqn &&
\int_{t_1}^{t_2}\,dt\int_{A_t(\dd, R)}\,\frac {I_2(x, t)}{\r(x, t)|\log \r(x, t)|}\,d\vol \nonumber\\&\leq&C_2 \int_{t_1}^{t_2}\,d\nu(t) \int_{\dd}^R\,\frac 1{\r|\log \r|}\Big(|\log \r|+1\Big)\,\r d\r\nonumber\\
&\leq &2C_2\cdot (R-\dd)\nu((t_1, t_2))
\eeqn and
\beq
\frac 1{\dd}\int_{t_1}^{t_2}\,dt\int_{\frac {\dd}2<\r(x, t)< \dd}\,I_2(x, t)\,d\vol\leq C_2\cdot (\dd+\dd|\log \dd| )\nu((t_1, t_2)).
\eeq
Finally, we estimate $I_3.$
Using the inequality (\ref{eqn:001}) for $c=1$ and  integration by parts,  we have
\beqs
&&\int_{t-\la}^t\, p_0(x, \xi(t), t-s)d|G(s)|
\\&\leq & \frac 1{4\pi}e^{\frac {\si^2\la}{4}}\int_{t-\la}^t\, \frac 1{t-s}e^{-\frac {\r(x, t)^2}{8(t-s)}}\\
&\leq&\frac 1{4\pi}e^{\frac {\si^2\la}{4}}\Big(-\frac 1{\la}e^{-\frac {\r(x, t)^2}{8\la}}+\int^t_{t-\la}\,|G(s)|\Big(\frac 1{(t-s)^2}+\frac {\r(x, t)^2}{8(t-s)^3}\Big)e^{-\frac {\r(x, t)^2}{8(t-s)}}\Big)\,ds\\
&\leq&\frac 1{4\pi}e^{\frac {\si^2\la}{4}}\sup_{(t-\la, t)}\frac {|G(s)|}{t-s}\cdot \int^t_{t-\la}\, \Big(\frac 1{(t-s)}+\frac {\r(x, t)^2}{8(t-s)^2}\Big)e^{-\frac {\r(x, t)^2}{8(t-s)}}\,ds\\
&\leq&C_3\sup_{(t-\la, t)}\frac {|G(s)|}{t-s}\cdot |\log \r(x, t)|,
\eeqs where $C_3$ depends on $\si$ and $\la$.
Thus,  we have
\beqn &&
\int_{t_1}^{t_2}\,dt\int_{A_t(\dd, R)}\,\frac {|I_3(x, t)|}{\r(x, t)|\log \r(x, t)|}\,d\vol \nonumber\\&\leq&
C(\si, \la)\sup_{(t-\la, t)}\frac {|G(s)|}{t-s}\int_{t_1}^{t_2}\,dt\int_{A_t(\dd, R)}\,\frac 1{\r(x, t)}\,d\vol\nonumber\\
&\leq &C(\si, \la)\sup_{(t-\la, t)}\frac {|G(s)|}{t-s}\cdot (R-\dd)(t_2-t_1)
\eeqn and
\beq \frac 1{\dd}\int_{t_1}^{t_2}\,dt\int_{\frac {\dd}2<\r(x, t)< \dd}\,|I_3(x, t)|\,d\vol
\leq C(\si, \la)\sup_{(t-\la, t)}\frac {|G(s)|}{t-s}\cdot (t_2-t_1)|\log \dd|\dd.\label{eqG:006}
\eeq
Combining (\ref{eqG:005})-(\ref{eqG:006}), we have (\ref{eq:C003})-(\ref{eq:C004}).

\end{proof}

\subsection{Estimates  of the solution with disjoint singularities}
In this subsection, we follow Section 4.1 of \cite{[KT]} to  construct some cutoff functions and show the
integrability of the solution across the singular set  when the singular curves are disjoint.
First, we construct some cutoff functions.

\begin{defi}\label{defi:001}(cf. Section 4.1 of \cite{[KT]})
\begin{enumerate}
  \item[(1).]  Let $t_3<t_1<t_2<t_4$ and $0<\dd<r_1$. Define $\zeta=\zeta(t; t_1, t_2, t_3, t_4, \dd, r_1)\in C^{\infty}(\RR)$ such  that
\beqn \zeta(t)&=&\dd, \;(t\in [t_1, t_2]),\quad \zeta(t)=r_1\, (t\in (-\infty, t_3]\cup [t_4, \infty)),\nonumber\\
 0&<&\Big|\pd {\zeta}t\Big|\leq 2r_1\Big(\frac 1{t_1-t_3}+\frac 1{t_4-t_2}\Big). \label{eq:zeta}
\eeqn

\item[(2)] Let $\eta$ be a smooth function on $\RR$ satisfying
\beq
\eta(z)=\left\{
          \begin{array}{ll}
            0, & (z\leq 0), \\
            1, & (z\geq 1),
          \end{array}
        \right.\quad 0<\eta'(z)\leq 2\;(0<z<1),
\eeq and define $H(z)=\int_0^z\,\eta(\tau)\,d\tau$. Then $H(z)$ satisfies the inequality
\beq  0\leq zH'(z)-H(z)\leq H'(z). \label{eqA:100}\eeq
We keep the same notation $H(z)$ as in \cite{[KT]}.
Throughout this section,  $H$ always denotes the function as above and it should not be confused with the mean curvature.

\item[(3).]Let $0<\un r<\bar r<1$, $T_1<\un T<\un t<\bar t<\bar T<T_2$ and $\xi:[T_1, T_2]\ri \Si$ be a continuous curve. We define $\phi_{\xi}=\phi_{\xi}(x, t; \un r, \bar r,  \un t, \bar t, \un T, \bar T, T_1, T_2)\in C^{\infty}(Q_{1, T_1, T_2})$ satisfying
\beq
0\leq \phi_{\xi}\leq 1, \quad
\phi_{\xi}=\left\{
       \begin{array}{ll}
         1, & \hbox{on} \;\overline {Q_{\un r, \un t, \bar t}},\\
         0, & \hbox{on } \;Q_{1, T_1, T_2}\b  {Q_{\bar  r, \un T, \bar T}},
       \end{array}
     \right.\quad \Na_x\phi_{\xi}=0\, \;\hbox{in}\; Q_{\un r, T_1, T_2}. \label{eqA:103}
\eeq

\end{enumerate}

\end{defi}

A direct corollary of Lemma \ref{lemA:001b} is the following result.

\begin{lem}\label{lem:V} Under the assumption of Lemma \ref{lemA:001b}, we define
$
V(x, t)=e^{-2  v(x, t)}\in C^{\infty}(\overline{Q_{1, \un t, \bar t}}\b \overline{\Ga_{\un t, \bar t}}).
$
Then $V(x, t)$ satisfies
\beq
\pd Vt+\Delta V=4e^{-2 v}|\Na  v|^2,\quad \hbox{in}\quad Q_{1, \un t, \bar t}\b \Ga_{\un t, \bar t}
\eeq
By using the inequalities  (\ref{eqA:003})-(\ref{eqA:004}), for all $(x, t)\in Q_{r_0, \un t, \bar t}\b \Ga_{\un t, \bar t}$ the following inequalities hold:
\beqn
\r(x, t)^{2}&\leq& V(x, t)\leq \r(x, t)^{2\ga},\label{eqA:006a}\\
1&\leq&V(x, t)^{-1}|\Na V(x, t)|^2\leq 4\r(x, t)^{2\ga-2},\label{eqA:006b}\\
1&\leq &\pd Vt+\Delta V\leq 4 \r(x, t)^{2\ga-2}, \label{eqA:006}
\eeqn where $\ga\in (\frac 12, 1)$.

\end{lem}

\def\loc{{\mathrm{loc}}}

Consider the case that there is only one singular curve. We show that the solution of (\ref{eq:D004}) is in  $L^1$ across the singular set. The argument is the same as that of \cite{[KT]} and we give all the details for the readers' convenience.

\begin{lem} \label{lem:A003b}(cf. Lemma 4.2 of \cite{[KT]}) Fix $\ga\in (\frac 12, 1)$. Under the same assumption as in Theorem \ref{theoA:main}, if there is only one   singular curve $\xi:[T_1, T_2]\ri \Si$,   then  for any $(t_1, t_2)\subset (T_1, T_2)$ there exists $r_1=r_1(\rho_0, \Xi_0, \si,  t_1, t_2, T_1, T_2, \ga)>0$ such that
  \beq
  \|u\|_{L^1(Q_{r_1, t_1, t_2})}\leq C \|u\|_{L^1(K)},\label{eq:L003}
  \eeq where $C$ is a constant depending on $\|c\|_{L^{\infty}(Q_{1, T_1, T_2})}, \ga, \rho_0, \Xi_0, \si,  t_1, t_2, T_1, T_2 $ and  $K$ is defined by $K=\overline{Q_{2r_1, T_1, T_2}}\b Q_{r_1, T_1, T_2}.$
\end{lem}
\begin{proof}
   Let $T_1<t_5<t_3<t_1<t_2<t_4<t_6<T_2$, $\ga\in (\frac 12, 1)$ and $r_0=r_0(\rho_0, \Xi_0, \si, t_5, t_6, \ga)>0$ as in   Lemma \ref{lemA:001b}.  Let $0<\dd<r_1<\frac {r_0}{2}$.
   We  construct the function $$\phi(x, t)=\phi(x, t; r_1, 2r_1, t_3, t_4, t_5, t_6, T_1, T_2)$$ satisfying (\ref{eqA:103}), and
 the function $v\in C^{\infty}(\overline{Q_{\rho_0,  t_5, t_6}}\b \overline{\Ga_{t_5,  t_6}})$ satisfying (\ref{eqA:005}) with the properties (\ref{eqA:003})-(\ref{eqA:004}). Moreover, we define \beq V(x, t)=e^{-2 v(x, t)}, \quad w(x, t)=\zeta(t)^{-1}V(x, t)-1  \label{eq:K005}\eeq and \beq \varphi(x, t)=\phi(x, t)\zeta(t) (H\circ w)(x, t),\eeq where $ \zeta=\zeta(t; t_1, t_2, t_3, t_4, \dd, r_1)$ and $ H$ are  given in Definition \ref{defi:001}. Note that $H\circ w=0$ near $\Ga$ in $Q_{r_0, t_5, t_6}$. This implies that $\varphi\in C_0^{\infty}(Q_{r_0, t_5, t_6}\b \Ga_{t_5, t_6})$.
By (\ref{eq:D004}) we have
\beq
-\int_{Q_{r_0, t_5, t_6}}\; u\Big(\pd {\varphi}t+\Delta \varphi\Big)=\int_{Q_{r_0, t_5, t_6}}\; cu\varphi\; . \label{eqB:002xx}
\eeq
Note that (\ref{eqA:100}) and (\ref{eq:K005}) imply that
\beq \zeta H\circ w\leq \zeta w H'\circ w\leq VH'\circ w, \label{eqn:003}\eeq
 we have
$
\varphi
\leq \phi  VH'\circ w.
$ Thus,  the right hand side of (\ref{eqB:002xx}) can be estimated by  \beqn
\int_{Q_{r_0, t_5, t_6}}\;cu\varphi&\geq& - \,\|cV\|_{L^{\infty}({Q_{r_0, t_5, t_6}})}\int_{Q_{r_0, t_5, t_6}}\,u\phi  H'\circ w.\label{eqB:005}
\eeqn
On the other hand, direct calculation shows that $$\pd {\varphi}t+\Delta \varphi=\phi A+B $$ where
\beqs
A&=& (\partial_t V+\Delta V)H'\circ w-\partial_t\zeta\Big((w+1)H'\circ w-H\circ w\Big)+\zeta^{-1}|\Na V|^2 H''\circ w,\\
B&=&(\partial_t \phi+\Delta \phi)\zeta H\circ w+2\langle \Na \phi, \Na V \rangle H'\circ w.
\eeqs
By (\ref{eqA:100}) and (\ref{eqA:006}) we have
\beqs
A &\geq &(\partial_t V +\Delta V )H'\circ w -2|\partial_t \zeta| H'\circ w \geq (1-2\|\partial_t\zeta\|_{L^{\infty}(\RR)}) H'\circ w .
\eeqs
Note that
\beqs
\Supp(B)&\subset& \Supp(|\Na \phi|+|\p_t\phi|)\cap\{(x, t)\in Q_{2r_1, t_5, t_6}\;|\; w\geq 0\}\\&\subset&\{(x, t)\in \overline{Q_{2r_1, t_5, t_6}}\b\, Q_{r_1, t_3. t_4}\;|\; w\geq 0\}\\
&\subset &\{(x, t)\in \overline{Q_{2r_1, t_5, t_6}}\b \,Q_{r_1, t_3. t_4}\;|\; \r(x, t)\geq \zeta(t)^{\frac 1{2\ga}}\}\\
&\subset & \{(x, t)\in \overline{Q_{2r_1, t_5, t_6}} \;|\; \r(x, t)\geq  r_1\}\\
&=:&K,
\eeqs where we used the construction of $\zeta(t)$ in Definition \ref{defi:001}.
Thus,  we have
\beqs  |B | &\leq& \Big(\|\partial_t\phi+\Delta\phi\|_{L^{\infty}(K)}\|V\|_{L^{\infty}(K)
}+2\|\Na \phi\|_{L^{\infty}(K)
}\|\Na V\|_{L^{\infty}(K)
}\Big) \chi_{K}\\
&\leq& C(c_{\phi, K},  \ga, r_1)\chi_K,\eeqs
where $c_{\phi, K}=\sup_K\, (|\partial_t \phi|+|\Delta\phi|+|\Na \phi|).$
Combining the above estimates, we have
\beq
\pd {\varphi}t+\Delta \varphi\geq (1-2\|\partial_t\zeta\|_{L^{\infty}(\RR)})\phi H'\circ w-C(c_{\phi, K},  \ga, r_1)\chi_K.\label{eqB:100xx}
\eeq
Combining (\ref{eqB:005}) (\ref{eqB:100xx}) with (\ref{eqB:002xx}), we have
\beq
\Big(1-2\|\partial_t\zeta\|_{L^{\infty}(\RR)}- \,\|cV\|_{L^{\infty}({Q_{r_0, t_5, t_6}})}\Big)\int_{Q_{r_0, t_5, t_6}}\;u\,\phi H'\circ w\leq C(c_{\phi, K},  \ga, r_1)\int_{K}\;u. \nonumber
\eeq
Taking $r_0$ sufficiently small and using the assumption that $c(x, t)$ is locally bounded, we have
$$
1-2\|\partial_t\zeta\|_{L^{\infty}(\RR)}- \,\|cV\|_{L^{\infty}({Q_{r_0, t_5, t_6}})}\geq \frac 12.
$$
Therefore, by the definition of $\phi$ we have
\beq
\int_{Q_{r_1, t_1, t_2}}\;u\, H'\circ w\leq \int_{Q_{r_0, t_5, t_6}}\;u\,\phi H'\circ w\leq C(c_{\phi, K},  \ga, r_1)\int_{K}\;u.\label{eqn:002}
\eeq
Note that the function $H'\circ w$ converges to $1$ on $Q_{r_1, t_1, t_2}\b \Ga_{t_1, t_2}$ as $\dd\ri 0$. Thus, taking $\dd\ri 0$ in (\ref{eqn:002}), we have
that $u$ is integrable on $Q_{r_1, t_1, t_2}$. The lemma is proved. \\

\end{proof}

\subsection{Estimates  of the solution with multiple singularities}

In this subsection, we show that the solution of (\ref{eq:D004}) is $L^1$ across the singularities when  multiple singular curves exist. If any two singular curves don't coincide at any time, we can use Lemma \ref{lem:A003b} for each singular curve and get the $L^1$ estimates. Otherwise, the proof will be much more difficult. The idea  comes from a combination of the arguments in Lemma 4.2 and Lemma 4.3 of \cite{[KT]}, but we need to use some new cutoff functions in Definition \ref{defi:002}. We sketch the proof as follows. First, we control the $L^1$ norm of $u$ near the intersection point $(x_0, t_0)$ by an integral which characterizes the growth of $u$ near the singular curves away from $(x_0, t_0)$(c.f. (\ref{eqE:004})). Next, the integral of $u$ is bounded by the $L^1$ norm of $u$ on some compact set $K$ away from the singular curves(c.f. (\ref{eqE:006})). Combining the above two steps, we can bound  the $L^1$ norm of $u$ near the intersection point.

First, we introduce the following definition.

\begin{defi}Let $\{\xi_1, \xi_2, \cdots, \xi_l\}(t\in [T_1, T_2])$ be continuous curves in $\Si$, and  $I\subset [T_1, T_2].$ We say that $\{\xi_1(t), \cdots, \xi_l(t)\}$ are disjoint on  $I$, if for any time $t_0\in I$, we have
$$\xi_i(t_0)\neq \xi_j(t_0),\quad \forall\, i\neq j.$$
\end{defi}

Let $(x_0, t_0)$ be a point in the singular set. By Lemma \ref{prop:weakcompactness2}, there exists finitely many singular curves passing through $(x_0, t_0)$. There are two cases for the singular curves:
\begin{enumerate}
  \item[$(A)$.] There exists $(t_1, t_2)$ with $t_0\in (t_1, t_2)$ and singular curves $\{\xi_1(t), \cdots, \xi_l(t)\}(t\in (t_1, t_2))$ such that
$\{\xi_1(t), \cdots, \xi_l(t)\}$
are disjoint on $(t_1, t_2)\b \{t_0\}$ and
\beq \xi_1(t_0)=\xi_2(t_0)=\cdots=\xi_l(t_0).  \label{eq:disjoint}\eeq
  \item[$(B)$.]There exists $(t_1, t_2)$ with $t_0\in (t_1, t_2)$,  singular curves
  $\{c_1(t), \cdots, c_k(t)\}(t\in (t_1, t_0])$ and   $\{\td c_1(t), \cdots, \td c_{l}(t)\}(t\in [t_0, t_2))$ such that
\begin{enumerate}
  \item  $\{c_1(t), \cdots, c_k(t)\}$ are disjoint on $(t_1, t_0)$;
  \item  $\{\td c_1(t), \cdots, \td c_{l}(t)\}$ are disjoint on $(t_0, t_2)$;
  \item The singular curves coincide at $t_0$:
  \beq
  c_1(t_0)=\cdots=c_k(t_0)=\td c_1(t_0)=\cdots=\td c_l(t_0)=x_0.
  \eeq
\end{enumerate}
If $k=l$, then this is just the case $(A)$. Note that the union of two Lipschitz curves is still a Lipschitz curve. Thus, for $k<l$ we can construct the curves
\beq
\xi_i(t)=\left\{
           \begin{array}{ll}
             c_i(t), & \forall\; t\in (t_1, t_0], \\
             \td c_i(t), & \forall\; t\in (t_0, t_2),
           \end{array}
         \right. \;\hbox{for }\quad 1\leq i\leq k,
\eeq
and
 \beq
\xi_i(t)=\left\{
           \begin{array}{ll}
             c_k(t), & \forall\; t\in (t_1, t_0], \\
             \td c_i(t), & \forall\; t\in (t_0, t_2),
           \end{array}
         \right. \;\hbox{for }\quad k< i\leq l,
\eeq
Then $\{\xi_1(t), \cdots, \xi_{l}(t)\}(t\in (t_1, t_2))$ are Lipschitz curves. For $k>l$ we can also construct similar curves  $\{\xi_1(t), \cdots, \xi_{k}(t)\}(t\in (t_1, t_2))$.
\end{enumerate}

Summarizing the above discussion, we define
\begin{defi}\label{defi:around}Let $I=[t_1, t_2]$ or $(t_1, t_2)$ where $t_1<t_0<t_2$.
We call that the singular curves $\{\xi_1, \xi_2, \cdots, \xi_l\}(t\in I)$ are around $(x_0, t_0)$ on $I$, if the curves satisfy the conditions in Case $(A)$ or are constructed as in Case $(B)$ on $I$.

\end{defi}

We construct some cutoff functions when the  singular curves are not disjoint.

\begin{defi}\label{defi:002}

Let $0<\un r<\bar r<1$, $T_1<\un T<\un t<\bar t<\bar T<T_2$ and $\{\xi_1, \xi_2, \cdots, \xi_l\}(t\in [T_1, T_2])$ be $\si$-Lipschitz  curves. We assume that $\{\xi_1(t), \cdots, \xi_l(t)\}$
are around $(x_0, t_0)$ on $(T_1, T_2)$ for some $t_0\in (\un t, \bar t)$.
\begin{enumerate}
             \item[(1).] For each $\xi_k$ and $(t_1, t_2)\subset [T_1, T_2]$, we define the notations $Q_{r, t_1, t_2}^{(k)}$ and $ \Ga^{(k)}_{t_1, t_2}$ as in (\ref{eqE:001})-(\ref{eqE:002}), and we define
\beq
Q_{r, t_1, t_2}=\cup_{k=1}^lQ^{(k)}_{ r, t_1, t_2},\quad
\Ga_{t_1, t_2}=\cup_{k=1}^l\Ga^{(k)}_{t_1, t_2},\quad \hat Q_{r, t_1, t_2}=\cap_{k=1}^lQ^{(k)}_{ r, t_1, t_2}. \label{eq:E006}
\eeq

\item[(2).]
For each $\xi_k$ we define the function $\phi_{\xi_k}(x, t; \un r, \bar r,  \un t, \bar t, \un T, \bar T, T_1, T_2)\in C^{\infty}(Q^{(k)}_{1, T_1, T_2})$ as in (\ref{eqA:103}). Then the function
\beq
\phi(x, t; \un r, \bar r,  \un t, \bar t, \un T, \bar T)=1-(1-\phi_{\xi_1})(1-\phi_{\xi_2})\cdots (1-\phi_{\xi_l})\in C^{\infty}(\hat Q_{1, T_1, T_2} )
\eeq satisfies the properties:
\beq
0\leq \phi\leq 1, \quad
\phi=\left\{
       \begin{array}{ll}
         1, & \hbox{on} \;\hat Q_{1, T_1, T_2}\cap \overline {Q_{\un r, \un t, \bar t}},\\
         0, & \hbox{on } \; \hat Q_{1, T_1, T_2}\b \overline {Q_{\bar r, \un T, \bar T}},
       \end{array}
     \right. \label{eqA:103b}
\eeq Moreover,  $\phi$ satisfies the properties
\beq
\Supp(\phi)\cap \hat Q_{1, \un T, \bar T}\subset Q_{\bar r, \un T, \bar T}= \bigcup_{k=1}^l\,\{(x, t)\in \hat Q_{1, \un T, \bar T}\,|\, \r_k(x, t)\leq \bar r\},\label{eq:phi}\eeq and
\beqn &&
\Supp(|\Na \phi|+|\p_t\phi|)\cap \hat Q_{1, \un t,  \bar t} \subset \overline{Q_{\bar r, \un t,  \bar t}}\b Q_{\un  r, \un t,  \bar t}\nonumber\\&\subset&\bigcup_{k=1}^l\,\{(x, t)\in \hat Q_{1,  \un t,  \bar t}\;|\;
\un r\leq \r_k(x, t)\leq \bar r,\;\; \r_i(x, t)\geq \un r, \forall\, i\neq k\}.\label{eq:phigrad}
\eeqn
Here we assumed that $Q_{\bar r, \un T, \bar T}\subset \hat Q_{1, \un T, \bar T}$ by shrinking the interval $[T_1, T_2]$ if necessary.

\item[(3).] Fix $\ga\in (\frac 12, 1)$.  For each $\xi_k$, we define $v_k\in C^{\infty}(\overline{Q^{(k)}_{1, \un T, \bar T}}\b \overline{\Ga^{(k)}_{\un T, \bar T}})$  as  in (2) of  Lemma \ref{lemA:001b}, and let $r_0^{(k)}$ the constant in (2) of  Lemma \ref{lemA:001b} such that the inequalities (\ref{eqA:003})-(\ref{eqA:004}) hold for $(x, t)\in \overline{Q^{(k)}_{r_0^{(k)}, \un T, \bar T}}\b \overline{\Ga^{(k)}_{\un T, \bar T}}$. Set
\beq
r_0:=\min\{r_0^{(1)}, r_0^{(2)}, \cdots, r_0^{(l)}\}. \label{eq:r0}
\eeq
After shrinking the interval $[T_1, T_2]$ if necessary, we can assume that
\beq
 \Ga_{T_1, T_2}\subset \hat Q_{r_0, T_1, T_2}.\label{eq:Q}
\eeq
By (\ref{eq:r0})-(\ref{eq:Q}), we know that the inequalities (\ref{eqA:003})-(\ref{eqA:004}) hold for all functions $v_k$ and all $(x, t)\in {\hat Q_{r_0, \un T, \bar T}}\b {\Ga_{\un T, \bar T}}$.

\item[(4).]
 For any $\ee>0$ and $(x, t)\in  {\hat Q_{r_0, \un T, \bar T}}\b  {\Ga_{\un T, \bar T}}$, we define \beq  \td v(x, t)=\sum_{k=1}^l\,v_k(x, t),\quad \td w_{ \ee}(x, t)=2-\frac {\td v(x, t)}{\log \frac 1{\ee}},\quad \td \varphi_{ \ee}(x, t)= (H\circ \td  w_{ \ee})(x, t), \label{eqC:001}\eeq where $ H$ is  defined in (2) of Definition \ref{defi:001}. Note that $H\circ \td w_{ \ee}=0$ near each $\xi_k$ and this implies that $ \td \varphi_{ \ee}$ vanishes near $\Ga_{\un T, \bar T}$. Moreover, for any $(x, t)\in \hat Q_{r_0, \un T, \bar T}\b \Ga_{\un T, \bar T}$ we have
\beqn
&&\lim_{\ee\ri 0}\td \varphi_{ \ee}(x, t)=H(2),\quad \lim_{\ee\ri 0}|\Na \td\varphi_{ \ee}|(x, t)=0,\\
&& \pd {\td \varphi_{ \ee}}t+\Delta \td \varphi_{\ee}=H''\circ \td w_{\ee}|\Na \td w_{\ee}|^2. \label{eqB:203}
\eeqn
Let \beq \td \r(x, t)=e^{-\td v(x, t)}.\label{eqn:td_r}\eeq Then the inequalities (\ref{eqA:003}) imply that for any  $(x, t)\in {\hat Q_{r_0 \un T, \bar T}}\b {\Ga_{\un T, \bar T}}$
\beq
\r_1\r_2\cdots \r_l\leq \td \r(x, t)\leq (\r_1\r_2\cdots \r_l)^{\ga}. \label{eq:tdr}
\eeq
              \item[(5).]Under the above assumptions, for $\rho>0$ and  $h\in L^1(Q_{r_0, \un t, \bar t})$ we define
$$I(\rho; \un t, \bar t, h, r_0)=\int_{Q_{r_0, \un t, \bar t}\cap \{\rho\leq \td \r(x, t)\leq 1\}}\;\frac {h|\Na \td v|^2}{|\log \rho|^2},$$
where $\td \r(x, t)$ and $\td v(x, t)$ are the function defined in (4) above.

\item[(6). ] Assume that $\{\xi_1(t), \cdots, \xi_l(t)\}$ are disjoint on $[T_1, T_2]$.  We choose $\bar \rho>0$ such that  $Q_{\bar \rho, T_1, T_2}^{(i)}\cap Q_{\bar \rho, T_1, T_2}^{(j)}=\emptyset$ for any $1\leq i\neq j\leq l.$  For any $\rho\in (0, \bar \rho)$, $T_1\leq \un t<\bar t\leq T_2$ and $h\in L^1(Q^{(k)}_{1, \un t, \bar t})$, we define
\beq
I_{\xi_k}(\rho; \un t, \bar t, h, \bar \rho)=\frac 1{|\log \rho|^2}\int_{\un t}^{\bar t}\int_{\rho\leq \r_k(x, t)\leq \bar \rho}\,\frac {h}{\r_k(x, t)^2}.
\eeq
            \end{enumerate}

\end{defi}

The next result gives the $L^1$ estimate of the solution near the singularities when the  singular curves are not disjoint. 

\begin{lem} \label{lemA:003b}Under the same assumption as in Theorem \ref{theoA:main},       for any $(t_1, t_2)\subset (T_1, T_2)$ there exists $r_1=r_1(\rho_0, \Xi_0, l, t_1, t_2, T_1, T_2, \ga)>0$ such that
  \beq
  \|u\|_{L^1(Q_{r_1, t_1, t_2})}\leq C\|u\|_{L^1(K)},\label{eq:L002}
  \eeq where $C$ is a constant depending on $\|c\|_{L^{\infty}(Q_{1, T_1, T_2})}, \ga, \rho_0, \Xi_0, \si,  t_1, t_2, T_1, T_2 $ and $K$ is defined by $K=\overline{Q_{2r_1, T_1, T_2}}\b Q_{r_1, T_1, T_2}$. Moreover, we have
  \beq
  \sup_{\rho\in (0, \frac 12)}I(\rho; t_1, t_2, u, r_0)<+\infty.\label{eqB:207}
  \eeq
\end{lem}

\begin{proof}We divide the proof into several steps:

{\it Step 1.}
Without loss of generality, we can assume that $c(x, t)\geq  0$ on $Q_{1, T_1, T_2}$. In fact, let $u(x, t)$ be a solution of (\ref{eq:D004}). Then for any $k\in \RR$ the function $\td u(x, t)=u(x, t)e^{kt}$ satisfies the equation
$$\pd {\td u}t=\Delta \td u+(c+k)\td u, \quad \forall\;(x, t)\in
Q_{1, T_1, T_2}\b \Ga_{T_1, T_2}.$$
Since $c$ is locally bounded by the assumption, we can choose $k$ large such that $c+k\geq 0
$ on $Q_{1, T_1, T_2}$. Thus, it suffices to show Lemma \ref{lemA:003b} for $c(x, t)\geq 0.$\\

{\it Step 2.} Assume that $\{\xi_1(t), \cdots, \xi_l(t)\}$
are  around $(x_0, t_0)$ on $[T_1, T_2]$.
 Let $T_1<t_5<t_3<t_1<t_0<t_2<t_4<t_6<T_2$. We construct $v_k, r_0,  \td w_{ \ee}$ and $ \td \varphi_{ \ee}$  as in  Definition \ref{defi:002} by setting
 $$ \un T=t_5,\quad  \bar T=t_6,\quad \un t=t_3, \quad \bar t=t_4.$$
Assume that (\ref{eq:Q}) holds. Let $0<\dd<r_1<\frac {r_0}{2}$ and set
 $\un r=r_1, \bar r=2r_1$.  After shrinking $r_1$ and the interval $[T_1, T_2]$ if necessary, we assume that $Q_{2r_1, T_1, T_2}\subset \hat Q_{r_0, T_1, T_2}.$
We choose $t_7, t_8$ such that
$T_1<t_7<t_5<t_6<t_8<T_2$ and define the function
 \beq \phi=\phi(x, t; r_1, 2r_1,  t_5, t_6, t_7, t_8, T_1, T_2) \label{eqn:phi}\eeq as in Definition \ref{defi:002}. Then by  (\ref{eq:phigrad}) the function $\phi$ satisfies the properties
\beqn &&
\Supp(|\Na \phi|+|\p_t\phi|)\cap  Q_{2r_1, t_5, t_6} \subset  Q_{2r_1, t_5, t_6}\b Q_{r_1, t_5, t_6} \nonumber\\&\subset&\bigcup_{k=1}^l\,\Big\{(x, t)\in Q_{2r_1,  t_5, t_6}\;\Big|\;
 r_1\leq \r_k(x, t)\leq 2r_1,\;\; \r_i(x, t)\geq r_1, \;\forall\, i\neq k\Big\}. \label{eqn:102}
\eeqn
 Moreover,  we define the following functions on $\hat Q_{r_0, t_5, t_6}\b \Ga_{t_5, t_6}   $ \beqs V_k(x, t)&=&e^{-2v_k(x, t)},\quad V(x, t)=\sum_{k=1}^l\,V_k(x, t),\\ w(x, t)&=&\zeta(t)^{-1}V(x, t)-1,\quad \varphi_0(x, t)=\phi(x, t)\zeta(t)(H\circ w)(x, t), \eeqs where  $ \zeta=\zeta(t; t_1, t_2, t_3, t_4, \dd, r_1)$ is the function defined in (1) of Definition \ref{defi:001}.     By using the properties (\ref{eqA:006a})-(\ref{eqA:006}), for any $(x, t)\in  \hat Q_{r_0, t_5, t_6}\b \Ga_{t_5, t_6}   $ we have
\beqn
\sum_{k=1}^l\, \r_k(x, t)^{2}&\leq& V(x, t)\leq \sum_{k=1}^l\, \r_k(x, t)^{2\ga},  \label{eq:H001}\\
l&\leq &\pd Vt+\Delta V\leq 4\sum_{k=1}^l\, \r_k(x, t)^{2\ga-2}, \label{eq:H002}\\
|\Na V|&\leq&2\sum_{k=1}^l\, \r_k(x, t)^{2\ga-1}. \label{eq:H003}
\eeqn  Note that the function  $\varphi_0(x, t)$ vanishes near the point $(\xi_1(t_0), t_0)$, but $\varphi_0(x, t)$ may not be zero on $\Ga_{t_5, t_6}$. The function $\tilde \varphi_{\ee}$ defined in Definition
\ref{defi:002}
 vanishes near $\Ga_{t_5, t_6}$, but it doesn't satisfy the inequality (\ref{eq:H002}) and the inequality (\ref{eqB:100xx}). Therefore, the argument of Lemma \ref{lem:A003b} doesn't work any more.
\\

{\it Step 3.}
  Direct calculation as in the proof of Lemma \ref{lem:A003b}, for any $(x, t)\in Q_{2r_1, t_5, t_6}$ we have
\beqn \varphi_0(x, t)&\leq& \phi(x, t)V(x, t)H'\circ w,\label{eq:c006}\\
\pd {\varphi_0}t+\Delta \varphi_0&\geq& \Big(l -2\|\partial_t\zeta\|_{L^{\infty}(\RR)}\Big)\phi H'\circ w-C(c_{\phi, K},  \ga, r_1)\chi_{K},\label{eq:c005}
\eeqn where $K$ and $ c_{\phi, K}$ are  defined by
\beqn K&=&\Supp(|\Na \phi|+|\p_t\phi|)\cap Q_{2r_1, t_5, t_6}, \label{eq:K1}\\
c_{\phi, K}&=&\sup_{K}(|\p_t\phi|+|\Delta \phi|+|\Na \phi|). \label{eq:cK1}\eeqn
Let $\varphi=\varphi_0\td \varphi_{\ee}\in C_{0}^{\infty}({Q_{2r_1, t_5, t_6}}\backslash \Ga_{t_5, t_6})$. Then we have
\beqs &&
\pd {\varphi}t+\Delta\varphi\\&=&\Big(\pd {\varphi_0}t+\Delta\varphi_0\Big)\td \varphi_{\ee}+\Big(\pd {\td\varphi_{\ee}}t+\Delta\td\varphi_{\ee}\Big)\varphi_0+2\langle
\Na\td\varphi_{\ee}, \Na \varphi_0\rangle\\
&=&\Big(\pd {\varphi_0}t+\Delta\varphi_0\Big)\td \varphi_{\ee}+H''\circ \td \oo_{\ee}|\Na \td \oo_{\ee}|^2\varphi_0+2\phi H'\circ \td \oo_{\ee}H'\circ w\langle \Na \td \oo_{\ee}, \Na V\rangle\\&&+2\zeta H'\circ \td \oo_{\ee} H\circ w\langle\Na \td \oo_{\ee}, \Na \phi \rangle\\
&\geq& \Big((l-2\|\partial_t\zeta\|_{L^{\infty}(\RR)})\phi H'\circ w-C(c_{\phi, K}, \ga, r_1)\chi_{K}\Big)\td \varphi_{\ee}\\&&-2\phi H'\circ w |\Na V|\cdot |\Na \td \oo_{\ee}|\chi_{\{ \td \oo_{\ee}>0\}}-2VH'\circ w H'\circ \td \oo_{\ee}  |\Na \td \oo_{\ee}|\cdot |\Na \phi|,
\eeqs where we used (\ref{eqn:003}) (\ref{eqB:203}) (\ref{eq:c005})  and the definition of $w$.
Combining this with (\ref{eqB:002xx}) and using the assumption $c(x, t)\geq 0$, we have
\beqn 0&\leq&
\int_{Q_{2r_1, t_5, t_6}}\; cu\varphi =-\int_{Q_{2r_1, t_5, t_6}}\; u\Big(\pd {\varphi}t+\Delta \varphi\Big)\nonumber\\
&\leq& -\Big(l-2\|\partial_t\zeta\|_{L^{\infty}(\RR)} \Big)\int_{Q_{2r_1, t_5, t_6}}\;u\,\phi\td \varphi_{\ee} H'\circ w +C(c_{\phi, K},  \ga, r_1)\int_{K}\;u\, \td \varphi_{ \ee}\nonumber\\
&& +2\|\Na V\|_{L^{\infty}({Q_{2r_1, t_5, t_6}})}\int_{Q_{2r_1, t_5, t_6}}\; u\phi H'\circ w   |\Na \td \oo_{\ee}|\chi_{\{  \td \oo_{\ee}>0\}}\nonumber\\&&+2\|V\|_{L^{\infty}({Q_{2r_1, t_5, t_6}})}\int_{Q_{2r_1, t_5, t_6}}\,u  H'\circ w H'\circ \td \oo_{\ee}  |\Na \td \oo_{\ee}|\cdot |\Na \phi|.
\nonumber
\eeqn
Therefore, we have
\beqn &&
\Big(l-2\|\partial_t\zeta\|_{L^{\infty}(\RR)} \Big)\int_{Q_{2r_1, t_5, t_6}}\,u\phi  H'\circ w\,\td \varphi_{ \ee}\nonumber\\&\leq&2\|\Na V\|_{L^{\infty}({Q_{2r_1, t_5, t_6}})}\int_{Q_{2r_1, t_5, t_6}}\; u\phi H'\circ w  |\Na \td \oo_{\ee}|\chi_{\{  \td \oo_{\ee}>0\}}\nonumber\\
&&+C(c_{\phi, K},  \ga, r_1)\int_{K}\;u\, \td \varphi_{ \ee}+2\|V\|_{L^{\infty}({Q_{2r_1, t_5, t_6}})}\int_{Q_{2r_1, t_5, t_6}} u |\Na \phi|\cdot |\Na \td \oo_{\ee}|\nonumber\\
&\leq&2\|\Na V\|_{L^{\infty}({Q_{2r_1, t_5, t_6}})}\int_{Q_{2r_1, t_5, t_6}}\; u\phi H'\circ w \chi_{\{  \td \oo_{\ee}>0\}} \nonumber\\
&&+2\|\Na V\|_{L^{\infty}({Q_{2r_1, t_5, t_6}})}\int_{Q_{2r_1, t_5, t_6}}\; u\phi H'\circ w  |\Na \td \oo_{\ee}|^2\chi_{\{  \td \oo_{\ee}>0\}}\nonumber
\\
&&+C(c_{\phi, K},   \ga, r_1)\int_{K}\;u\, \td \varphi_{ \ee}+2\|V\|_{L^{\infty}({Q_{2r_1, t_5, t_6}})}\int_{Q_{2r_1, t_5, t_6}} u |\Na \phi|\cdot |\Na \td \oo_{\ee}| \label{eqE:004}
\eeqn
The main difficulty is to estimate the integral
\beq \int_{Q_{2r_1, t_5, t_6}}\; u\phi H'\circ w  |\Na \td \oo_{\ee}|^2\chi_{\{  \td \oo_{\ee}>0\}} \label{eqn:004}\eeq
on the right hand side of (\ref{eqE:004}).\\

{\it Step 4.} We estimate the integral (\ref{eqn:004}).
For any $\rho\in (0, \frac 12)$,  we define the functions
$$\bar w_{\rho}(x, t)=\frac 13\Big(2-\frac {\td v(x, t)}{\log(\frac 1{\rho})}\Big),\quad \psi(x, t) =\phi(x, t)H\circ \bar w_{\rho}\in C_0^{\infty}(Q_{2r_1, t_5, t_6}\b \Ga_{t_5, t_6}),$$ where $\td v$ is the function defined in (\ref{eqC:001}).
Note that $\bar w_{\rho}(x, t)$ satisfies $\partial_t \bar w_{\rho}+\Delta  \bar w_{\rho}=0.$
 Direct calculation shows that
\beq
\pd {\psi}t+\Delta\psi=\phi |\Na \bar w_{\rho}|^2 H''\circ \bar w_{\rho}+\Big(\pd {\phi}t+\Delta \phi\Big)H\circ \bar w_{\rho}+2\langle \Na \phi, \Na \bar w_{\rho}\rangle H'\circ \bar w_{\rho}.\nonumber
\eeq
Since $u$ satisfies
\beq
-\int_{Q_{2r_1, t_5, t_6}}\; u\Big(\pd {\psi}t+\Delta \psi\Big)=\int_{Q_{2r_1, t_5, t_6}}\; cu \psi\geq 0\; , \nonumber
\eeq
we have
\beqn &&
\int_{Q_{2r_1, t_5, t_6}}\,u\phi  |\Na \bar w_{\rho}|^2 H''\circ \bar w_{\rho}\nonumber\\
&\leq& -\int_{Q_{2r_1, t_5, t_6}}\,u \Big(\Big(\pd {\phi}t+\Delta \phi\Big)H\circ \bar  w_{\rho}+2\langle \Na \phi, \Na \bar w_{\rho}\rangle H'\circ \bar w_{\rho}\Big)\; . \label{eqC:002}
\eeqn
We estimate each term of (\ref{eqC:002}).
Note that
\beqn&&\{(x, t)\in {Q_{2r_1, t_5, t_6}}\,|\, H''\circ \bar w_{\rho}\geq \min_{\frac 13\leq z\leq \frac 23}H''(z)\}\nonumber\\&\supset& \{(x, t)\in {Q_{2r_1, t_5, t_6}}\,|\, \frac 13\leq \bar \oo_{\rho}\leq \frac 23\}\nonumber\\&
=&\{(x, t)\in {Q_{2r_1, t_5, t_6}}\,|\, 1\leq \td \oo_{\rho}\leq2\}\nonumber\\
&=&\{(x, t)\in {Q_{2r_1, t_5, t_6}}\,|\, \rho\leq \td \r(x, t)\leq 1\}, \label{eq:C101} \eeqn
where $\td \oo_{\rho}$ is defined in (\ref{eqC:001}) and $\td \r(x, t)=e^{-\td v(x, t)}$.
Thus, the  left-hand side of (\ref{eqC:002}) satisfies the inequality
\beqn
\int_{Q_{2r_1, t_5, t_6}}\,u\phi |\Na \bar w_{\rho}|^2 H''\circ \bar w_{\rho} \; &\geq&  C \int_{{Q_{2r_1, t_5, t_6}}}\,\frac {u \phi |\Na \td v|^2\chi_{\{1\leq \td \oo_{\rho}\leq 2\}}}{|\log \rho|^2}\,\nonumber\\&=&C \int_{{Q_{2r_1, t_5, t_6}}}\,u \phi|\Na \td \oo_{\rho}|^2\chi_{\{1\leq \td \oo_{\rho}\leq 2\}}, \label{eqC:003xx}
\eeqn where $C$ is a universal constant.
We choose $2r_1<1$ and by (\ref{eq:tdr}) we have  $\td \r(x, t)< 1$ on $\Supp(\phi)\cap Q_{2r_1, t_5, t_6}$. Thus,  on $\Supp(\phi)\cap Q_{2r_1, t_5, t_6}$ we have
\beq \bar \oo_{\rho}=\frac 13\Big(2-\frac {\td v(x, t)}{\log(\frac 1{\rho})}\Big)\leq \frac 23, \quad H\circ \bar \oo_{\rho}\leq \bar \oo_{\rho}\leq \frac 23. \label{eq:C007bb}  \eeq
Combining this with (\ref{eqC:002}), the first term of the right-hand side of (\ref{eqC:002}) satisfies the inequality
\beqn
-\int_{Q_{2r_1, t_5, t_6}}\,u \Big(\pd {\phi}t+\Delta \phi\Big)H\circ \bar  w_{\rho}&\leq&C(c_{\phi, K})\int_{K}\,u,\label{eqn:101}
\eeqn where $K$ and $c_{\phi, K} $ are given by (\ref{eq:K1})-(\ref{eq:cK1}).
Note that by (\ref{eq:phigrad}) for any $i$ we have $\r_i(x, t)\geq r_1$ on $\Supp(|\Na \phi|)\cap Q_{2r_1, t_5, t_6}$. Combining this with (\ref{eqA:004}),  we have
\beqs
|\Na \td v|^2&\leq& l \sum_{k=1}^l\; |\Na v_k|^2\leq l \sum_{k=1}^l\;\frac 1{\r_k^2}\leq \frac {l^2}{r_1^2},\quad \forall\; (x, t)\in \Supp(|\Na \phi|)\cap Q_{2r_1, t_5, t_6}.
\eeqs
Thus, when $\rho\in (0, \frac 12)$ we have
\beq |\Na \bar \oo_{\rho}|=\frac 23 \frac {|\Na \td v|}{\log \frac 1{\rho}}\leq \frac {2l^2}{(3\log 2)\,r_1^2}, \quad \forall\; (x, t)\in \Supp(|\Na \phi|)\cap Q_{2r_1, t_5, t_6}. \label{eq:K006} \eeq
This implies that the second term of the right-hand side of (\ref{eqC:002}) satisfies
\beq
-\int_{Q_{2r_1, t_5, t_6}}\,2u \langle \Na \phi, \Na \bar w_{\rho}\rangle H'\circ \bar w_{\rho}\leq C(c_{\phi, K}, r_1, l) \int_{K}\,u . \label{eq:C100}
\eeq
Let  $\rho=\ee^2$.  Note that $\td \r(x, t)\leq 1$ on $\Supp(\phi)\cap Q_{2r_1, t_5, t_6}$. By (\ref{eq:C101}) we have
\beqn &&
\{(x, t)\in {Q_{2r_1, t_5, t_6}}\,|\, \td \oo_{\ee}>0 \}\cap \Supp(\phi)\nonumber\\&=&\{(x, t)\in {Q_{2r_1, t_5, t_6}}\,|\,  \td \r(x, t)> \ee^2\} \cap \Supp(\phi)\nonumber\\&
=&\{(x, t)\in {Q_{2r_1, t_5, t_6}}\,|\, \rho\leq \td \r(x, t)\leq 1\} \cap \Supp(\phi), \nonumber\\
&=&\{(x, t)\in {Q_{2r_1, t_5, t_6}}\,|\, 1\leq \td \oo_{\rho}\leq2\}\cap \Supp(\phi). \label{Z:001}
\eeqn
Combining (\ref{eqn:101})(\ref{eq:C100}) with (\ref{eqC:002}), we have
\beq
\int_{Q_{2r_1, t_5, t_6}}\,u\phi  |\Na \bar w_{\rho}|^2 H''\circ \bar w_{\rho}\leq
C(c_{\phi, K}, r_1, l)\int_{K}\,u.
\eeq
This together with (\ref{eqC:003xx}) implies that
\beqn
\int_{{Q_{2r_1, t_5, t_6}}}\,u \phi|\Na \td \oo_{\rho}|^2\chi_{\{1\leq \td \oo_{\rho}\leq 2\}}&\leq& C\int_{Q_{2r_1, t_5, t_6}}\,u\phi  |\Na \bar w_{\rho}|^2 H''\circ \bar w_{\rho}\nonumber\\&\leq&C(c_{\phi, K}, r_1, l)\int_{K}\,u. \label{Z:002}
\eeqn
Thus,
by (\ref{Z:002}) and (\ref{Z:001}) we have
\beqn
\int_{Q_{2r_1, t_5, t_6}\cap \{\rho<\td \r<1\}}\, {u\phi|\Na \td \oo_{\rho}|^2}&=& \int_{{Q_{2r_1, t_5, t_6}}}\,u \phi|\Na \td \oo_{\rho}|^2\chi_{\{1\leq \td \oo_{\rho}\leq 2\}}\nonumber\\
&\leq& C(c_{\phi}, r_1, l)\int_{K}\,u.\label{eq:L001}\eeqn
Moreover,  we have the estimate for the integral (\ref{eqn:004})
\beqn
\int_{Q_{2r_1, t_5, t_6}}\; u\phi H'\circ w  |\Na \td \oo_{\ee}|^2\chi_{\{  \td \oo_{\ee}>0\}}&\leq&  \int_{Q_{2r_1, t_5, t_6}\cap \{\td \oo_{\ee}>0\}}\,\frac {u\phi|\Na \td v|^2}{|\log \ee|^2}\nonumber\\
&=&
4\int_{Q_{2r_1, t_5, t_6}\cap \{1\leq \td \oo_{\rho}\leq2\}}\, {u\phi|\Na \td \oo_{\rho}|^2}\nonumber\\
&\leq&C(c_{\phi, K}, r_1, l)\int_{K}\,u. \label{eqE:006}
\eeqn \\

{\it Step 5.} Now we turn back to the inequality (\ref{eqE:004}).
Moreover, by (\ref{eq:K006}) we have
\beq
 \int_{Q_{2r_1, t_5, t_6}} u |\Na \phi|\cdot |\Na \td \oo_{\ee}|= 2\int_{Q_{2r_1, t_5, t_6}} u |\Na \phi|\cdot \frac {|\Na \td v|}{\log \frac 1{\rho}}\leq C(c_{\phi, K}, l, r_1)\int_{K } u. \label{eq:K007}
 \eeq
Combining (\ref{eqE:004}), (\ref{eq:K007}) with (\ref{eqE:006}), we have
\beqn &&
\Big(l-2\|\partial_t\zeta\|_{L^{\infty}(\RR)} \Big)\int_{Q_{2r_1, t_5, t_6}}\,u\phi  H'\circ w\,\td \varphi_{ \ee}\nonumber\\
&\leq&2\|\Na V\|_{L^{\infty}({Q_{2r_1, t_5, t_6}})}\int_{Q_{2r_1, t_5, t_6}}\; u\phi H'\circ w \,\chi_{\{  \td \oo_{\ee}>0\}} \nonumber\\&&+C(c_{\phi, K}, r_1, l)\Big(\|\Na V\|_{L^{\infty}({Q_{2r_1, t_5, t_6}})}+\|V\|_{L^{\infty}({Q_{2r_1, t_5, t_6}})}\Big)\int_{K}\,u     \nonumber
\\
&&+C(c_{\phi, K},  \ga, r_1)\int_{K}\;u\, \td \varphi_{ \ee}. \label{eq:C102}
\eeqn
Since all singular curves are disjoint on $Q_{2r_1, t_5, t_6}\cap \{w\geq 0\}$ by our assumption,  by Lemma \ref{lem:A003b} $u$ is integrable on $Q_{2r_1, t_5, t_6}\cap \{w\geq 0\}$.
Taking $\ee\ri 0$ in (\ref{eq:C102}) and using the dominated convergence theorem, we have
\beqn &&
\Big(l-2\|\partial_t\zeta\|_{L^{\infty}(\RR)} \Big)\int_{Q_{2r_1, t_5, t_6}}\,u\phi  H'\circ w\,\nonumber\\
&\leq&2\|\Na V\|_{L^{\infty}({Q_{2r_1, t_5, t_6}})}\int_{Q_{2r_1, t_5, t_6}}\; u\phi H'\circ w+C(c_{\phi, K},  \ga, r_1)\int_{K}\;u \, \nonumber\\&&+C(c_{\phi, K}, r_1, l)\Big(\|\Na V\|_{L^{\infty}({Q_{2r_1, t_5, t_6}})}+\|V\|_{L^{\infty}({Q_{2r_1, t_5, t_6}})}\Big)\int_{K}\,u   .  \label{eq:C103}
\eeqn
It follows that
\beqn &&
\Big(l-2\|\partial_t\zeta\|_{L^{\infty}(\RR)}-2\|\Na V\|_{L^{\infty}({Q_{2r_1, t_5, t_6}})} \Big)\int_{Q_{2r_1, t_5, t_6}}\,u\phi  H'\circ w\,\nonumber\\
&\leq&C(c_{\phi, K},  \ga, r_1)\int_{K}\;u. \label{eq:C103b}
\eeqn
By (\ref{eq:zeta}) and (\ref{eq:H003}), we choose $r_1$ small such that $$l-2\|\partial_t\zeta\|_{L^{\infty}(\RR)}-2\|\Na V\|_{L^{\infty}({Q_{2r_1, t_5, t_6}})}>\frac 12l.$$
Combining this with (\ref{eq:C103b}), we have
\beq
\int_{Q_{2r_1, t_5, t_6}}\,u\phi  H'\circ w\,
\leq C(c_{\phi, K}, \ga, r_1, l)\int_{K}\,u.   \label{eq:C001a}
\eeq
Note that the function $H'\circ w$ converges to $1$ on $Q_{r_1, t_1, t_2}\b \Ga_{t_1, t_2}$ as $\dd\ri 0$. Thus, taking $\dd\ri 0$ in (\ref{eq:C001a})  we have
\beq
\int_{Q_{r_1, t_1, t_2}}\,u
\leq C(c_{\phi, K},   \ga, r_1, l)\int_{K}\,u,   \label{eq:C001b}
\eeq which implies (\ref{eq:L002}).
Note that (\ref{eq:L001}) implies (\ref{eqB:207}) since
\beqn
I(\rho; t_1, t_2, u, r_0)&=&\int_{Q_{r_0, t_1, t_2}\cap \{\rho\leq \td \r\leq 1\}}\;\frac {u |\Na \td v|}{|\log \rho|^2} \nonumber\\&= &\int_{Q_{r_1, t_5, t_6}\cap \{\rho<\td \r<1\}}\, {u |\Na \td \oo_{\rho}|^2}+ \int_{Q_{r_0, t_5, t_6}\b Q_{r_1, t_5, t_6} }\, {u |\Na \td \oo_{\rho}|^2}\nonumber\\
&\leq&\int_{Q_{2r_1, t_5, t_6}\cap \{\rho<\td \r<1\}}\, {u\phi|\Na \td \oo_{\rho}|^2}+C(l, r_1)\int_{Q_{r_0, t_5, t_6}\b Q_{r_1, t_5, t_6} }\, u\nonumber\\
&\leq&C(c_{\phi}, r_1, l)\int_{K}\,u+C(l, r_1)\int_{Q_{r_0, t_5, t_6}\b Q_{r_1, t_5, t_6} }\, u<+\infty. \label{Z:003}
\eeqn
 The lemma is proved.

\end{proof}

As a byproduct of the above proof, we have the following result.

\begin{lem}\label{lem:measure2} Under the assumption of Lemma \ref{lem:A003b}, for the singular curve $\xi:[T_1, T_2]\ri \Si$ we have
\beq
\sup_{\rho\in (0, \frac 12)}I_{\xi}(\rho; t_1, t_2, u, 2r_1 )<+\infty. \label{lem:measure2:001}
\eeq

\end{lem}
\begin{proof} (\ref{lem:measure2:001}) follows directly from the inequality (\ref{Z:003}) and Step 4 of the proof of Lemma \ref{lemA:003b} by choosing $l=1$.

\end{proof}
By using Lemma \ref{lem:A003b}, Lemma \ref{lem:measure2} and following the same arguments as in \cite{[KT]}, we have the following results when  the singular curves
are disjoint.

\begin{lem} \label{lemA:003}(cf. \cite{[KT]}) Under the same assumption as in Theorem \ref{theoA:main}, if we assume that  $\{\xi_1(t), \cdots, \xi_l(t)\}$
are disjoint on $[T_1, T_2]$ and $\bar \rho$ is the constant  in (6) of Definition \ref{defi:002}, then  we have
\begin{enumerate}
   \item[(1).] For each $\xi_k$ and $(t_1, t_2)\subset [T_1, T_2]$, the mapping $\cJ_k: C_0^{\infty}(Q_{\bar \rho, t_1, t_2}^{(k)})\ri \RR$
  $$\cJ_k(f)=\int_{Q_{\bar \rho, t_1, t_2}^{(k)}}\,\Big(u(-\pd {f}t-\Delta f)-cuf\Big)\,dxdt  $$
defines a distribution whose support is contained in $\Ga^{(k)}_{t_1, t_2}$, and satisfies
\beq
|\cJ_k(f)|\leq C(\sup_{\Ga^{(k)}_{t_1, t_2}}|f|)\liminf_{\rho\ri 0}I_{\xi_k}(\rho; t_1, t_2, u, \bar \rho),\label{eqC:005}
\eeq where $C$ is a universal constant. Here $I_{\xi_k}(\rho; t_1, t_2, u, \bar \rho)$ is defined in (6) of Definition \ref{defi:002} and it is finite by Lemma \ref{lem:measure2}.

\item[(2).] There exists  linear functionals $\{\mu_1, \cdots, \mu_l\}$ with
each $\mu_k\in (C_0((T_1, T_2)))'$ such that for all $\varphi\in C_0^{\infty}(Q_{1, T_1, T_2}),$
\beq
\int_{Q_{1, T_1, T_2}}\,u(-\pd {\varphi}t-\Delta \varphi)=\int_{Q_{1, T_1, T_2}}\;cu\varphi
+\sum_{k=1}^l\int_{(T_1, T_2)}\,\varphi(\xi_k(t), t)d\mu_k(t). \label{eqA:001}
\eeq The identity (\ref{eqA:001}) can be rewritten as
\beq
\pd ut-\Delta u=cu+\sum_{k=1}^l\,\dd_{\xi_k}\otimes \mu_k, \quad \hbox{in}\quad \cD'(Q_{1, T_1, T_2}).
\eeq

\item[(3).] Let $\mu_k$ be one of  the measures in (2). For any $\psi\in C_0^{\infty}((T_1, T_2))$ with $\Supp(\psi)\subset (t_1, t_2)$, we have
\beq
\int_{T_1}^{T_2}\,\psi \,d\mu_k=2\lim_{\rho\ri 0}\frac 1{|\log\rho|^2} \,\int_{Q^{(k)}_{\bar \rho, T_1, T_2}}\,|\Na v_k|^2\chi_{\{ v_k\leq |\log \rho|\}} \psi u.\label{eqG:002}
\eeq

\item[(4).] Each measure $\mu_k$ obtained in (3) is positive.

\end{enumerate}

\end{lem}

\begin{proof} Since $\{\xi_1(t), \cdots, \xi_l(t)\}$
are disjoint on $[t_1, t_2]$, we can consider each $\xi_k$ as in  \cite{[KT]}.
After replacing the function $g$ in $(4.21)$ of \cite{[KT]} by the function $c(x, t)$, we know that  part (1) follows directly from Lemma 4.4 of \cite{[KT]}. Part (2) follows from the proof of Theorem 2.1 of \cite{[KT]}(See Page 7303 of \cite{[KT]}), (3) follows from Lemma 5.2 of \cite{[KT]} and (4) follows from the non-negativity of the right-hand side of (\ref{eqG:002}). Since the proof is exactly the same as in \cite{[KT]}, we omit the details here.

\end{proof}

When the singular curves are around $(x_0, t_0)$, the measures $\mu_k$ constructed in Lemma \ref{lemA:003} may blow up as $t\ri t_0$. The next result shows that $\mu_k$ is actually bounded when $t$ is close to $t_0.$

\begin{lem}\label{lem:measure}The same assumption as in Theorem \ref{theoA:main}. Suppose that the singular curves $\{\xi_1(t), \cdots, \xi_l(t)\}$ are around $(x_0, t_0)$ on $[t_1, t_2]$ as in Definition \ref{defi:around}.
Define the measure $\mu$ on $(t_1, t_2)$ by
\beq
\int_{t_1}^{t_2}\,\psi \,d\mu=\lim_{\rho\ri 0}\frac 2{|\log\rho|^2} \,\int_{Q_{1, t_1, t_2}}\,|\Na \td v|^2\chi_{\{ \td v\leq |\log \rho|\}} \psi u\, d\vol \, dt,\label{eq:G005}
\eeq where the right-hand side is finite by (\ref{eqB:207}).  Here $\td v$ is the function defined by (\ref{eqC:001}).
Then $\mu\in (C_0((t_1, t_2)))'$ and for each $\xi_k$ the measure $\mu_k$ obtained by Lemma \ref{lemA:003} satisfies
$$0\leq \ga^4\,\mu_k(t)\leq \mu(t), \quad \forall\,t\in (t_1, t_0)\cup (t_0, t_2),$$
where $\ga\in (\frac 12, 1)$ is the constant chosen in Lemma \ref{lemA:001b}.
\end{lem}

\begin{proof} Since $\{\xi_1(t), \cdots, \xi_l(t)\}$ are around $(x_0, t_0)$ on $(t_1, t_2)$, by Definition \ref{defi:around} we can assume that $\{\xi_1(t), \cdots, \xi_{l'}(t)\}$ are disjoint for some $l'\leq l$ on $[t_1, t_0)$ and
\beq
\xi_{l'}(t)=\xi_{l'+1}(t)=\cdots=\xi_{l}(t),\quad \forall\; t\in (t_1, t_0).
\eeq
 Let $r_0$ be the constant defined by (\ref{eq:r0}). After shrinking $(t_1, t_2)$ if necessary, we can assume that $Q_{r_0', t_1, t_2}\subset \hat Q_{r_0, t_1, t_2}$ for some $r_0'>0$.  Let  $\rho_1>0$ be the constant such that for any $(x, t)\in Q_{r_0', t_1, t_2}$ and $1\leq i\leq l'$ we have $\r_i(x, t)\leq \rho_1.$
Since for any $\dd>0$ the curves $\{\xi_1(t), \cdots, \xi_{l'}(t)\}$ are disjoint on $[t_1, t_0-\dd]$, we define
\beq
d_{\dd}:=\min\{d_g(\xi_i(t), \xi_j(t))\;|\;1\leq i\neq j\leq l', t\in [t_1, t_0-\dd]\,\}>0.
\eeq
Let $\al_0=\frac {d_{\dd}}2$ and  $(x, t)\in Q_{r_0', t_1, t_0-\dd}$.  By the choice of $\al_0$, if $\r_k(x, t)<\al_0$, then we have
\beq
\r_i(x, t)\geq \al_0, \quad \forall\; i\neq k.\label{Z:007}
\eeq
For any $\rho_2>0$, we can find some integer $k\in [1, l']$ such that   if $t\in [t_1, t_0-\dd]$  and $\r_1 \r_2\cdots \r_{l'}\leq \rho_2$ we have
\beq \r_i\geq \al_0,\quad \forall\; i\neq k,\quad \hbox{and}\quad \r_k\leq \frac {\rho_2}{\al_0^{l'-1}}. \label{Z:006}\eeq
We choose $\rho_2$  such that
\beq
\frac {\rho_2}{\rho_1^{l'-1}}=\al_0.\label{Z:010}
\eeq
By  (\ref{Z:006})   for any $k\in \{1, 2, \cdots, l'\}$ and $\rho\in (0, \rho_2)$,
\beqs
\Big\{(x, t)\in Q_{r_0', t_1, t_0-\dd}\;\Big|\;\td v(x, t)\leq \log \frac 1{\rho}\Big\}&\supseteq& \Big\{(x, t)\in Q_{r_0', t_1, t_0-\dd}\;\Big|\;\r_1\r_2\cdots \r_{l'}\geq \rho\Big\}\\
&\supseteq& \Big\{(x, t)\in Q_{r_0', t_1, t_0-\dd}\;\Big|\;\rho_2\geq \r_1\r_2\cdots \r_{l'}\geq \rho \Big\}\\
&=&\bigcup_{k=1}^{l'}\;\Om_{k, \rho},
\eeqs where $\Om_{k, \rho}$ is defined by
\beqn \Om_{k, \rho}&:=&\Big\{ (x, t)\in Q_{r_0', t_1, t_0-\dd} \;\Big|\; \al_0\leq \r_i\leq \rho_1,\;\forall\; i\neq k,  \frac {\rho}{\al_0^{{l'}-1}}\leq \r_k\leq \frac {\rho_2}{\rho_1^{{l'}-1}}\Big\} \nonumber \\
&=& \Big\{ (x, t)\in Q_{r_0', t_1, t_0-\dd} \;\Big|\; \al_0\leq \r_i\leq \rho_1,\;\forall\; i\neq k,  \frac {\rho}{\al_0^{{l'}-1}}\leq \r_k\leq \al_0\Big\}. \label{Z:011}\eeqn
Note that we used (\ref{Z:010}) in the equality of (\ref{Z:011}).
By the definition of $r_0'$ and Lemma \ref{lemA:001b}, for any $(x, t)\in \Om_{k, \rho }$ with $1\leq k\leq l'-1$ we have
\beqs
|\Na \td v|^2&\geq& |\Na v_k|^2-\sum_{i\neq k}|\Na v_i|^2\\
&\geq& \frac {\ga^2}{ \r_k^2}-\sum_{i\neq k}\,\frac 1{\r_i^2}\geq \frac {\ga^2}{ \r_k^2}-\frac {{l}-1}{\al_0^2},
\eeqs
and for any $(x, t)\in \Om_{k, \rho }$ with $l'\leq k\leq l$ we have
\beqs
|\Na \td v|^2&\geq& (l-l'+1)|\Na v_k|^2-\sum_{i=1}^{l'-1}|\Na v_i|^2\\
&\geq&(l-l'+1) \frac {\ga^2}{ \r_k^2}-\sum_{i=1}^{l'-1}\,\frac 1{\r_i^2}\geq (l-l'+1)\frac {\ga^2}{ \r_k^2}-\frac {{l'}-1}{\al_0^2}.
\eeqs
Consequently, by (\ref{eq:G005}) for any $k\in \{1, 2, \cdots, {l'}-1\}$ we have
\beqn
\int_{t_1}^{t_0-\dd}\,\psi \,d\mu&\geq& \lim_{\rho\ri 0}\frac 2{|\log\rho|^2} \,\int_{ \Om_{k, \rho }}\,|\Na \td v|^2 \psi u\,\nonumber\\
&\geq& \lim_{\rho\ri 0}\frac 2{|\log\rho|^2} \,\int_{ \Om_{k, \rho }}\,\Big(\frac {\ga^2}{ \r_k^2}-\frac {{l}-1}{\al_0^2}\Big)\psi u\,\nonumber\\
&=&\ga^2\lim_{\rho\ri 0}\frac 2{|\log\rho|^2} \,\int_{ \Om_{k, \rho }}\,\frac {\psi u}{ \r_k^2} \label{Z:004}
\eeqn and for $k\in \{{l'}, \cdots, l\}$ we have
\beqn
\int_{t_1}^{t_0-\dd}\,\psi \,d\mu &\geq& (l-l'+1)\ga^2\lim_{\rho\ri 0}\frac 2{|\log\rho|^2} \,\int_{ \Om_{k, \rho }}\,\frac {\psi u}{ \r_k^2}\nonumber\\
&\geq&\ga^2\lim_{\rho\ri 0}\frac 2{|\log\rho|^2} \,\int_{ \Om_{k, \rho }}\,\frac {\psi u}{ \r_k^2}.   \label{Z:004a}
\eeqn
On the other hand, taking $\td \rho^{\frac 1{\ga}}=\frac {\rho}{\al_0^{{l'}-1}}$ and using  (\ref{eqG:002}) we have
\beqn
\int_{t_1}^{t_0-\dd}\,\psi \,d\mu_k&=&2\lim_{\td \rho\ri 0}\frac 1{|\log\td \rho|^2} \,\int_{Q^{(k)}_{\bar \rho, t_1, t_0-\dd}}\,|\Na v_k|^2\chi_{\{ v_k\leq |\log\td \rho|\}} \psi u \nonumber\\
&=&2\lim_{\td \rho\ri 0}\frac 1{\ga^2|\log \rho|^2} \,\int_{Q^{(k)}_{\al_0, t_1, t_0-\dd}}\,|\Na v_k|^2\chi_{\{ v_k\leq |\log\td \rho|\}} \psi u\nonumber\\
&\leq&\frac 2{\ga^2}\lim_{ \rho\ri 0}\frac 1{|\log \rho|^2} \,\int_{Q^{(k)}_{\al_0, t_1, t_0-\dd}\cap \{\frac {\rho}{\al_0^{{l'}-1}}\leq \r_k\}}\,\frac { \psi u}{\r_k^2}, \label{Z:005}
\eeqn where we used the fact that
$\{v_k\leq \log \frac 1{\td \rho}\}\subseteq \{\td \rho^{\frac 1{\ga}}\leq \r_k\}$.
Note that
\beqn &&
Q^{(k)}_{\al_0, t_1, t_0-\dd}\cap \{\frac {\rho}{\al_0^{{l'}-1}}\leq \r_k\}\nonumber\\
&=&\Big\{(x, t)\in Q_{\al_0, t_1, t_0-\dd}\;\Big|\;\frac {\rho}{\al_0^{{l'}-1}}\leq \r_k\leq \al_0\Big\}\nonumber\\
&=&\Big\{(x, t)\in Q_{r_0', t_1, t_0-\dd}\;\Big|\;\frac {\rho}{\al_0^{{l'}-1}}\leq \r_k\leq \al_0,\; \al_0\leq \r_i\leq \rho_1,\; \forall\; i\neq k \Big\}= \Om_{k, \rho}, \label{Z:012}
\eeqn where we used (\ref{Z:007}) and (\ref{Z:011}). Combining (\ref{Z:012}) with (\ref{Z:005}), we have
\beq
\int_{t_1}^{t_0-\dd}\,\psi \,d\mu_k\leq \frac 2{\ga^2}\lim_{ \rho\ri 0}\frac 1{|\log \rho|^2} \,\int_{\Om_{k, \rho}   }\,\frac { \psi u}{\r_k^2}.\label{Z:013}
\eeq
The inequalities (\ref{Z:004})-(\ref{Z:004a}) and (\ref{Z:013}) implies that
\beq
\int_{t_1}^{t_0-\dd}\,\psi \,d\mu\geq \ga^4\int_{t_1}^{t_0-\dd}\,\psi \,d\mu_k.
\eeq
Thus, we have
\beq
0\leq \ga^4\mu_k\leq \mu,\quad \forall\; t\in (t_1, t_0).
\eeq  Similarly, we can consider the case when $\{\xi_1(t), \cdots, \xi_{l'}(t)\}$ are disjoint for some $l'\leq l$ on $(t_0, t_2].$
The lemma is proved.

\end{proof}

\subsection{Proof of Theorem \ref{theoA:main}}

In this subsection we show Theorem \ref{theoA:main}. Part (1) of  Theorem \ref{theoA:main} follows from
(\ref{eq:L003}) and (\ref{eq:L002}). For part (2), the proof divides into the following steps.

\emph{Step 1}. Without loss of generality, we can assume that $c(x, t)\leq 0. $  In fact, let $u(x, t)$ be a solution of (\ref{eq:D004}). Then for any $k\in \RR$ the function $\td u(x, t)=u(x, t)e^{kt}$ satisfies the equation
$$\pd {\td u}t=\Delta \td u+(c+k)\td u.$$
Therefore, for any compact set $K$ in $\Si\times [T_1, T_2]$ we can choose $k$ such that the function $\td c:=c+k
$ is nonpositive on $K$. Thus, it suffices to show Theorem \ref{theoA:main} for $c(x, t)\leq 0.$

\emph{Step 2}. Suppose that the curves $\{\xi_1(t), \cdots, \xi_l(t)\}$ are disjoint on $[T_1, T_2]$. Let $T_1<t_1<t_2<T_2$.   Lemma \ref{lem:A003b} implies that $u$ is in $L^1$.
For any $(x, t)\in \Si\times (t_1, t_2)$, we define
\beqn
w_k(x, t)&=&\int_{t_1}^t\,ds\int_{\Si}\, p(x, y, t-s)\td g_k(y, s)d\vol_y,\\
\td g_k(x, t)&=&c(x, t)u(x, t)\chi^{(k)}_{Q_{\frac 12, t_1-\dd_0, t_2}},\\
U_k(x, t)&=&\int_{(t_1, t)}\, p(x, \xi_k(s), t-s)\,d\mu_k(s), \label{Z:014}
\eeqn where $\mu_k$ is the measure obtained in Lemma \ref{lemA:003}, and $\dd_0>0$ is a constant chosen such that $t_1-\dd_0>T_1$.  Then $u-\sum_{k=1}^l\,(U_k+w_k)$ satisfies the heat equation in $\cD'(Q_{\frac 12, t_1, t_2})$, which implies that $u-\sum_{k=1}^l\,(U_k+w_k)$ is bounded in $Q_{\frac 12, t_1, t_2}$. Since $c(x, t)\leq 0$, we have $w_k(x, t)\leq 0$ and
$$u(x, t)\leq  \sum_{k=1}^l\,U_k(x, t)+f(x, t)$$ where $f(x, t)$ is a bounded function on $Q_{\frac 12, t_1, t_2}$. Therefore, by Lemma \ref{lem:estimate_U} $u$ satisfies the inequalities (\ref{eqB:007})-(\ref{eqB:008}).

\emph{Step 3}. In general, the singular curves may not be disjoint. In this case,  we assume that the curves $\{\xi_1(t), \cdots, \xi_l(t)\}$ are around $(x_0, t_0)$ on $(t_1, t_2)$. Consider the interval $(t_1, t_0)$.  By   Definition \ref{defi:around}, we can find an integer $l'\in [1, l]$ such that $\{\xi_1(t), \cdots, \xi_{l'}(t)\}$ are disjoint on  $ (t_1, t_0)$.  By Lemma \ref{lemA:003} we get  positive measures $\mu_k\in (C_0((t_1, t_0)))'$  for each $\xi_k$ with $k\in [1, l']$, and by Lemma \ref{lem:measure} we have
\beq 0\leq\ga^4 \mu_k(t)\leq \mu(t),\quad \forall\; t\in (t_1, t_0). \label{eq:F003}\\
\eeq
For each $k$, we define $U_k$ as in (\ref{Z:014}).
 Using the same argument as in (2), for any $t\in (t_1, t_0)$ we have
\beq u(x, t)\leq \sum_{k=1}^l U_k(x, t)+f(x, t),\quad \forall\, t\in (t_1, t_0),
\eeq
where $f(x, t)$ is a  bounded function. By (\ref{eq:F003}), we have
\beq
u(x, t)\leq \frac 1{\ga^4}\sum_{k=1}^{l'}\,\int_{t_1}^t\,p(x, \xi_k(s), t-s)\,d\mu+f(x, t),\quad \forall\; t\in (t_1, t_0).\label{Z:015}
\eeq Similarly, we can prove that (\ref{Z:015}) also holds for $t\in (t_0, t_2)$.
Therefore, by Lemma \ref{lem:estimate_U} $u$ satisfies the inequalities (\ref{eqB:007})-(\ref{eqB:008}). The theorem is proved.

\section{Proof of main theorems }

In this section, we prove Theorem \ref{theo:main1} and Corollary \ref{cor:main}.

\begin{proof}[Proof of Theorem \ref{theo:main1}] Suppose that the mean curvature flow (\ref{eq:MCF}) reaches a singularity at $(x_0, T)$ with $T<+\infty. $ Then Corollary 3.6 of \cite{[Eckbook]} implies that for all $t<T$ we have
\beq
d(\Si_t, x_0)\leq 2\sqrt{T-t}. \label{eq:I005}
\eeq
We rescale the flow  by
\beq s=-\log(T-t),\quad  \td \Si_s=e^{\frac s2}\Big(\Si_{T-e^{-s}}-x_0\Big). \label{eq:Z002}\eeq
such that the flow $\{(\td\Si_s, \td \x(p, s)), -\log T\leq s< +\infty\}$ satisfies the following properties:
\begin{enumerate}
  \item[(1)]  $\td \x(p, s)$ satisfies the equation
  \beq \Big(\pd {\td \x}s\Big)^{\perp}=-\Big(\td H-\frac 12 \langle \td \x, \n\rangle\Big) \n; \label{eq:A001}\eeq
  \item[(2)] the mean curvature of $\td \Si_s$ satisfies $|\td H(p, s)|\leq \La_0  $ for some $\La_0>0$;
  \item[(3)] $d(\td \Si_s, 0)\leq 2$.
\end{enumerate}
Fix $\tau>0$. By Theorem \ref{theo:removable}, for any sequence $s_i\ri +\infty$ there exists a subsequence, still denoted by $\{s_i\}$, such that   the flow  $\{\td \Si_{s_i+s}, -\tau<s<\tau\}$ converges smoothly to a self-shrinker  with multiplicity one.
 In other words, taking $c_j=e^{\frac {s_j}{2}}$ the flow $\{\td \Si_s^j, -\tau<s<\tau\}$ where $\td \Si_s^j:=c_j e^{\frac s2}(\Si_{T-c_j^{-2}e^{-s}}-x_0)$ converges smoothly to a self-shrinker  with multiplicity one as $j\ri +\infty$. Consider the corresponding flow
$$\td t=-e^{-s}, \quad \Si_{\td t}^j:=\sqrt{-\td t}\;\td \Si^j_{ -\log (-\td t)}=c_j\Big(\Si_{T+c_j^{-2}\td t}-x_0\Big).$$
Therefore,  for fixed $\tau>0$ the flow $\{\Si_{\td t}^j, -e^{\tau}<\td t<-e^{-\tau} \}$ converges smoothly to a smooth self-shrinker flow  with multiplicity one as $j\ri +\infty$.
Theorem \ref{theo:main1} is proved.\\
\end{proof}

\begin{proof}[Proof of Corollary \ref{cor:main}] We follow the argument in the proof of Theorem \ref{theo:main1}. Suppose that
\beq \dd_0:= \sup_{\Si\times [0, T)}\Big(\sqrt{T-t}\cdot |H|(p, t)\Big)<+\infty. \label{eq:A106} \eeq
Then the rescaled mean curvature flow (\ref{eq:Z002}) satisfies $|\td H|\leq \dd_0.$
There exists a sequence of times $s_i\ri +\infty$ such that for any fixed $\tau>0$ the flow  $\{\td \Si_{s_i+s}, -\tau<s<\tau\}$ converges smoothly to a self-shrinker  $\Si_{\infty}\in \cC(N, \rho)$ with multiplicity one. Moreover, the mean curvature of the limit self-shrinker satisfies $\sup_{\Si_{\infty}}|H|\leq \dd_0.$
On the other hand, we have

\begin{lem}\label{lem:gap} For any $ N>0$ and any increasing function $\rho$, there  exists a constant $\dd( N, \rho)>0$ such that any self-shrinker $\Si\in \cC(   N, \rho)$ with $|H|\leq \dd$ must be a plane passing through the origin.

\end{lem}
\begin{proof}[Proof of Lemma \ref{lem:gap}] For otherwise, there exists  a sequence of non-flat self-shrinkers  $\Si_i\in \cC( N, \rho)$  with $\sup_{\Si_i}|H|\leq \dd_i\ri 0.$ By the smooth compactness result of self-shrinkers in \cite{[CM1]},   we can assume that
$\Si_i$ converges smoothly to a self-shrinker $\Si_{\infty}\in \cC( N, \rho)$ with multiplicity one. Since the convergence is smooth, the limit self-shrinker $\Si_{\infty}$ has zero mean curvature and by Corollary 2.8 of \cite{[CM2]} it  must be a plane passing through the origin.

 Let $\Si_{i, t}=\sqrt{1-t}\Si_i$. Then  $\{\Si_{i, t}, 0\leq t<1\}$ is a solution of mean curvature flow (\ref{eq:MCF}) which reaches $x_0=0$ at $T=1$. Consider the Heat kernel function
$$\Phi_{(x_0, T)}(x, t)=\frac 1{4\pi(T-t)} e^{-\frac {|x-x_0|^2}{4(T-t)}},\quad
\forall\; (x, t)\in \Si_{i, t}\times [0, T).$$
Thus, Huisken's monotonicity formula(cf. Theorem 3.1 in \cite{[Hui2]}) implies that
\beqs
\Te(\Si_{i, t},  0, 1):&=&\lim_{t\ri 1}\,\int_{\Si_{i, t}}\, \Phi_{(0, 1)}(x, t)\,d\mu_{i, t}\\
&=&  \frac 1{4\pi}\int_{\Si_{i}}\, e^{-\frac {| x|^2}{4}}\,d \mu_{i}\ri \frac 1{4\pi}\int_{\Si_{\infty}}\, e^{-\frac {| x|^2}{4}}\,d \mu_{\infty}= 1,
\eeqs where we used the fact $\Si_i$ converges smoothly to the plane $\Si_{\infty}$ with multiplicity one. Therefore, by  Theorem 3.5 of White \cite{[White2]} or Theorem 5.6 of \cite{[Eckbook]} we have
\beq |A_{\Si_{i, t}}|(x, t)\leq \frac {C}{r_0}, \label{Z1}\eeq
for some $C, r_0>0$ and for all $(x, t)\in (\Si_{i, t}\cap B_{r_0}(0))\times (1-r_0^2, 1)$. For any $p\in \Si_i$, there exists $t_p\in (1-r_0^2, 1)$ such that for all $t\in (t_p, 1)$ we have $\sqrt{1-t}p\in \Si_{i, t}\cap B_{r_0}(0)$.
Thus, (\ref{Z1}) implies that for any $t\in (t_p, 1)$,
$$|A_{\Si_i}|(p)=\sqrt{1-t}|A_{\Si_{i, t}}|(\sqrt{1-t}p, t)\leq \frac C{r_0}\sqrt{1-t}.$$
Letting $t\ri 1$, we have $|A_{\Si_i}|(p)=0$ which contradicts our assumption that  $\Si_i$ is non-flat.

Alternatively, one can also quote the results of C.Bao (cf. Theorem 1.2 of~\cite{[Bao]}) or Guang-Zhu (cf.~\cite{[GZ]}) to obtain that each $\Sigma_i$ is a plane and derive the same contradiction.
The lemma is proved.

\end{proof}

Therefore, by Lemma \ref{lem:gap} the limit self-shrinker $\Si_{\infty}$ must be a plane passing through the origin.
Thus, Huisken's monotonicity formula implies that
\beqs
\Te(\Si_t,  x_0, T):&=&\lim_{t\ri T}\,\int_{\Si_t}\, \Phi_{(x_0, T)}(x, t)\,d\mu_t\\
&=& \lim_{s_i\ri +\infty} \frac 1{4\pi}\int_{\td \Si_{s_i}}\, e^{-\frac {| x|^2}{4}}\,d\td \mu_{s_i}=1,
\eeqs which implies that $(x_0, T)$ is a regular point by Theorem 3.1 of White \cite{[White2]}. Thus, the flow $\{\Si_t, 0\leq t<T\}$ cannot blow up at $(x_0, T)$. The corollary is proved.
\end{proof}

\begin{appendices}

\section{Krylov-Safonov's parabolic Harnack inequality}\label{appendixA}
In this appendix, we include the parabolic Harnack inequality from Krylov-Safonov \cite{[Kry]}.
First, we introduce some notations. Let  $x=(x^1, x^2, \cdots, x^n)\in \RR^n$. Denote
\beqs |x|&=&\Big(\sum_{i=1}^n(x^i)^2\Big)^{\frac 12},\quad B_R(x)=\{y\in \RR^n\;|\; |x-y|<R\},\\
Q(\te, R)&=&B_R(0)\times (0, \te R^2).\eeqs
Consider the parabolic operator
\beq
L u=-\pd ut+a^{ij}(x, t)u_{ij}+b^i(x, t)u_i-c(x, t)u, \label{eqB:104}
\eeq where the coefficients are measurable and satisfy the conditions
\beqn
\mu |\xi|^2&\leq &a^{ij}(x, t)\xi_i\xi_j\leq \frac 1{\mu}|\xi|^2,\label{eqB:001}\\
|b(x, t)|&\leq &\frac 1{\mu},\label{eqB:002}\\
0&\leq &c(x, t)\leq \frac 1{\mu}.\label{eqB:003}
\eeqn Here $b(x, t)=(b^1(x, t), \cdots, b^n(x, t)). $
Then we have

\begin{theo}\label{theo:B1} (Theorem 1.1 of \cite{[Kry]}) Suppose the operator $L$ in (\ref{eqB:104}) satisfies the conditions (\ref{eqB:001})-(\ref{eqB:003}).  Let $\te>1, R\leq 2, u\in W^{1, 2}_{n+1}(Q(\te, R)), u\geq 0$ in $Q(\te, R)$, and $Lu=0$ on $Q(\te, R)$. Then there exists a constant $C$, depending only on $\te, \mu$ and $n$, such that
\beq
u(0, R^2)\leq C\, u(x, \te R^2),\quad \forall\; x\in B_{\frac R2}(0).
\eeq Moreover, when $\frac 1{\te-1}$ and $\frac 1{\mu}$ vary within finite
bounds, $C$ also varies within finite bounds.
\end{theo}

Note that in our case the equation (\ref{eq:F007}) doesn't satisfy the assumption that $c(x, t)\geq 0$ in (\ref{eqB:003}). Therefore, we cannot use Theorem \ref{theo:B1} directly. The following result shows that  the Harnack inequality still works when $c(x, t)$ is bounded.

\begin{theo}\label{theo:B2}Let $\te>1, R\leq 2$. Suppose that $u(x, t)\in W^{1, 2}_{n+1}(Q(\te, R))$ is  a nonnegative solution to the equation
\beq
L u=-\pd ut+a^{ij}(x, t)u_{ij}+b^i(x, t)u_i+c(x, t) u=0,  \label{eqB:004}
\eeq where  the coefficients $a^{ij}(x, t)$ and $b^i(x, t)$ satisfy (\ref{eqB:001})-(\ref{eqB:002}), and $c(x, t)$ satisfies
\beq
|c(x, t)|\leq \frac 1{\mu}.\quad \forall\; (x, t)\in Q(\te, R)\label{eqA:C}
\eeq
 Then there exists a constant $C$, depending only on $\te, \mu$ and $n$, such that
\beq
u(0, R^2)\leq C\, u(x, \te R^2),\quad \forall\; |x|<\frac 12 R.
\eeq
\end{theo}
\begin{proof}Since $u(x, t)$ is a solution of (\ref{eqB:004}) and $c(x, t)$ satisfies (\ref{eqA:C}), the function $v(x, t)=e^{-\frac 1{\mu}t}u$ satisfies  \beq
-\pd vt+a^{ij}(x, t)v_{ij}+b^i(x, t)v_i+\td c(x, t)=0. \label{eq:B100}
\eeq where \beq -\frac 2{\mu}\leq \td c(x, t)=c(x, t)-\frac 1{\mu}\leq 0.\eeq
Applying Theorem \ref{theo:B1} to the equation (\ref{eq:B100}), we have $$v(0, R^2)\leq C\, v(x, \te R^2),\quad \forall\; |x|<\frac 12 R, $$
where $C$ depends only on $\te, \mu$ and $n$. Thus, for any $x\in B_{\frac R2}(0)$ we have
$$u(0, R^2)\leq Ce^{-k(\te-1)R^2}u(x, \te R^2)\leq C'u(x, \te R^2),$$
where $C'$ depends only on $\te, \mu$ and $n$. Here we used $R\leq 2$  by the assumption. The theorem is proved.

\end{proof}

We generalize Theorem \ref{theo:B2} to a general bounded domain in $\RR^n.$

\begin{theo}\label{theo:B4} Let $\Om$ be a bounded domain in $\RR^n$. Suppose that $u(x, t)\in W^{1, 2}_{n+1}(\Om\times (0, T))$ is  a nonnegative solution to the equation
\beq
L u=-\pd ut+a^{ij}(x, t)u_{ij}+b^i(x, t)u_i+c(x, t) u=0,  \label{eqB:0041}
\eeq where the coefficients $a^{ij}(x, t)$ and $b^i(x, t)$ satisfy (\ref{eqB:001})-(\ref{eqB:002}), and $c(x, t)$ satisfies (\ref{eqA:C}) for a constant $\mu>0$. For any $s, t$ satisfying $0<s<t<T$ and any $x, y\in \Om$ with the following properties
\begin{enumerate}
  \item[(1).] $x$ and $y$ can be connected by a line segment $\ga$ with the length $|x-y|\leq l;$
  \item[(2).] Each point in $\ga$ has a positive distance at least $\dd>0$ from the boundary of $\Om;$
\item[(3)] $s$ and $ t$ satisfy $T_1\leq t-s\leq T_2$ for some $T_1, T_2>0;$
\end{enumerate}
 we have
\beq
u(y, s)\leq C \,u(x, t) , \label{eq:J004}
\eeq where $C$ depends only on $ n, \mu, \min\{s, \dd^2\}, l, T_1$ and $T_2$.
\end{theo}
\begin{proof}
Let $\ga$ be the line segment with the property $(1)$ and $(2)$ connecting $x$ and $y$. We set $$p_0=y,\quad  p_N=x, \quad p_i=p_0+\frac {x-y}{N}i\in \ga$$ for any $0\leq i\leq N.$ Here we choose $N$ to be the smallest integer satisfying
\beq
N>\max\Big\{\frac {2(t-s)}{s}, \frac {l}{\min\{\frac {\sqrt{s}}{4}, \frac {\dd}4 \}}\Big\}. \label{eqB:100}
\eeq We define
\beq
R=\frac {2l}{N}, \quad  \te=1+\frac {t-s}{R^2N}.\label{eqB:101}
\eeq We can check that $R\leq \frac {\dd}2.$
For any $s, t\in (0, T)$, we choose $\{t_i\}_{i=0}^N$ such that $t_0=s, t_N=t$ and \beq t_{i}-t_{i-1}=\frac {t-s}N \label{eqB:105}\eeq for all integers $1\leq i\leq N$. Note that (\ref{eqB:100})-(\ref{eqB:105}) imply that for any $0\leq i\leq N-1$,
$$t_{i+1}-\te R^2\geq s-\te R^2=s-R^2-\frac {t-s}N\geq \frac s4>0$$
and
$$|p_{i+1}-p_i|=\frac {|x-y|}{N}\leq \frac lN=\frac R2.$$
Therefore, for any $0\leq i\leq N-1$ we have  $(t_{i+1}-\te R^2, t_{i+1})\subset (0, T) $ and $p_{i+1}\in B_{\frac R2}(p_i)$.  Applying Theorem \ref{theo:B2} on $B_R(p_i)\times (t_{i+1}-\te R^2, t_{i+1})\subset \Omega\times (0, T)$, we have
\beq
u(p_i, t_i)\leq C\,u(p_{i+1}, t_{i+1}),
\eeq where $C$ depends only on $c, n, \mu$ and $\frac 1{\te-1}=\frac {R^2N}{t-s}. $ Here we used the fact that
$t_i=(t_{i+1}-\te R^2)+R^2.$
Therefore,
\beq
u(y, s)=u(p_0, t_0)\leq C^Nu(p_N, t_N)=C' u(x, t)  \label{eqB:103}
\eeq
where the constant
$C'$ in (\ref{eqB:103}) depends only on $c, n, \mu, \min\{s, \dd^2\}, l,  T_1$ and $T_2$. The theorem is proved.

\end{proof}

A direct corollary of Theorem \ref{theo:B4} is the following result.

\begin{theo}\label{theo:B3} Let $\Om$ be a bounded domain in $\RR^n$. Suppose that $u(x, t)\in W^{1, 2}_{n+1}(\Om\times (0, T))$ is  a nonnegative solution to the equation
\beq
L u=-\pd ut+a^{ij}(x, t)u_{ij}+b^i(x, t)u_i+c(x, t) u=0,  \label{eqB:0042}
\eeq where the coefficients $a^{ij}(x, t)$ and $b^i(x, t)$ satisfy (\ref{eqB:001})-(\ref{eqB:002}), and $c(x, t)$ satisfies (\ref{eqA:C}) for a constant $\mu>0$. Suppose that $\Om', \Om''$ are subdomains in $\Om$ satisfying the following properties:
\begin{enumerate}
  \item[(1).]$\Om'\subset \Om''\subset \Om$, and $\Om''$ has a positive distance $\dd>0$ from the boundary of $\,\Om$;

  \item[(2).] $\Om'$ can be covered by $k$ balls with radius $r$, and all  balls are contained in $\Om''.$
\end{enumerate}
Then for any $s, t$ satisfying $0<s<t<T$ and any $x, y\in \Om'$,
 we have
\beq
u(y, s)\leq C \,u(x, t) , \label{eq:J0041b}
\eeq where $C$ depends only on $n, \mu, \min\{s, \dd^2\}, t-s, r$ and $k$.
\end{theo}
\begin{proof}By the assumption, we can find finite many points $\cA=\{q_1, q_2, \cdots, q_k\}$ such that
\beq
\Om'\subset \cup_{q\in \cA}B_r(q)\subset \Om''.
\eeq
For any $x, y\in \Om'$, there exists two points in $\cA$, which we denote by $q_1$ and $q_2$, such that $x\in B_{r}(q_1)$ and $y\in B_{r}(q_2)$. Then $x$ and $y$ can be connected by a polygonal chain $\ga$, which consists of two line segments $\overline{xq_1}, \overline{yq_2}$ and a polygonal chain   with vertices in $\cA$ connecting $q_1$ and $q_2$. Clearly, the number of the vertices of $\ga$ is bounded by $k+2$ and the total length of $\ga$ is bounded by $(k+2) r.$ Moreover, by the assumption we have $\ga\subset \Om''$ and  each point in $\ga$ has a positive distance at least $\dd>0$ from the boundary of $\Om.$

 Assume that the polygonal chain $\ga$ has consecutive  vertices $\{p_0, p_1, \cdots, p_N\}$ with $p_0=y, p_N=x$ and $1\leq N\leq k+2$.  We apply Theorem \ref{theo:B4} for each line segment $\overline{p_ip_{i+1}}$ and the interval $[t_i, t_{i+1}]$, where $\{t_i\}$ is chosen as in (\ref{eqB:105}). Note that
$$\frac {t-s}{k+2}\leq t_{i+1}-t_i=\frac {t-s}{N}\leq t-s.$$
Thus,  for any $0\leq i\leq N-1$ we have
\beq
u(p_i, t_i)\leq Cu(p_{i+1}, t_{i+1}),\label{eqC:003}
\eeq where $C$ depends only on $c, n, \mu, \min\{s, \dd^2\}, r, k$ and $t-s$, and (\ref{eqC:003}) implies (\ref{eq:J0041b}).  This finishes the proof of Theorem \ref{theo:B3}.

\end{proof}

Theorem \ref{theo:B3} can be generalized to Riemannian manifolds by using the partition of unity. Here we omit the proof since the argument is standard. Note that the constant in (\ref{eq:J0041b}) depends on the geometry of $(M, g)$.

\begin{theo}\label{theo:B5} Let $(M, g)$ be a  Riemannian manifold
with boundary $\partial M$ and $\Om\subset M$ a bounded domain which doesn't intersect with $\partial M.$
 Suppose that $u(x, t)\in W^{1, 2}_{n+1}(\Om\times (0, T))$ is  a nonnegative solution to the equation
\beq
L u=-\pd ut+a^{ij}(x, t)\Na_i\Na_ju+b^i(x, t) \Na_i u+c(x, t) u=0,  \label{eqB:0042}
\eeq where the coefficients $a^{ij}(x, t)$ and $b^i(x, t)$ satisfy (\ref{eqB:001})-(\ref{eqB:002}), and $c(x, t)$ satisfies (\ref{eqA:C}) for a constant $\mu>0$. Suppose that $\Om', \Om''$ are subdomains in $\Om$ satisfying the following properties:
\begin{enumerate}
  \item[(1).]$\Om'\subset \Om''\subset \Om$, and $\Om''$ has a positive distance $\dd>0$ from the boundary of $\,\Om$;

  \item[(2).] $\Om'$ can be covered by $k$ balls with radius $r$, and all  balls are contained in $\Om''.$
\end{enumerate}
Then for any $s, t$ satisfying $0<s<t<T$ and any $x, y\in \Om'$,
 we have
\beq
u(y, s)\leq C \,u(x, t) , \label{eq:J0041b}
\eeq where $C$ depends only on $c, n, \mu, \min\{s, \dd^2\}, t-s, r, k$ and $(M, g)$.
\end{theo}

\section{Li-Yau's parabolic Harnack inequality}\label{appendixB}

In this appendix, we include Li-Yau's parabolic Harnack inequality
in \cite{[LY]}. Compared with the Harnack inequality in appendix \ref{appendixA}, Li-Yau's result gives explicit dependence of the constants  on the geometric quantities of the metric. Thus, we can apply Li-Yau's result to a class of Riemannian manifolds and we   obtain uniform bounds of the constants in the Harnack inequality.

\begin{theo}\label{theo:LY}(cf.  Theorem 2.1 of \cite{[LY]}) Let $M$ be a Riemannian manifold with boundary $\partial M.$ Assume $p\in M$ and let $\cB_{2R}(p)$ be a geodesic ball of radius $2R$ centered at $p$ which does not intersect $\partial M$. We denote $-K(2R)$, with $K(2R)\geq 0$, to be a lower bound of the Ricci curvature on $\cB_{2R}(p).$ Let $q(x, t)$ be a function defined on $M\times [0, T]$ which is $C^2$ in the $x$-variable and $C^1$ in the $t$-variable. Assume that
\beq \Delta q\leq \te(2R),\quad |\Na q|\leq \ga(2R)\eeq
on $\cB_{2R}(p)\times [0, T]$ for some constants $\te(2R)$ and $\ga(2R)$. If $u(x, t)$ is a positive solution of the equation
\beq
\Big(\Delta-q-\frac {\partial}{\partial t}\Big)u(x, t)=0
\eeq on $M\times (0, T]$, then for any $\al>1$, $0<t_1<t_2\leq T$, and $x, y\in \cB_R(p)$, we have the inequality
\beq
u(x, t_1)\leq u(y, t_2)\Big(\frac {t_2}{t_1}\Big)^{\frac {n\al}{2}}e^{A(t_2-t_1)+\rho_{\al, R}(x, y, t_2-t_1)},
\eeq where
\beq
A=C\Big(\al R^{-1}\sqrt{K}+\al^3 (\al-1)^{-1}R^{-2}+\ga^{\frac 23}(\al-1)^{\frac 13}\al^{-\frac 13}+(\al \te)^{\frac 12}+\al(\al-1)^{-1}K\Big)
\eeq and
\beq
\rho_{\al, R}(x, y, t_2-t_1)=\inf_{\ga\in \Ga(R)}\Big(\frac {\al}{4(t_2-t_1)}\int_0^1\,|\dot \ga|^2+(t_2-t_1)\int_0^1\,q(\ga(s), (1-s)t_2+st_1)\,ds\Big),
\eeq with $\inf$ taken over all paths in $\cB_R(p)$ parametrized by $[0, 1]$ joining $y$ to $x$.

\end{theo}

A direct corollary of Theorem \ref{theo:LY} is the following result.

\begin{theo}\label{theo:LYB2} The same assumptions as in Theorem \ref{theo:LY} on $M, \cB_{2R}(p)$ and the function $q(x, t)$.
 If $u(x, t)$ is a positive solution of the equation
\beq
\Big(\Delta-q-\frac {\partial}{\partial t}\Big)u(x, t)=0
\eeq on $\Om\times (0, T]$, where $\Om$ is a connected open subset of $\cB_R(p)$. Let $\Om', \Om''$ are connected open subsets of $\Om$ satisfying the following properties, which we called $(\dd, k, r)$ property:
\begin{enumerate}
  \item[(1).]$\Om'\subset \Om''\subset \Om$, and $\Om''$ has a positive distance $\dd>0$ from the boundary of $\,\Om$;

  \item[(2).] $\Om'$ can be covered by $k$ geodesic balls with radius $r$, and all  balls are contained in $\Om''.$
\end{enumerate}
Then for any $0<t_1<t_2\leq T$ and $x, y\in \Om'$, we have the inequality
\beq
u(x, t_1)\leq Cu(y, t_2),
\eeq where $C$ depends only on $n,   K(2R), \te(2R), \ga(2R), t_1, t_2-t_1, k, \dd$ and $r$.

\end{theo}
\begin{proof}By the assumption on $\Om', \Om''$ and $\Om$, $x$ and $y$ can be connected by a path $\ga$ in $\Om''$ with bounded length and every point in $\ga$ has a distance at least $\dd$ from the boundary of $\Om$. Thus, the theorem follows directly from Theorem \ref{theo:LY} by choosing  $R=\dd$ and $\al=2$.

\end{proof}

In the proof of Lemma \ref{claim2}, we need to use Theorem \ref{theo:LYB2} to a class of surfaces with bounded geometry. In order to show that the constants in the Harnack inequality is uniformly bounded, we have the following result.

\begin{theo}\label{theo:LYB3}Fix $R>0.$ We assume that
\begin{enumerate}
\item[(1).]   $\Si_i^2\subset \RR^3$ is a sequence of complete surfaces which converges smoothly to a complete surface $\Si$ in $\RR^3$;

\item[(2).] The Ricci curvature of $\Si\cap B_R(0)$ is bounded by a constant $-K$ with $K\geq 0.$ Here $B_R(0)\subset \RR^3$ denotes the extrinsic ball centered at $0$ with radius $R$;

\item[(3).] $\Om_i, \Om_i', \Om_i''$ are bounded domains in $\Si_i\cap B_{\frac R2}(0)$ with $\Om_i'\subset \Om_i''\subset \Om_i$, and  $\Om_i, \Om_i', \Om_i''$ converges smoothly to $\Om, \Om', \Om'' $ with $\Om'\subset\Om''\subset \Om\subset \Si\cap B_{\frac R2}(0)$ respectively. Here the smooth convergence of $\Om_i$ to $\Om$ means that for any $\ee>0$ and sufficiently large $i$, there exists a smooth function $u_i$ on $\Om$ with $|u_i|_{C^2(\Om)}\leq \ee$ such that $\Om_i$ can be written as a normal exponential graph of $u_i$ over $\Om$;

  \item[(4).] $\Om''$ has a positive geodesic distance $\dd>0$ from the boundary of $\,\Om$;

  \item[(5).] $\Om'$ can be covered by $k$ geodesic balls with radius $r\in (0, \frac {\dd}2)$, and all  balls are contained in $\Om''$;

\item[(6).] $q_i(x, t)$ is a function defined on $\Si_i\times [0, T]$ which is $C^2$ in the $x$-variable and $C^1$ in the $t$-variable. Assume that
\beq \Delta_{g_i} q_i\leq \te,\quad |\Na q_i|_{g_i}\leq \te\eeq
on $\Om_i\times [0, T]$ for some constant $\te$.

\end{enumerate}
If $f_i(x, t)$ are positive functions satisfying
\beq
\Big(\Delta_{g_i}-q_i(x, t)-\frac {\partial}{\partial t}\Big)f_i(x, t)=0 \label{eq:fib}
\eeq on $\Om_i\times (0, T]$, where $q_i(x, t)\in C^2(\Si_i\times [0, T])$,   then for any $0<t_1<t_2\leq T$ and $x, y\in \Om'_i$, we have the inequality
\beq
f_i(x, t_1)\leq Cf_i(y, t_2),
\eeq where $C$ depends only on $n,   K, \te,  t_1, t_2-t_1, k, \dd$ and $r$.

\end{theo}
\begin{proof}It suffices to show that $\Om_i', \Om_i''$ and $\Om_i$ satisfy the $(\dd', k', r')$ property of  Theorem \ref{theo:LYB2} with uniform constants $\dd', k'$ and $r'$. By the smooth convergence of $\Om_i$ to $\Om$, we define the map $\varphi_i: \Om\ri \Om_i$ by
\beq
\varphi_i(x)=x+u_i(x)\n(x),\quad \forall\;x\in \Om,\label{eqB:A1}
\eeq where $u_i(x)$ is the graph function of $\Om_i$ over $\Om$ and $\n(x)$ denotes the normal vector of $\Si$ at $x$.   Note that $\varphi_i(\Om)=\Om_i$ and $\varphi_i$ converges in $C^2$  to the identity map on $\Om$ as $i\ri +\infty. $ By the assumption (5), there exists $k$ points $\{p_{\al}\}_{\al=1}^k\subset \Om'$ and $\ee>0$ such that
\beq \Om'\subset \cup_{\al=1}^k\cB_r(p_{\al}),\quad \cB_r(p_{\al})\subset \Om''_{4\ee}, \eeq
where $\Om_{4\ee}''=\{x\in \Om''\,|\;d_{\Si}(x, \p \Om'')\geq 4\ee\}.$
Therefore, we have
\beq
\Om_i'=\varphi_i(\Om')\subset \varphi_i\Big(\cup_{\al=1}^k\cB_r(p_{\al})\Big)
=\cup_{\al=1}^k\varphi_i\Big(\cB_r(p_{\al}).\label{eqB:A3}
\Big)
\eeq Since the $C^l$ norms of $u_i$ in (\ref{eqB:A1}) are small, for large $i$ we have
\beq
 \varphi_i(\cB_{r}(p_{\al}))\subset \cB_{i, r+\ee}(\varphi_i(p_{\al}))\subset \Om_{i, 2\ee}'', \label{eqB:A2}
\eeq where $\Om_{i, 2\ee}''=\{x\in \Om_{i}''\,|\;d_{\Si_i}(x, \p \Om_{i}'' )\geq 2\ee\}$ and $\cB_{i, r}(p)$ denotes the geodesic ball of $\Si_i$ centered at $p$ with radius $r$.
Combining (\ref{eqB:A3}) with (\ref{eqB:A2}), we have
\beq
\Om_i'\subset \cup_{\al=1}^k\,\cB_{r+\ee}(\varphi_i(p_{\al}))\subset \Om_{i, 2\ee}''\subset \Om_i''.
\eeq Therefore, $\Om_i'$ can be covered by $k$ geodesic balls with radius $r+\ee$, and all  balls are contained in $\Om''_i$. It is clear that $\Om''_i$ has a positive geodesic distance $\frac {\dd}2>0$ from the boundary of $\,\Om_i$ for large $i$. Thus, $\Om_i', \Om_i''$ and $\Om_i$ satisfy the $(\frac {\dd}2, k, r+\ee)$ property and the theorem follows directly from Theorem \ref{theo:LYB2}.

\end{proof}

\section{The linearized equation  of rescaled mean curvature flow}\label{appendixC}
In this appendix, we follow the calculation in Appendix A of Colding-Minicozzi \cite{[CM4]} to show (\ref{eq:F008}). See also Appendix A of Colding-Minicozzi \cite{[CM5]}.
 Let $\Si$ be a hypersurface in $\RR^{n+1}$ and $\Si_u$ the graph of a function $u$ over $\Si.$ Then $\Si_u$ is give by
\beq
\Si_u=\{x+u(x)\n(x)\;|\;x\in \Si\},
\eeq where $\n(x)$ denotes the normal vector of $\Si$ at $x$. We assume that $|u|$ is small.
Let $e_{n+1}$ be the gradient of the signed distance function to $\Si$ and $e_{n+1}$ equals $\n$ on $\Si.$  We define
\beq
\nu_u(p)=\sqrt{\frac {\det g_{ij}^u(p)}{\det g_{ij}(p)}},\quad w_i(p)=\langle e_{n+1}, \n_u\rangle,\quad  \eta_u(p)=\langle p+u(p)\n(p), \n_u\rangle \label{eqC:C1}
\eeq where $g_{ij}$ denotes the metric on $\Si$ at $p$, $g_{ij}^u$ is the induced metric on $\Si_u$ and $\n_u$ is the normal to $\Si_u.$

\begin{lem}\label{lem:C1}(Lemma A.3 of \cite{[CM4]}) There are functions $w, \nu, \eta$ depending on $(p, s, y)\in \Si\times \RR\times T_p\Si$ that are smooth for $|s|$ less than the normal injectivity radius of $\Si$ so that
\beqn
w_u(p)&=&w(p, s, y)=\sqrt{1+|B^{-1}(p, s)(y)|^2},\label{eqC:C2}\\
\nu_u(p)&=&\nu(p, s, y)=w(p, s, y)\det(B(p, s)),\label{eqC:C3}\\
\eta_u(p)&=&\eta(p, s, y)=\frac {\langle p, \n(p) \rangle+s-\langle p, B^{-1}(p, s)(y) \rangle}{w(p, s, y)} \label{eqC:C4}
\eeqn where the linear operator $B(p, s)=Id-s A(p)$. Finally, we have
\begin{enumerate}
\item[(1).] $w$ satisfies \beqn
w(p, s, 0)&=&1, \quad \p_sw(p, s, 0)=0,\label{eqC:C7}\\
\p_{y_{\al}}w(p, s, 0)&=&0,\quad  \p_{y_{\al}}\p_{y_{\bb}}w(p, 0, 0)=\dd_{\al\bb}.\label{eqC:C8}
\eeqn

\item[(2).] $\nu$ satisfies  \beqn \nu(p, 0, 0)&=&1,\quad
\p_s\nu(p, 0, 0)=H(p),\\
\p_{p_j}\p_s\nu(p, 0, 0)&=&H_j(p),\quad \p_{y_{\al}}\p_{y_{\bb}}\nu(p, 0, 0)=\dd_{\al\bb},\\
\p_s^2\nu(p, 0, 0)&=&H^2(p)-|A|^2(p).
\eeqn
\item[(3).] $\eta$ satisfies
\beqn
\eta(p, 0, 0)&=&\langle p, \n\rangle,\quad  \p_s\eta(p, 0, 0)=1,\\
\p_{y_{\al}}\eta(p, 0, 0)&=&-p_{\al}.
\eeqn

\item[(4).] Furthermore, we have
\beqn
\p_{y_i}\nu(p, 0, 0)&=&0,\quad \p_{p_{j}}\p_{y_i}\nu(p, 0, 0)=0,\label{eqC:E1}\\
\p_s\p_{p_{j}}\p_{y_i}\nu(p, 0, 0)&=&0,\quad \p_{y_k}\p_{p_{j}}\p_{y_i}\nu(p, 0, 0)=0.\label{eqC:E2}
\eeqn
\end{enumerate}

\end{lem}
\begin{proof} Part (1)-(3) and (\ref{eqC:C2})-(\ref{eqC:C4}) follow directly from Lemma A.3 of \cite{[CM4]}. It suffices to show Part (4). Following the notations in the proof of Lemma A.3 of \cite{[CM4]}, we assume that $(p, s)$ is the Fermi coordinates on the normal tubular neighorhood of $\Si$ so that $s$ measures the signed distance to $\Si.$ We define
\beq
B(p, s)\equiv(\mathrm{Id}-s A(p)): T_p\Si\ri T_p\Si. \label{eqC:C6}
\eeq
Let $\cB(p, s)=\det(B(p, s))$ and $J(p, s)=B^{-1}(p, s)$.
Then we have
\beqn \cB(p, 0)&=&1,\quad \p_s\cB(p, 0)=H(p), \\
\p_{y_i}\cB(p, s)&\equiv& 0,\quad \p_{p_j}B(p, 0)=-s\p_{p_j}A|_{s=0}=0, \label{eqC:D1}\\
\p_{p_j}\cB(p, 0)&=&\cB(p, 0)\cdot \tr( \p_{p_j}B(p, 0))=0.\label{eqC:D5}
\eeqn
Since $J=B^{-1}$, we have
$$\p_{p_j}JB+J\p_{p_j}B=0.$$
This implies that \beq
\p_{p_j}J(p,  0)=-J(p, 0)\cdot\p_{p_j}B(p, 0)\cdot J(p, 0)=0.\label{eqC:C9}
\eeq
Note that  by (\ref{eqC:C2}), $w$ can be rewritten as
\beq
w(p, s, y)=\sqrt{1+J_{\al\bb}J_{\al \ga}y_{\bb}y_{\ga}}. \label{eqC:C5}
\eeq It follows immediately that
\beqs
\p_{y_i}w&=&\frac 1{2w}J*J*y,\\
\p_{p_i}w&=&\frac 1{w}\p_{p_i}J*J*y*y,\\
\p_s\p_{y_i}w&=&-\frac 1{2w^2}\p_sw\cdot J*J*y+\frac 1w\p_sJ*J*y,\\
\p_{p_j}\p_{y_i}w&=&-\frac 1{2w^2}\p_{p_j}w\cdot J*J*y+\frac 1w\p_{p_j}J*J*y,
\eeqs where the notation ``*" denotes the multiplication of two matrices. Furthermore, we calculate
\beqs\p_s\p_{p_j}\p_{y_i}w&=&w^{-3}\p_sw\p_{p_j}w\cdot J*J*y-\frac 1{2w^2}\p_s\p_{p_j}w\cdot J*J*y\\
&&-\frac 1{w^2}\p_{p_j}w\cdot \p_sJ*J*y-\frac 1{w^2}\p_sw\p_{p_j}J*J*y\\&&+
\frac 1w\p_s\p_{p_j}J*J*y+\frac 1w\p_{p_j}J*\p_sJ*y,\\
\p_{y_k}\p_{p_j}\p_{y_i}w&=&w^{-3}\p_{y_k}w\p_{p_j}w\cdot J*J*y-\frac 1{2w^2}\p_{y_k}\p_{p_j}w\cdot J*J*y\\
&&-\frac 1{w^2}\p_{p_j}w\cdot \p_{y_k}J*J*y-\frac 1{2w^2}\p_{p_j}w\cdot J*J\\&&-\frac 1{w^2}\p_{y_k}w\p_{p_j}J*J*y+\frac 1w\p_{y_k}\p_{p_j}J*J*y\\&&+\frac 1w\p_{p_j}J*\p_{y_k}J*y+\frac 1w\p_{p_j}J*J.
\eeqs
Combining the above identities with (\ref{eqC:C7})(\ref{eqC:C8}) and (\ref{eqC:C9}), we have
\beqn
\p_{y_i}w(p, 0, 0)&=&0,\quad \p_{p_i}w(p, 0, 0)=0,\label{eqC:D2} \\
\p_s\p_{y_i}w(p, 0, 0)&=&0,\quad
\p_{p_j}\p_{y_i}w(p, 0, 0)=0,\label{eqC:D3}\\
\p_s\p_{p_j}\p_{y_i}w(p, 0, 0)&=&0,\quad \p_{y_k}\p_{p_j}\p_{y_i}w(p, 0, 0)=0.\label{eqC:D4}
\eeqn
Moreover, we calculate the derivatives of the function $\nu(p, s, y)=w(p, s, y)\cB(p, s)$
\beqs
\p_{y_i}\nu&=&\p_{y_i}w \cB+w\p_{y_i}\cB,\\
\p_{p_{j}}\p_{y_i}\nu&=&\p_{p_{j}}\p_{y_i}w \cB+\p_{y_i}w \p_{p_{j}}\cB
+\p_{p_{j}}w\p_{y_i}\cB+w\p_{p_{j}}\p_{y_i}\cB,\\
\p_s\p_{p_{j}}\p_{y_i}\nu&=&\p_s\p_{p_{j}}\p_{y_i}w \cB+\p_{p_{j}}\p_{y_i}w \p_s\cB+\p_s\p_{y_i}w \p_{p_{j}}\cB+\p_{y_i}w \p_s\p_{p_{j}}\cB\\
&&+\p_s\p_{p_{j}}w\p_{y_i}\cB+\p_{p_{j}}w\p_{y_i}\p_s\cB+\p_sw\p_{p_{j}}\p_{y_i}\cB
+w\p_s\p_{p_{j}}\p_{y_i}\cB,\\
\p_{y_k}\p_{p_{j}}\p_{y_i}\nu&=&\p_{y_k}\p_{p_{j}}\p_{y_i}w \cB+\p_{p_{j}}\p_{y_i}w \p_{y_k}\cB+\p_{y_k}\p_{y_i}w \p_{p_{j}}\cB+\p_{y_i}w \p_{y_k}\p_{p_{j}}\cB\\
&&+\p_{y_k}\p_{p_{j}}w\p_{y_i}\cB+\p_{p_{j}}w\p_{y_i}\p_{y_k}\cB+\p_{y_k}w\p_{p_{j}}\p_{y_i}\cB
+w\p_{y_k}\p_{p_{j}}\p_{y_i}\cB.
\eeqs
Combining this with (\ref{eqC:D1})-(\ref{eqC:D5}), (\ref{eqC:D2})-(\ref{eqC:D4}), we have (\ref{eqC:E1})-(\ref{eqC:E2}). The lemma is proved.

\end{proof}

We have the following expression for the mean curvature of $\Si_u$.
\begin{lem}\label{lem:C2}(Corollary A.30 of \cite{[CM4]}) The mean curvature $H_u$ of $\Si_u$ is given by
\beq
H_u(p)=\frac {w}{v}\Big(\partial_s\nu-\p_{p_{\al}}\p_{y_{\al}}
\nu-(\p_s\p_{y_{\al}}\nu)u_{\al}(p)-(\p_{y_{\bb}}\p_{y_{\al}}\nu)
u_{\al\bb}(p)\Big),
\eeq where $w, \nu$ and their derivatives are all evaluated at $(p, u(p), \Na u(p)). $\\

\end{lem}

Combining Lemma \ref{lem:C1} with Lemma \ref{lem:C2}, we can show (\ref{eq:F008}).

\begin{lem}\label{lem:linearized2} The function $u_i=u_i^+-u_i^-$  satisfies the following parabolic equations  on $\Om_{\ee, R}(I)\times I$
\beq
\pd {u_i}t=\Delta_0u_i-\frac 12\langle x, \Na u_i\rangle+|A|^2 u_i+\frac {u_i}2+a_i^{pq}u_{i, pq}+b_i^pu_{i, p}+c_i\,u_i,
\eeq where $\Delta_0$ denotes the Laplacian operator on $\Si_{\infty}$ with respect to the induced metric, and
the coefficients $a_i^{pq}, b_i^p $ and $c_i$ are small and tend to zero as $u_i^+$ and $u_i^-$ tend to zero.

\end{lem}

\begin{proof}We divide the proof into several steps.

\emph{Step 1.}We calculate the difference of the mean curvature of $\Si_{u_i^+}$ and $\Si_{u_i^-}$.
 Let $u=u_i^+-u_i^-$ and $\td u_{\tau}=u_i^-+\tau u$ for $\tau\in [0, 1]$. Thus,  we have $\td u_0=u_i^-$ and $\td u_1=u_i^+$.
Note that
\beq
H_{u_i^+}(p)-H_{u_i^-}(p)=\int_0^1\,\p_{\tau}(H_{\td u_{\tau}}(p))\,d\tau.\label{eqC:E6}
\eeq
For any function $f(p, s, y)$, we calculate the derivative with respect to $\tau$
\beq
\p_{\tau}(f(p, \td u_{\tau}, \Na \td u_{\tau}))=(\p_sf)(p, \td u_{\tau}, \Na \td u_{\tau})\cdot u+(\p_{y_{\al}}f)(p, \td u_{\tau}, \Na \td u_{\tau})\cdot  u_{\al},\label{eqC:E11}
\eeq where $u_{\al}=\p_{x_{\al}}u.$ Therefore, we have
\beqs
\p_{\tau}(\p_s \nu)&=&\p_s^2\nu\cdot u+\p_{y_i}\p_{s}\nu\cdot u_i,\\
\p_{\tau}(\p_{p_{\al}}\p_{y_{\al}}
\nu )&=&\p_s\p_{p_{\al}}\p_{y_{\al}}
\nu\cdot u+\p_{y_i}\p_{p_{\al}}\p_{y_{\al}}
\nu\cdot u_i,\\
\p_{\tau}\Big((\p_s\p_{y_{\al}}\nu)\td  u_{\tau, \al} \Big)&=&(\p_s\p_{y_{\al}}\nu) u_{\al}+(\p_s^2\p_{y_{\al}}\nu)   u \td u_{\tau, \al}+(\p_s\p_{y_i}\p_{y_{\al}}\nu)  u_i \td u_{\tau, \al}\\
\p_{\tau}\Big( (\p_{y_{\bb}}\p_{y_{\al}}\nu)
 \td  u_{\tau, \al\bb}\Big)&=&(\p_{y_{\bb}}\p_{y_{\al}}\nu)
  u_{\al\bb}+ \p_s\p_{y_{\bb}}\p_{y_{\al}}\nu\cdot u\td u_{\tau, \al\bb} +\p_{y_i}\p_{y_{\bb}}\p_{y_{\al}}\nu \cdot u_i \td u_{\tau, \al\bb},
\eeqs where $\td u_{\tau, \al}=\p_{x_{\al}}\td u_{\tau}$ and $\td u_{\tau, \al\bb}=\p_{x_{\al}}\p_{x_{\bb}}\td u_{\tau}.
$
By Lemma \ref{lem:C2} we have
\beqn &&
\p_{\tau}(H_{\td u_{\tau}}(p))\nonumber\\&=& \Big(\p_s\Big(\frac w{\nu}\Big)u+\p_{y_i}\Big(\frac w{\nu}\Big) u_i\Big)\cdot \Big(\partial_s\nu-\p_{p_{\al}}\p_{y_{\al}}
\nu-(\p_s\p_{y_{\al}}\nu)\td u_{\tau, \al}-(\p_{y_{\bb}}\p_{y_{\al}}\nu)
\td u_{\tau, \al\bb}\Big)\nonumber\\
&&+\Big(\frac w{\nu}\Big)\cdot \p_{\tau} \Big(\partial_s\nu-\p_{p_{\al}}\p_{y_{\al}}
\nu-(\p_s\p_{y_{\al}}\nu)\td u_{\tau, \al}-(\p_{y_{\bb}}\p_{y_{\al}}\nu)
\td u_{\tau, \al\bb}\Big)\nonumber\\
&=&E u+F_{\al}u_{\al}+G_{\al\bb}u_{\al\bb}, \label{eqC:E7}
\eeqn where $E, F$ and $G$ are given by
\beqn
E(p, \td u_{\tau})&=& \p_s\Big(\frac w{\nu}\Big)\Big(\partial_s\nu-\p_{p_{\al}}\p_{y_{\al}}
\nu-(\p_s\p_{y_{\al}}\nu)\td u_{\tau, \al}-(\p_{y_{\bb}}\p_{y_{\al}}\nu)
\td u_{\tau, \al\bb}\Big) \nonumber\\
&&+\Big(\frac w{\nu}\Big)\cdot\Big(\p_s^2\nu-\p_s\p_{p_{\al}}\p_{y_{\al}}
\nu-(\p_s^2\p_{y_{\al}}\nu)    \td u_{\tau, \al}-\p_s\p_{y_{\bb}}\p_{y_{\al}}\nu\cdot \td u_{\tau, \al\bb}\Big) \label{eqD:E1}\\
F_{\ga}(p, \td u_{\tau})&=&\p_{y_\ga}\Big(\frac w{\nu}\Big)\Big(\partial_s\nu-\p_{p_{\al}}\p_{y_{\al}}
\nu-(\p_s\p_{y_{\al}}\nu)\td u_{\tau, \al}-(\p_{y_{\bb}}\p_{y_{\al}}\nu)
\td u_{\tau, \al\bb}\Big) \nonumber\\
&&+\Big(\frac w{\nu}\Big)\Big(\p_{y_\ga}\p_{s}\nu-\p_{y_\ga}\p_{p_{\al}}\p_{y_{\al}}
\nu-\p_s\p_{y_{\ga}}\nu -(\p_s\p_{y_i}\p_{y_{\ga}}\nu) \td u_{\tau, i}\nonumber\\&&-\p_{y_\ga}\p_{y_{\bb}}\p_{y_{\al}}\nu \cdot  \td u_{\tau, \al\bb}\Big)\label{eqD:E2}\\
G_{\al\bb}(p, \td u_{\tau})&=&-\Big(\frac w{\nu}\Big)\cdot\p_{y_{\bb}}\p_{y_{\al}}\nu. \label{eqD:E3}
\eeqn
In view of (\ref{eqD:E1})-(\ref{eqD:E3}), we define the functions depending on $(p, s, y, Q)\in \Si\times \RR\times T_p(\Si)\times GL(2, \RR)$ such that
\beqs
E(p, s, y, Q)&=&\p_s\Big(\frac w{\nu}\Big)(p, s, y)\Big(\partial_s\nu-\p_{p_{\al}}\p_{y_{\al}}
\nu-(\p_s\p_{y_{\al}}\nu) y_{ \al}-(\p_{y_{\bb}}\p_{y_{\al}}\nu)
Q_{\al\bb}\Big) \nonumber\\
&&+\Big(\frac w{\nu}\Big)(p, s, y)\cdot\Big(\p_s^2\nu-\p_s\p_{p_{\al}}\p_{y_{\al}}
\nu-(\p_s^2\p_{y_{\al}}\nu)    y_{ \al}-\p_s\p_{y_{\bb}}\p_{y_{\al}}\nu\cdot Q_{\al\bb}\Big),\\
F_{\ga}(p, s, y, Q)&=&\p_{y_\ga}\Big(\frac w{\nu}\Big)(p, s, y)\cdot\Big(\partial_s\nu-\p_{p_{\al}}\p_{y_{\al}}
\nu-(\p_s\p_{y_{\al}}\nu) y_{ \al}-(\p_{y_{\bb}}\p_{y_{\al}}\nu)
Q_{  \al\bb}\Big) \nonumber\\
&&+\Big(\frac w{\nu}\Big)\Big(\p_{y_\ga}\p_{s}\nu-\p_{y_\ga}\p_{p_{\al}}\p_{y_{\al}}
\nu-\p_s\p_{y_{\ga}}\nu -(\p_s\p_{y_i}\p_{y_{\ga}}\nu) y_{ i}\nonumber\\&&-\p_{y_\ga}\p_{y_{\bb}}\p_{y_{\al}}\nu \cdot  Q_{ \al\bb}\Big),\\
G_{\al\bb}(p, s, y)&=&-\Big(\frac w{\nu}\Big)(p, s, y)\cdot\p_{y_{\bb}}\p_{y_{\al}}\nu.
\eeqs
Let $\hat u_\la=\la \td u_{\tau}$ for $\la\in [0, 1]$. Then we have
\beqn
E(p, u_{\tau})&=&E(p, 0)+\int_0^1\;\p_{\la}(E(p, \hat u_{\la}))\,d\la,\label{eqC:E8}\\
F_{\ga}(p, u_{\tau})&=&F_{\ga}(p, 0)+\int_0^1\;\p_{\la}(F_{\ga}(p, \hat u_{\la}))\,d\la, \label{eqC:E9}\\
G_{\al\bb}(p, u_{\tau})&=&G_{\al\bb}(p, 0)+\int_0^1\;\p_{\la}(G_{\al\bb}(p, \hat u_{\la}))\,d\la.\label{eqC:E10}
\eeqn
Note that
\beqn
\p_{\la}(E(p, \hat u_{\la}))&=&(\p_sE)\cdot \td u_{\tau}+(\p_{y_i}E)\cdot\td u_{\tau, i}+(\p_{Q_{\al\bb}}E)\cdot \td u_{\tau, \al\bb} \label{eq:XY001}\\
\p_{\la}(F_{\ga}(p, \hat u_{\la}))&=&(\p_sF_{\ga})\cdot \td u_{\tau}+(\p_{y_i}F_{\ga})\cdot\td u_{\tau, i}+(\p_{Q_{\al\bb}}F_{\ga})\cdot \td u_{\tau, \al\bb},\\
\p_{\la}(G_{\al\bb}(p, \hat u_{\la}))&=&(\p_sG_{\al\bb})\cdot \td u_{\tau}+(\p_{y_i}G_{\al\bb})\cdot\td u_{\tau, i}, \label{eq:XY002}
\eeqn where the right-hand sides of (\ref{eq:XY001})-(\ref{eq:XY002})  are evaluated at $(p, s, y, Q)=(p,  \hat u_{\la}, \Na  \hat u_{\la}, \Na^2  \hat u_{\la})$. By Lemma \ref{lem:C1}, we have
\beqn
E(p, 0)&=&-|A|^2,\label{eqC:E3}\\
F_{\ga}(p, 0)&=&0,\label{eqC:E4}\\
G_{\al\bb}(p, 0)&=&-\dd_{\al\bb}.\label{eqC:E5}
\eeqn
Combining (\ref{eqC:E8})(\ref{eq:XY001}) with (\ref{eqC:E3}), we have
\beqn
E(p, u_{\tau})&=&-|A|^2+\td u_{\tau}\int_0^1\,(\p_sE)(p, \hat u_{\tau}, \Na \hat u_{\tau}, \Na^2\hat u_{\tau} )\,d\la\nonumber\\
&&+\td u_{\tau, i}\int_0^1\,(\p_{y_i}E)(p, \hat u_{\tau}, \Na \hat u_{\tau}, \Na^2\hat u_{\tau} )\,d\la\nonumber\\
&&+\td u_{\tau, \al\bb}\int_0^1\,(\p_{Q_{\al\bb}}E)(p, \hat u_{\tau}, \Na \hat u_{\tau}, \Na^2\hat u_{\tau} )\,d\la.\label{eqC:E12}
\eeqn
Similar, we have
\beqn
F_{\ga}(p, u_{\tau})&=&\td u_{\tau}\int_0^1\,(\p_sF_{\ga})(p, \hat u_{\tau}, \Na \hat u_{\tau}, \Na^2\hat u_{\tau} )\,d\la\nonumber\\
&&+\td u_{\tau, i}\int_0^1\,(\p_{y_i}F_{\ga})(p, \hat u_{\tau}, \Na \hat u_{\tau}, \Na^2\hat u_{\tau} )\,d\la\nonumber\\
&&+\td u_{\tau, \al\bb}\int_0^1\,(\p_{Q_{\al\bb}}F_{\ga})(p, \hat u_{\tau}, \Na \hat u_{\tau}, \Na^2\hat u_{\tau} )\,d\la\label{eqC:E13}
\eeqn
and
\beqn
G_{\al\bb}(p, u_{\tau})&=&-\dd_{\al\bb}+\td u_{\tau}\int_0^1\,(\p_sG_{\al\bb})(p, \hat u_{\tau}, \Na \hat u_{\tau} )\,d\la\nonumber\\
&&+\td u_{\tau, i}\int_0^1\,(\p_{y_i}G_{\al\bb})(p, \hat u_{\tau}, \Na \hat u_{\tau})\,d\la.\label{eqC:E14}
\eeqn
Combining (\ref{eqC:E12})-(\ref{eqC:E14}), (\ref{eqC:E6}) with (\ref{eqC:E7}), we have
\beqn
H_{u_i^+}(p)-H_{u_i^-}(p)&=&\int_0^1\,
\p_{\tau}(H_{u_{\tau}}(p))\,d\tau\nonumber\\
&=&-|A|^2u-\Delta u+a_1^{\al\bb}u_{\al\bb}+b_1^{\ga}u_{\ga}+c_1u,\label{eqD:D1}
\eeqn where the coefficients $a_1^{\al\bb}, b_1^i$ and $c_1$ are small, and tend to zero as $u_i^+$ and $u_i^-$ tend to zero.\\

\emph{Step 2. } We calculate the difference of $\eta_{u_i^+}$ and $\eta_{u_i^-}$.
Note that
\beq
\eta_{u_i^+}(p)=\eta_{u_i^-}(p)+\int_0^1\,
\p_{\tau}(\eta_{\td u_{\tau}}(p))\,d\tau. \label{eqD:D2}
\eeq
By (\ref{eqC:E11}), we have
\beq
\p_{\tau}(\eta_{\td u_{\tau}}(p))=(\p_s\eta)(p, \td u_{\tau}, \Na \td u_{\tau})\cdot u+(\p_{y_{\al}}\eta)(p, \td u_{\tau}, \Na \td u_{\tau})\cdot u_{\al}, \label{eqD:D3}
\eeq where the function $\eta$ of the right-hand side is defined by (\ref{eqC:C4}).
Let $\hat u_{\la}=\la \td u_{\tau}$ as in Step 1. Then we have
\beqn
(\p_s\eta)(p, \td u_{\tau}, \Na \td u_{\tau})&=&(\p_s\eta)(p,0, 0)+\td u_{\tau}\int_0^1\,(\p_s^2\eta)(p, \hat u_{\la}, \Na \hat u_{\la})\,d\la \nonumber\\
&&+\td u_{\tau, \al}\int_0^1\,(\p_{y_{\al}}\p_s\eta)(p, \hat u_{\la}, \Na \hat u_{\la})\,d\la. \label{eq:D4}
\eeqn
Similarly, we have
\beqn
(\p_{y_{\al}}\eta)(p, \td u_{\tau}, \Na \td u_{\tau})&=& (\p_{y_{\al}}\eta)(p, 0, 0)
+\td u_{\tau}\int_0^1\,(\p_s\p_{y_{\al}}\eta)(p, \hat u_{\la}, \Na \hat u_{\la}),d\la \nonumber\\
&&+\td u_{\tau, \bb}\int_0^1\,(\p_{y_{\al}}\p_{y_{\bb}}\eta)(p, \hat u_{\la}, \Na \hat u_{\la})\,d\la. \label{eq:D5}
\eeqn
Combining (\ref{eqD:D2})-(\ref{eq:D5}) with Part (3) of Lemma \ref{lem:C1}, we have
\beq
\eta_{u_i^+}(p)=\eta_{u_i^-}(p)+u-\langle p, \Na u\rangle+  b_2^iu_i+c_2 u, \label{eq:D6}
\eeq where $c_2$ and $b_2^{\al}$ are small and tend to zero as $u_i^+$ and $u_i^-$ tend to zero. \\

\emph{Step 3.} We calculate the difference of $\phi_{u_i^+}(p)$ and $\phi_{u_i^-}(p)$ where $\phi_u=H_u-\frac 12\langle \x_u, \n_u\rangle.$
  Combining (\ref{eq:D6}) with (\ref{eqD:D1}), we have
\beq
\phi_{u_i^+}(p)-\phi_{u_i^-}(p) = -Lu+a_3^{\al\bb}u_{\al\bb}+b_3^{\ga}u_{\ga}+c_3u,\label{eqC:F2}
\eeq where $a_3^{\al\bb}, b_3^{\ga}$ and $c_3$ are small and tend to zero as $u_i^+$ and $u_i^-$ tend to zero.
Note that
\beq
 (\p_t {\x_{u_i^+}} )^{\perp}=\langle \p_t{\x_{u_i^+}}, \n_{u_i^+}\rangle=\p_t {u_i^+}\langle \n, \n_{u_i^+} \rangle=\p_t {u_i^+}w_{u_i^+}, \label{eq:D7}
\eeq where $w_{u_i^+}$ is defined by (\ref{eqC:C2}), and $\x_{u_i^-}$ satisfies a similar equation as (\ref{eq:D7}).
Moreover, we have
\beqn
\p_t {u_i^+}w_{u_i^+}-\p_t {u_i^-}w_{u_i^-}=\int_0^1\,\p_{\tau}(\p_t\td u_{\tau}w_{\td u_{\tau}})\,d\tau.
\eeqn
As in (\ref{eqC:E11}), we have
\beqn &&
\p_{\tau}(\p_t\td u_{\tau}w_{\td u_{\tau}})\nonumber\\&=& \p_tu \,w_{\td u_{\tau}}+
\p_t\td u_{\tau}\,\p_{\tau}w_{\td u_{\tau}}\nonumber\\
&=&\p_tu \,w_{\td u_{\tau}}+
\p_t\td u_{\tau}\Big((\p_sw)(p, \td u_{\tau}, \Na \td u_{\tau} )\cdot u+(\p_{y_i}w)(p, \td u_{\tau}, \Na \td u_{\tau} )\cdot u_i\Big).\label{eq:D8}
\eeqn
Since $w(p, 0, 0)=1$ by (\ref{eqC:C7}), we have
\beqs w_{\td u_{\tau}}(p)&=&1+\int_0^1\,(\p_{\la}w_{\hat u_{\la}})(p)\,d\la\\
&=&1+\td u_{\tau}\int_0^1\,(\p_sw)(p, \hat u_{\la}, \Na \hat u_{\la})\,d\la+\td u_{\tau, i}\int_0^1\,(\p_{y_i}w)(p, \hat u_{\la}, \Na \hat u_{\la})\,d\la.
\eeqs
Combining the above identities, we have
\beqn
\p_t {u_i^+}w_{u_i^+}-\p_t {u_i^-}w_{u_i^-}&=&\p_t u(1+b_4^i\td u_{\tau, i}+c_4 \td u_{\tau})+b_5^i u_{ i}+c_5  u, \label{eqC:F1}
\eeqn where $c_4, c_5, b_4^i $ and $b_5^{i}$ are small and tend to zero as $u_i^+$ and $u_i^-$ tend to zero. Combining (\ref{eqC:F1}) (\ref{eqC:F2}) with the equation of rescaled mean curvature flow, we have
\beqs
\pd ut&=&\frac 1{1+b_4^i\td u_{\tau, i}+c_4 \td u_{\tau}}\Big(\p_t {u_i^+}w_{u_i^+}-\p_t {u_i^-}w_{u_i^-}-b_5^i u_{ i}-c_5  u\Big)\\
&=&\frac 1{1+b_4^i\td u_{\tau, i}+c_4 \td u_{\tau}}\Big(Lu+a_6^{\al\bb}u_{\al\bb}+b_6^{\ga}u_{\ga}+c_6u\Big)\\
&=&Lu+a_7^{\al\bb}u_{\al\bb}+b_7^{\ga}u_{\ga}+c_7u,
\eeqs where $a_6^{\al\bb}, b_6^{\ga}, c_6, a_7^{\al\bb}, b_7^{\ga}$ and $c_7$ are small and tend to zero as $u_i^+$ and $u_i^-$ tend to zero.  The lemma is proved.
\end{proof}

\end{appendices}

\vskip10pt
Haozhao Li,  Institute of Geometry and Physics, and Key Laboratory of Wu Wen-Tsun Mathematics,   School of Mathematical Sciences, University of Science and Technology of China, No. 96 Jinzhai Road, Hefei, Anhui Province, 230026, China;  hzli@ustc.edu.cn.\\

Bing  Wang,  Institute of Geometry and Physics, and Key Laboratory of Wu Wen-Tsun Mathematics, School of Mathematical Sciences, University of Science and Technology of China, No. 96 Jinzhai Road, Hefei, Anhui Province, 230026, China; topspin@ustc.edu.cn.\\


\begin{thebibliography}{10}

  \bibitem{[Andrews]}  B. Andrews. \emph{ Noncollapsing in mean-convex mean curvature flow},  Geom. Topol., 16(2012), no. 3, 1413-1418.

 \bibitem{[AB]} B. Andrews, P. Bryan,  \emph{Curvature bound for curve shortening flow via distance comparison and a direct proof of Grayson's theorem},  J. Reine Angew. Math. 653 (2011), 179-187.

\bibitem{[Bam]} R. Bamler, \emph{Convergence of Ricci flows with bounded scalar curvature},  Ann. Math. , 188(2018), 753-831.

\bibitem{[Bao]}C. Bao, \emph{A note on the entropy of mean curvature flow},  Sci. China Math.,  58(2015), no. 12,   2611-2620.

\bibitem{[BGV]}N. Berline, E. Getzler, M. Vergne, \emph{Heat Kernels and Dirac Operators},  Grundlehren Text Editions. Springer-Verlag, Berlin, 2004. x+363 pp.

 \bibitem{[Bra]}K. A. Brakke, \emph{The motion of a surface by its mean curvature},  Math. Notes 20, Princeton Univ. Press, Princeton, N.J., 1978.

\bibitem{[Brendle]}S. Brendle, \emph{Embedded self-similar shrinkers of genus 0},  Ann. of Math. (2) 183 (2016), no. 2, 715-728.

\bibitem{[BH]}S. Brendle, G. Huisken. \emph{Mean curvature flow with surgery of mean convex surfaces in $\mathbb{R}^3$},  Invent. Math. 203 (2016), no. 2, 615-654.



 \bibitem{[Calabi]} E. Calabi, \emph{Extremal K\"ahler metrics},  Seminar on Differential Geometry, pp. 259-290, Ann. of Math. Stud.,  102, Princeton Univ. Press, Princeton, N.J., 1982.

 \bibitem{[CC]}X. X. Chen, J. R. Cheng, \emph{On the constant scalar curvature K\"ahler metrics (II)-existence results}, arXiv:1801.00656.


\bibitem{[CW1]}X. X. Chen, B. Wang, \emph{Space of Ricci flows I},  Comm. Pure Appl. Math. 65 (2012), no. 10, 1399-1457.

\bibitem{[CW2]}X. X. Chen, B. Wang, \emph{Space of Ricci flow II}, arXiv:1405.6797.

\bibitem{[CW2A]}X. X. Chen, B. Wang, \emph{Space of Ricci flow II---Part A: moduli of singular Calabi-Yau spaces},  Forum of Mathematics, Sigma(2017), Vol. 5.

\bibitem{[CW2B]} X.X. Chen, B. Wang, \emph{Space of Ricci flow II---Part B: weak compactness of the flows},  J. Differ. Geom.116(2020), 1-123.

\bibitem{[CW3]} X.X. Chen, B. Wang, \emph{Remarks of weak-compactness along K\"ahler Ricci flow}, Proceedings of the Seventh International Congress of Chinese Mathematicians, 203-234.

\bibitem{[ChenYin]} B. L. Chen, L. Yin, \emph{Uniqueness and pseudolocality theorems of the mean curvature flow},  Comm. Anal. Geom. 15 (2007), no. 3, 435-490.


\bibitem{[CS]}H. Choi, R. Schoen, \emph{The space of minimal embeddings of a surface into a three-dimensional manifold of positive Ricci curvature},  Invent. Math. 81 (1985), no. 3, 387-394.

\bibitem{[CMbook2]}T. H. Colding, W. P. Minicozzi II, \emph{A course in minimal surfaces}, Graduate Studies in Mathematics, 121. American Mathematical Society, Providence, RI, 2011. xii+313 pp.

\bibitem{[CM1]}T. H. Colding, W. P. Minicozzi II, \emph{Smooth compactness of self-shrinkers}, Comment. Math. Helv. 87 (2012), no. 2, 463-475.

\bibitem{[CM2]}T. H. Colding, W. P. Minicozzi II,  \emph{Generic mean curvature flow I: generic singularities},  Ann. of Math. (2) 175 (2012), no. 2, 755-833.

\bibitem{[CM4]}T. H. Colding, W. P. Minicozzi II, \emph{Uniqueness of blowups and Lojasiewicz inequalities},  Ann. of Math. (2) 182 (2015), 221-285.

\bibitem{[CM3]}T. H. Colding, W. P. Minicozzi II,  \emph{The singular set of mean curvature flow with generic singularities},  Invent. Math. 204 (2016), no. 2, 443-471.



\bibitem{[CM5]}T. H. Colding, W. P. Minicozzi II, \emph{Dynamics of closed singularities},  arXiv:1808.03219.



\bibitem{[Cooper]} A. Cooper, \emph{A characterization of the singular time of the mean curvature flow},  Proc. Amer. Math. Soc. 139 (2011), no. 8, 2933-2942.

\bibitem{[Eckbook]}K. Ecker,  \emph{Regularity theory for mean curvature flow}, Progress in Nonlinear Differential Equations and their Applications, 57. Birkh\"auser Boston, Inc., Boston, MA, 2004.

\bibitem{[EH1]}K. Ecker, G. Huisken, \emph{Mean curvature evolution of entire graphs},  Ann. of Math. (2) 130 (1989), no. 3, 453-471.

\bibitem{[EH2]}K. Ecker, G. Huisken, \emph{Interior estimates for hypersurfaces moving by mean curvature}, Invent. Math. 105 (1991), no. 3, 547-569.

\bibitem{[Gage1]}M. E. Gage, \emph{An isoperimetric inequality with applications to curve shortening}, Duke Math. J. 50 (1983), no. 4, 1225-1229.

\bibitem{[Gage2]}M. E. Gage, \emph{Curve shortening makes convex curves circular}, Invent. Math. 76 (1984), no. 2, 357-364.

\bibitem{[GH]} M. E. Gage,  R. S. Hamilton,\emph{ The heat equation shrinking convex plane curves},  J. Differ. Geom. 23 (1986), no. 1, 69-96.

\bibitem{[Grayson]} M. A. Grayson,\emph{ The heat equation shrinks embedded plane curves to round points},  J. Differ. Geom. 26 (1987), no. 2, 285-314.

\bibitem{[GZ]}Q. Guang,  J. J. Zhu, \emph{Rigidity and curvature estimates for graphical self-shrinkers},  Calc. Var. (2017) 56:176.



\bibitem{[HK13]}R. Haslhofer, B. Kleiner, \emph{Mean curvature flow of mean convex hypersurfaces}, Comm. Pure Appl. Math.,  70(3):511-546, 2017.

\bibitem{[HK14]}R. Haslhofer, B. Kleiner, \emph{Mean curvature flow with surgery}, Duke Math. J.,  166(9):1591--1626, 2017.

\bibitem{[Hui1]}G. Huisken, \emph{Flow by mean curvature of convex surfaces into spheres},  J. Differ. Geom. 20 (1984), no. 1, 237-266.

\bibitem{[Hui2]}G. Huisken, \emph{Asymptotic behavior for singularities of the mean curvature flow},  J. Differ. Geom. 31 (1990), no. 1, 285-299.

\bibitem{[HS99a]}G. Huisken,  C. Sinestrari, \emph{Mean curvature flow singularities for mean convex surfaces},  Calc. Var. PDE. 8 (1999), no. 1, 1-14.

\bibitem{[HS99b]}G. Huisken,  C. Sinestrari, \emph{Convexity estimates for mean curvature flow and singularities of mean convex surfaces},  Acta Math. 183 (1999), no. 1, 45-70.

\bibitem{[HS09]}G. Huisken,  C. Sinestrari, \emph{Mean curvature flow with surgeries of two-convex hypersurfaces},  Invent. Math. 175 (2009), no. 1, 137-221.

\bibitem{[Il]}T. Ilmanen, \emph{Singularities of mean curvature flow of surfaces},  Preprint.

\bibitem{[I2]}T. Ilmanen, \emph{Lectures on mean curvature flow and related equations}, Preprint.

\bibitem{[KT]}T. Kan, J. Takahashi, \emph{Time-dependent singularities in semilinear parabolic equations: Behavior at the singularities},  J. Differential Equations,  260(2016), no. 10,  7278-7319.

\bibitem{[LS1]}N. Q. Le, N. Sesum,  \emph{The mean curvature at the first singular time of the mean curvature flow},  Ann. I. H. Poincar\'e-AN 27 (2010), no. 6, 1441-1459.



\bibitem{[LS3]}N. Q. Le, N. Sesum, \emph{Blow-up rate of the mean curvature during the mean curvature flow and a gap theorem for self-shrinkers}, Comm. Anal. Geom. 19 (2011), no. 4, 633-659.

\bibitem{[LW]}H. Li, B. Wang, \emph{The extension problem of the mean curvature flow (I)},  Invent. Math. 218 (2019), no. 3, 721-777.

\bibitem{[LWZ]}H. Li, B. Wang, K. Zheng, \emph{ Regularity scales and convergence of the Calabi flow},  J. Geom. Anal. 28 (2018), no. 3, 2050-2101.



 \bibitem{[Li]} P. Li,  \emph{Gemetric analysis}, Cambridge Studies in Advanced Mathematics, 134. Cambridge University Press, Cambridge, 2012.

\bibitem{[LY]}P. Li, S. T. Yau, \emph{On the parabolic kernel of the Schr\"odinger operator},  Acta Math. 156 (1986), no. 3-4, 153-201.

\bibitem{[Lieb]}G. M. Lieberman,  \emph{Second order parabolic differential equations},  World Scientific Publishing Co., Inc., River Edge, NJ, 1996. xii+439 pp.

\bibitem{[LS4]}L. Z. Lin, N.  Sesum, \emph{Blow-up of the mean curvature at the first singular time of the mean curvature flow},  Calc. Var. PDE. 55 (2016), no. 3, Art. 65, 16 pp.

\bibitem{[Kry]}N.  Krylov, M.  Safonov, \emph{A certain property of solutions of parabolic equations with measurable coefficients},  Izv. Akad. Nauk SSSR Ser. Mat., 44:1 (1980), 161-175; Math. USSR-Izv., 16:1 (1981), 151-164.



\bibitem{[Sesum1]} N. Sesum, \emph{Curvature tensor under the Ricci flow},  Amer. J. Math. 127 (2005), no. 6, 1315-1324.

\bibitem{[Sun]}A. Sun, \emph{Local entropy and generic multiplicity one singularities of mean curvature flow of surfaces}, arXiv:1810.08114.

\bibitem{[TZ]} G.Tian, Z.L.  Zhang, \emph{Regularity of K\"ahler-Ricci flows on Fano manifolds}, Acta math. 216(2016), no. 1, 127-176.

\bibitem{[WangBing11]} B. Wang, \emph{On the conditions to extend Ricci flow(II)}, Int. Mat. Res. Not. 2012 (14), 3192-3223.

\bibitem{[WangLu]}L. Wang, \emph{Asymptotic structure of self-shrinkers},  arXiv:1610.04904.

\bibitem{[WangMT]}M. T. Wang, \emph{The mean curvature flow smoothes Lipschitz submanifolds},  Comm. Anal. Geom. 12 (2004), no. 3, 581-599.

\bibitem{[White3]}B. White,  \emph{Curvature estimates and compactness theorems in 3-manifolds for surfaces that are stationary for parametric elliptic functionals},  Invent. Math. 88 (1987), no. 2, 243-256.

\bibitem{[White00]} B. White, \emph{The size of the singular set in mean curvature flow of mean-convex sets},  J. Amer. Math. Soc. 13 (2000), no. 3, 665-695.

\bibitem{[White03]} B. White, \emph{The nature of singularities in mean curvature flow of mean-convex sets},  J. Amer. Math. Soc. 16 (2003), no. 1, 123-138.

\bibitem{[White2]} B. White, \emph{A local regularity theorem for mean curvature flow}, Ann. of Math. (2) 161 (2005), no. 3, 1487-1519.

\end{thebibliography}
\end{document}